\documentclass[a4paper,12pt,leqno]{amsart}
\usepackage[a4paper, top=3cm, bottom=3cm, left=3cm, right=3cm]{geometry}

\usepackage{amsmath,amsfonts,amsthm,amssymb}
\usepackage[alphabetic]{amsrefs}
\usepackage[utf8]{inputenc}
\usepackage{enumerate,bbm}
\usepackage{enumerate}
\usepackage{graphicx}
\usepackage{upref} 

\newtheorem*{mainthm*}{Main Theorem}
\newtheorem{thm}{Theorem}[section]
 
\newtheorem{lem}[thm]{Lemma}
\newtheorem{lemma}[thm]{Lemma}
\newtheorem{prop}[thm]{Proposition}
\newtheorem{cor}[thm]{Corollary}

\newtheorem{partt}{Part}

\theoremstyle{definition}
\newtheorem{rem}[thm]{Remark}
\newtheorem{defin}[thm]{Definition}
\newtheorem{constr}[thm]{Construction}
\newtheorem{exa}[thm]{Example}

\newcommand{\Bir}{\mathrm{Bir}}
\newcommand{\SL}{\textbf{SL}}
\newcommand{\Tame}{\mathrm{Tame}}
\newcommand{\LL}{\mathcal{L}}
\newcommand{\N}{\mathbb{N}}
\newcommand{\Z}{\mathbb{Z}}
\newcommand{\R}{\mathbb{R}}
\newcommand{\C}{\mathbb{C}}
\renewcommand{\H}{\mathbb{H}}

\newcommand{\GL}{\textbf{GL}}
\newcommand{\K}{\mathbbm{k}}
\newcommand{\T}{\mathrm{Tame}(\K^3)}
\newcommand{\id}{\mathrm{id}}
\renewcommand{\int}{\mathrm{int}}
\newcommand{\Ap}{\mathbf E}

\newcommand{\Aut}{\mathrm{Aut}}
\newcommand{\X}{\mathbf{X}}
\newcommand{\St}{\mathrm{Stab}}
\newcommand{\Stab}{\mathrm{Stab}}

\newcommand{\eps}{\varepsilon}

\newcommand{\parent}[1]{\textup{(}#1\textup{)}}
\usepackage[colorlinks,linkcolor=blue,citecolor=blue,urlcolor=blue]{hyperref}
\usepackage[alphabetic]{amsrefs}

\begin{document}

\title{Tits Alternative for the $3$-dimensional tame automorphism group}

\author[S.~Lamy]{St\'ephane Lamy$^{*}$}
\address{Institut de Mathématiques de Toulouse UMR 5219, Universit\'e de Toulouse, UPS
F-31062 Toulouse Cedex 9, France} \email{slamy@math.univ-toulouse.fr}

\author[P.~Przytycki]{Piotr Przytycki$^{\dag*}$}
\address{
Department of Mathematics and Statistics, McGill University, Burnside Hall, 805 Sherbrooke Street West, Montreal, QC, H3A 0B9, Canada}
\email{piotr.przytycki@mcgill.ca}
\thanks{$\dag$ Partially supported by NSERC, AMS, CNRS and National Science Centre, Poland UMO-2018/30/M/ST1/00668.}
\thanks{$*$ This work was partially supported by
the grant 346300 for IMPAN from the Simons Foundation and the matching 2015--2019 Polish MNiSW fund}

\begin{abstract}
We prove the strong Tits alternative for the tame automorphism group of $\K^3$, for $\K$ of characteristic zero.
\end{abstract}

\maketitle

\section{Introduction}

A group~$G$ satisfies the \emph{strong Tits alternative} (resp.\ \emph{weak Tits alternative})  if each subgroup of~$G$ (resp.\ each finitely generated subgroup of $G$) is virtually solvable or contains a non-abelian free group.
For each~$n\geq 1$ and each field $\K$ of characteristic zero (resp.\ of positive characteristic), Tits proved that the linear group~$\GL_n(\K)$ satisfies the strong (resp.\ weak) Tits alternative \cite{T}.
This result was then extended to many classes of groups.
We mention first some groups of polynomial or rational maps, also known as Cremona transformations.
The group~$\Aut(\C^2)$ of polynomial automorphisms of the complex plane satisfies the strong Tits alternative by \cite{L}.
For an arbitrary field $\K$, one can adapt the argument and obtain the (strong, or weak, according to the characteristic of $\K$) Tits alternative for~$\Aut(\K^2)$ \cite{LamyBook}.
For birational maps, the weak and then strong Tits alternative for the 2-dimensional Cremona group $\Bir(\C^2)$ was established in \cite{Cantat} and \cite{Urech}.
Again, the argument can be extended to any field of characteristic zero \cite{LamyBook}, while the case of positive characteristic remains unclear.
In higher dimension, very little is known.
However, the strong Tits alternative for $\Tame(\SL_2(\C))$, a particular subgroup of the Cremona group $\Bir(\C^3)$, was established in \cite{BFL}.

The aim of this article is to establish the Tits alternative for the following subgroup of the 3-dimensional Cremona group.
The \emph{tame automorphism group} $\T$ of the affine space $\K^3$ is the subgroup  of the polynomial automorphism group generated by the affine
automorphisms and by the automorphisms of the form
\[
(x_1, x_2, x_3) \mapsto (x_1 + P(x_2, x_3), x_2, x_3), \ P \in \K[x_2,x_3].
\]

\begin{mainthm*}
Let $\K$ be a field of characteristic zero.
Then $\T$ satisfies the strong Tits alternative.
\end{mainthm*}

A common feature of the proof of the Tits alternative for $\Aut(\C^2)$, $\Bir(\C^2)$ and $\Tame(\SL_2(\C))$ is the use of an action on a CAT(0)
metric space. Namely, $\Aut(\C^2)$ acts on the Bass--Serre tree associated to its structure of amalgamated product, $\Bir(\C^2)$ acts on an infinite
dimensional hyperbolic space~$\H^\infty$ and $\Tame(\SL_2(\C))$ acts on a CAT(0) square complex. None of these actions is proper.

In general, suppose that a group $G$ acts by isometries on a Gromov-hyperbolic geodesic metric space $X$. By \cite{G}*{\S3.1}, the group $G$ (or its
index $2$ subgroup) fixes a point of $X$ or the Gromov boundary of $X$, or contains a non-abelian free group. In particular, if $G$ is
word-hyperbolic, if we take $X$ to be the Cayley graph of~$G$, the stabilisers of points in the Gromov boundary of $X$ are virtually cyclic and
so~$G$ satisfies the strong Tits alternative. More generally, if $G$ acts acylindrically on a Gromov-hyperbolic geodesic metric space $X$, and the
subgroups of $G$ with bounded orbits in $X$ satisfy the weak (resp.\ strong) Tits alternative, then $G$ satisfies the weak (resp.\ strong) Tits
alternative \cite{DGO}. Taking for $X$ the coned-off Deligne complex, this gives the strong Tits alternative for $2$-dimensional Artin groups of
hyperbolic type \cite{MP2}.

For groups $G$ that are not word-hyperbolic, but possess some hyperbolic-like features, various strategies were employed. Strong Tits alternative was
proved for mapping class groups of hyperbolic surfaces \cites{Iv, McC} and outer automorphism groups of free groups \cites{BFH, BFH2}. This was
generalised to outer automorphism groups of right-angled Artin groups \cites{CV,Ho}.

It is conjectured that if $X$ is a CAT(0) metric space, and $G$ acts on $X$ properly and cocompactly, then it satisfies the strong Tits alternative.
This conjecture was verified for $2$-dimensional complexes $X$ in \cite{OP}. The methods can be extended to `recurrent' spaces \cite{OP0}, implying
that the Artin groups of large type satisfy the strong Tits alternative. For CAT(0) cube complexes $X$ the conjecture was established in \cite{SW}.
More generally, if $G$ acts on a finite-dimensional CAT(0) cube complex $X$ with the stabilisers of points in $X\cup \partial_\infty X$ satisfying the
weak (resp.\ strong) Tits alternative, then $G$ also satisfies the weak (resp.\ strong) Tits alternative \cite{CS}. Applying this, one verifies the
strong Tits alternative for Artin groups $G$ of type FC \cite{MP1}, by computing the stabilisers of points in $\partial_\infty X$, where $X$ is the (cubical) Deligne
complex.

The group $\T$ is the amalgamated product of three subgroups along their pairwise intersection \cites{Wright, LamyToulouse}. This gives a natural
action of $\T$ on a $2$-dimensional simplicial complex, which was proved in \cite{LP1} to be contractible and Gromov-hyperbolic in order to find normal subgroups in $\T$. However, this complex cannot be
endowed with an equivariant CAT(0) metric, and so we do not know how to classify the subgroups of $\T$ with bounded orbits in that complex.

In \cite{LP2}, a new $2$-dimensional CAT(0) complex $\X$ with an action of $\T$ was constructed, leading to the description of the finite subgroups of $\T$.
However, this action is not proper and $\X$ is not Gromov-hyperbolic. Therefore, in order to prove the Main Theorem,
we need to classify and
analyse all elements of $\T$ that do not fix a point of $\X$ and are not loxodromic of `rank~1'. In order to do this, we classify locally and
globally nearly flat subcomplexes of $\X$.

We apply an original mixture of metric and combinatorial techniques, where we introduce `relative disc diagrams', which are, roughly speaking, disc diagrams that are required to be combinatorial only away from the boundary. We often study the limit of relative disc diagrams $\phi_n\colon D\to \X$, whose domain $D$ might degenerate in the limit into a `frilling' $D'$. We believe that these concepts will be useful for studying the actions of other groups on non-proper CAT(0) complexes.

\smallskip

\noindent \textbf{Organisation.}
We present now a simplified outline of the proof of the Main Theorem; see Section~\ref{sec:proof} for the full proof.

The group $\T$ acts on the non-proper 2-dimensional $\mathrm{CAT}(0)$ complex $\X$, which we discuss in Section~\ref{sec:X}.
An important property of $\X$ is the existence of a proper quotient $\rho_+\colon \X\to \nabla^+$, with $\nabla^+$ isometric to a Weyl chamber in $\R^2$.

Let $G<\T$ and suppose that $G$ does not stabilise a point in $\X$, or a `principal' or an `antiprincipal' point $\zeta$ in $\partial_\infty \X$. (Stabilisers of such points $\zeta$ are conjugate to subgroups denoted by $C,B'<\T$, which satisfy the strong Tits alternative.) In particular, by \cite{NOP}*{Cor~2.5}, there is an element of $G$ that is not elliptic, i.e.\ has
no fixed point in $\X$. If all the elements of $G$ are elliptic or loxodromic of `rank~1', one can find a non-abelian free group in $G$ appealing to
the work of Bestvina and Fujiwara. However, one needs to carry over the classical results of Ballmann and Ruane to the non-proper setup, which we do
in Section~\ref{sec:rank1}.

It remains to classify the elements of $\T$ that are parabolic or loxodromic not of rank~1. We do so eventually in Sections~\ref{sec:rays}
and~\ref{sec:classification}, where we show that all such elements are conjugate into $C$ or $B'$. To do that, we first classify all cycles in the vertex links of $\X$ of length $<2\pi+\eps_0$, for a uniform $\eps_0$
(Sections~\ref{sec:triple} and~\ref{sec:links}). We then study nearly flat diagrams $\phi\colon D\to\X$, and we describe their `folding locus' in $D$
where the composition $\rho_+\circ \phi\colon D\to \nabla^+$ is not a local isometry (Section~\ref{sec:diagramsinX}). Note that the classification of
$g$ parabolic, which are all conjugate into $C$, is much easier and can be read without preparation.

In Section~\ref{sec:xieta}, we study diagrams for pairs $(\xi,\eta)$ of limit points of distinct isometries $g,g'\in G$ that are parabolic or
loxodromic not of rank~1. Since $\X$ is not proper, we have to study the limits of disc diagrams $\phi_n$ by postcomposing them with $\rho_+$. This
requires introducing a procedure of passing to the limit with non-combinatorial `relative disc diagrams' (see Section~\ref{sec:diagrams}, portions of which could be
ignored at a first reading by pretending that all the diagrams are combinatorial). Finally, in Section~\ref{sec:limitnotfar} we prove that since $G$
does not stabilise a principal or an antiprincipal point in $\partial_\infty \X$, there exist $g,g'\in G$ with $\xi$ and $\eta$ `far' enough so that
some $g'^ng^n$ has rank~1, giving rise to the non-abelian free group.

\smallskip

\noindent \textbf{Acknowledgements.} We thank Shaked Bader, Werner Ballmann, Pierre-Emmanuel Caprace, Sami Douba, Bruno Duchesne, Koji Fujiwara, Camille Horbez, and Marcin Sabok
for helpful explanations.

\section{Constructing rank~1 isometries}
\label{sec:rank1} In this section, we consider an arbitrary complete $\mathrm{CAT}(0)$ space $(X,d_X)$. We allow $X$ to be not proper. We generalise
the methods of Ballmann and Ruane of constructing rank~1 isometries to this setting.

\subsection{Boundary}

Two geodesic rays $r,r'\colon [0,\infty)\to X$ are \emph{asymptotic} if there is $d>0$ with $d_X\big(r(t),r'(t)\big)\leq d$ for all $t$. This defines an
equivalence relation, and we denote by $[r]$ the equivalence class of a geodesic ray $r$. The \emph{visual boundary} $\partial_\infty X$ is the set
of equivalence classes of geodesic rays. Note that for each geodesic ray~$r$ and each $y\in X$ there is a geodesic ray representing $[r]$ starting at $y$
\cite{BH}*{II.8.2}, which we call the \emph{geodesic ray from $y$ to $[r]$}. We extend the topology on $X$ to the \emph{cone topology} on
$X\cup\partial_\infty X$ described in \cite{BH}*{II.8.6}. Namely, given a geodesic ray~$r$, and $t, \epsilon>0$, a \emph{basic neighbourhood of
$[r]$, based at $r(0)$,} is the set of $x\in X\cup\partial_\infty X$ such that the geodesic (possibly a ray) $r'$ from $r(0)$ to $x$ has length $>t$
and satisfies $d_X\big(r'(t),r(t)\big)< \epsilon$. Note that since $X$ is not assumed to be proper, $\partial_\infty X$ might not be compact.

For $\xi, \eta\in \partial_\infty X$, the \emph{angle} $\angle (\xi, \eta)$ is the supremum
of the Alexandrov angle $\angle_x(r, r')$ (see
\cite{BH}*{I.1.12}) over $x\in X$ and $r$ representing $\xi$ and $r'$ representing~$\eta$ starting at~$x$. The angle equips~$\partial_\infty X$ with
a complete $\mathrm{CAT}(1)$ metric \cite{BH}*{II.9.13}. However, the induced topology on~$\partial_\infty X$ is stronger than the cone topology.

\subsection{Isometries}

Let $g$ be an isometry of $X$. The \emph{translation length} of $g$ is $|g|=\inf_{x\in X}d_X\big(x,g(x)\big)$. If the infimum is attained, we say
that $g$ is \emph{elliptic} if $|g|=0$ and \emph{loxodromic} otherwise. If the infimum is not attained, we say that $g$ is \emph{parabolic}. If $g$
is loxodromic, then by \cite{BH}*{II.6.8.(1)} it has (at least one) \emph{axis}, which is a $g$-invariant geodesic line in $X$.

Let $\xi\in \partial_\infty X$. We say that $\xi$ is the \emph{positive} (resp.\ \emph{negative}) \emph{limit point of $g$}, if for some (hence any) point $x\in X$, we have $g^n(x)\to \xi$ for $n\to \infty$ (resp.\ $n\to -\infty$). 
Note that a limit point is fixed by $g$. 
If $g$ is loxodromic, it has the positive limit point $\xi^+$ and the negative limit point $\xi^-$ that are the endpoints of any of the axes for $g$. In
particular $\angle(\xi^+,\xi^-)=\pi$. If $g$ is parabolic with positive translation length, then by \cite{D}*{lem~23.2} $g$~also has the positive and
negative limit points $\xi^+,\xi^-$ with $\angle(\xi^+,\xi^-)=\pi$. A parabolic isometry of translation length zero without a point $\xi\in
\partial_\infty X$ that is its positive and negative limit point is called \emph{vile}. (We do not know if vile isometries exist.)

\subsection{Rank~1}

Let $\xi,\eta\in \partial_\infty  X$. In what follows we generalise the condition $d_T(\xi,\eta)>\pi$ used by Ruane in the setting of proper $X$
(where $d_T$ is the \emph{Tits metric}, i.e.\ the path metric induced by the angle metric). In \cite{R}*{Prop~2.4} she proved that for proper $X$ this
condition is equivalent to $\xi,\eta$ being the endpoints of a geodesic line that does not bound an isometrically embedded Euclidean half-plane.

\begin{defin}
Let $p\in X$ and let $\eps,R>0$. A geodesic quadrilateral $xyy'x'$ in $X$ is a \emph{$(p,\eps,R)$-quadrilateral} if $p\in xy$, the concatenation
$yy'\cdot y'x' \cdot x'x$ is at distance $>R$ from $p$, and the sum of the four Alexandrov angles of $xyy'x'$ is $>2\pi-\eps$.
\end{defin}

\begin{lemma}
\label{lem:square} Let $\square=xyy'x'$ be a $(p,\eps,R)$-quadrilateral in $X$. Then there is an embedded $(p,\eps,R)$-quadrilateral $\overline \square =\bar x\bar
y\bar y'\bar x'$ in $X$ with $\bar x \bar y\subset xy$ and all the other sides of $\overline \square$ contained in some sides of $\square$.
\end{lemma}
\begin{proof} If for some vertex $v$ of $\square$ the consecutive sides containing $v$ intersect along a geodesic $v\bar v$, then we replace $v$ by $\bar v$.
Note that all the Alexandrov angles remain the same except that the angle $0$ at $v$ is replaced by a nonnegative angle at $\bar v$.

If opposite sides of $\square$ intersect at a point $z$ of $X$, then $\square$ decomposes as a concatenation of two geodesic triangles with Alexandrov angle sum
$>2\pi-\eps$. Let $\triangle$ be the triangle containing $p$. By \cite{BH}*{II.1.7(4)}, the Alexandrov angle sum in the other triangle is $\leq \pi$, so
the sum of the Alexandrov angles of $\triangle$ is $>\pi-\eps$.

As in the first paragraph, we can assume that $\triangle$ is embedded. We introduce an extra vertex (of Alexandrov angle $\pi$) on a side of $\triangle$ not
containing $p$ to treat $\triangle$ as a $(p,\eps,R)$-quadrilateral.
\end{proof}

\begin{defin}
\label{def:Rsquare} Let $\xi,\eta\in\partial_\infty  X$. We say that $\xi,\eta$ are \emph{$(\eps,R)$-far} (or, shortly, \emph{far}), if there are
neighbourhoods $U,V$ of $\xi,\eta$ in $X\cup \partial_\infty  X$ and a closed metric ball $P\subset X$ such that for each $u\in U, v\in V,$ with at
least one of $u,v$ in $X$,
\begin{itemize}
\item
the geodesic (possibly a ray) $uv$ in $X$ intersects $P$, and, furthermore,
\item
there is $p\in uv\cap P$ that does not lie in any segment of $uv$ that is a side of a $(p,\eps,R)$-quadrilateral.
\end{itemize}
\end{defin}

\begin{lem}
\label{lem:twoinone} Let $x\in X$, let $\xi,\eta\in \partial_\infty X$, and let $u,v\in X$ distinct from $x$ lie on the geodesic rays from $x$ to
$\xi,\eta$, respectively. Then for any point $p$ on the geodesic $uv$, distinct from $u,v$, we have $\angle_u(x,v)+\angle_v(x,u)\geq \pi-\angle_{p}(\xi,\eta)$.
\end{lem}
\begin{proof}
We have $\angle_u(x,v)\geq \pi-\angle_u(\xi,v)$, which is $\geq \angle_{p}(\xi,u)$ by \cite{BH}*{II.9.3(1)}. See Figure~\ref{fig:angle}. Analogously,
we have $\angle_v(x,u)\geq \angle_{p}(\eta,v)$ and the lemma follows.
\end{proof}

\begin{figure}
\includegraphics{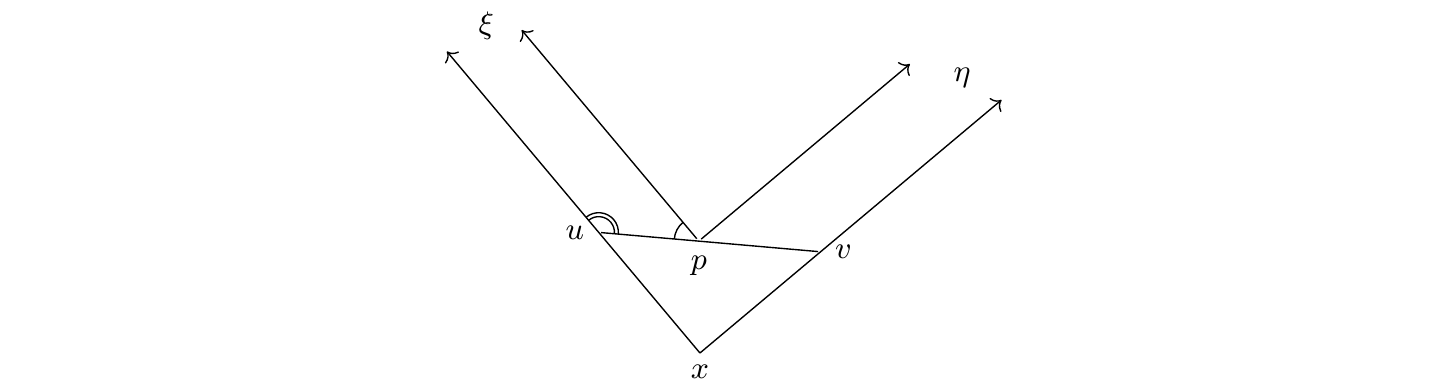}
\caption{}
\label{fig:angle}
\end{figure}

\begin{lem}
\label{lem:farpi}
Suppose that $\xi,\eta$ are far and $P$ is as in Definition~\ref{def:Rsquare}. Then $\angle (\xi,\eta)=\pi$ and for each $\delta>0$
there is $p\in P$ with $\angle_p(\xi,\eta)>\pi- \delta$.
\end{lem}

\begin{proof} Let $x\in X$ and let $r,r'\colon [0,\infty)\to X$ be the geodesic rays from $x$ to $\xi,\eta$. Then for $t$ sufficiently large, we have $r(t)\in
U, r'(t)\in V$. If $\angle (\xi,\eta)<\pi$, then by \cite{BH}*{II.9.8(4)} the geodesic $r(t)r'(t)$ is arbitrarily far from $x$ for $t$ sufficiently
large. Consequently, $r(t)r'(t)$ is disjoint from $P$, which is a contradiction.

Thus $\angle (\xi,\eta)=\pi$, and
so we can assume $\angle_x(\xi,\eta)>\pi- \delta$. Then $\angle_{r(t)}\big(x,r'(t)\big)+\angle_{r'(t)}\big(r(t),x\big)<\delta$. Consequently, for $p\in P\cap r(t)r'(t)$, by Lemma~\ref{lem:twoinone}
we have $\angle_p(\xi,\eta)>\pi- \delta$.
\end{proof}

Following \cite{BF}*{Def~5.1}, an isometry $g$ of $X$ has \emph{rank~1} if for some (hence any) $x\in X$, there is $M>0$ such that, for each $n>0$,
\begin{itemize}
\item
the geodesic $\beta_n=g^{n}(x)g^{-n}(x)$ is at Hausdorff distance $\leq M$ from the set $\{g^{n}(x), g^{n-1}(x), \ldots, g^{-n}(x)\}$, and
\item
$\beta_n$
is \emph{$M$-contracting}, in the sense that the projection to $\beta_n$ of any closed metric ball $B$ disjoint from $\beta_n$ has length $\leq M$.
\end{itemize}

It is not hard to prove (but we will not need it) that if $X$ is proper, then $\xi,\eta$ are $(\eps,R)$-far (for some $\eps,R$) if and only if
$\xi,\eta$ are the endpoints of a geodesic line that does not bound an isometrically embedded Euclidean half-plane. Thus the following theorem
generalises \cite{BF}*{Thm~5.4}.

\begin{thm}
\label{thm:contracting} 
Let $g$ be an isometry of $X$ with
far limit
points. Then $g$ has rank~1.
\end{thm}

Note that if $X$ has finite telescopic dimension, then such $g$ is loxodromic, since by \cite{CL}*{Thm 1.1} (see also \cite{D}*{thm~6.1}), for
parabolic $g$ all of its fixed points in~$\partial_\infty X$ are at Tits distance $\leq \pi$.

Before giving the proof, we record the following immediate consequence of Theorem~\ref{thm:contracting} and  \cite{BF}*{Prop~5.9}.
\begin{cor}
\label{cor:BF}
Let $g$ (resp.\ $g'$) be isometries of $X$ with far limit points $\xi,\eta$ (resp.\ $\xi',\eta'$). If $\{\xi,\eta\},\{\xi',\eta'\}$ are disjoint, then some powers of $g,g'$ are a basis of a free subgroup in the isometry group of $X$.
\end{cor}

To prove Theorem ~\ref{thm:contracting}, we first establish the following.

\begin{lemma}
\label{lem:new} Let $\eps,R>0$ and let $\beta\subset X$ be a geodesic with a family of points $\mathcal P$ dividing it into segments of length $\leq
D$. Then $\beta$ is $2R+\big\lceil\frac{2\pi}{\eps}\big\rceil(D+2R)$-contracting or there is $p\in \mathcal P$ and a $(p,\eps,R)$-quadrilateral with a side
on $\beta$.
\end{lemma}

\begin{proof}
Suppose that $\beta$ is not $2R+\big\lceil\frac{2\pi}{\eps}\big\rceil(D+2R)$-contracting. Then there is a closed ball~$B\subset X$ of radius $R'>R$ and centre
$z$ disjoint from $\beta$, and points $x'',y''\in B$  with projections to $\beta$ at distance
$>2R+\big\lceil\frac{2\pi}{\eps}\big\rceil(D+2R)$. Let $x'$ be the point on the geodesic $x''z$ at distance $\min \{R,d_X(x'',z)\}$ from $x''$.
Define $y'$ analogously. Then the geodesic $x'y'$ is contained in the closed ball of radius $R'-R$ and centre $z$ and hence it is at distance $>R$
from $\beta$. Let $x,y$ be the projections of $x',y'$ to $\beta$. Note that $x$ and $y$ are at distance $\leq R$ from the projections of $x'',y''$.
Thus the geodesic~$xy$ has length $>\big\lceil\frac{2\pi}{\eps}\big\rceil(D+2R)$. Consequently, $xy$ can be divided into
$N=\big\lceil\frac{2\pi}{\eps}\big\rceil$ segments $\gamma_1=xx_1,\gamma_2=x_1x_2,\ldots, \gamma_N=x_{N-1}y$ of length $>D+2R$.

\begin{figure}
\includegraphics[scale=1]{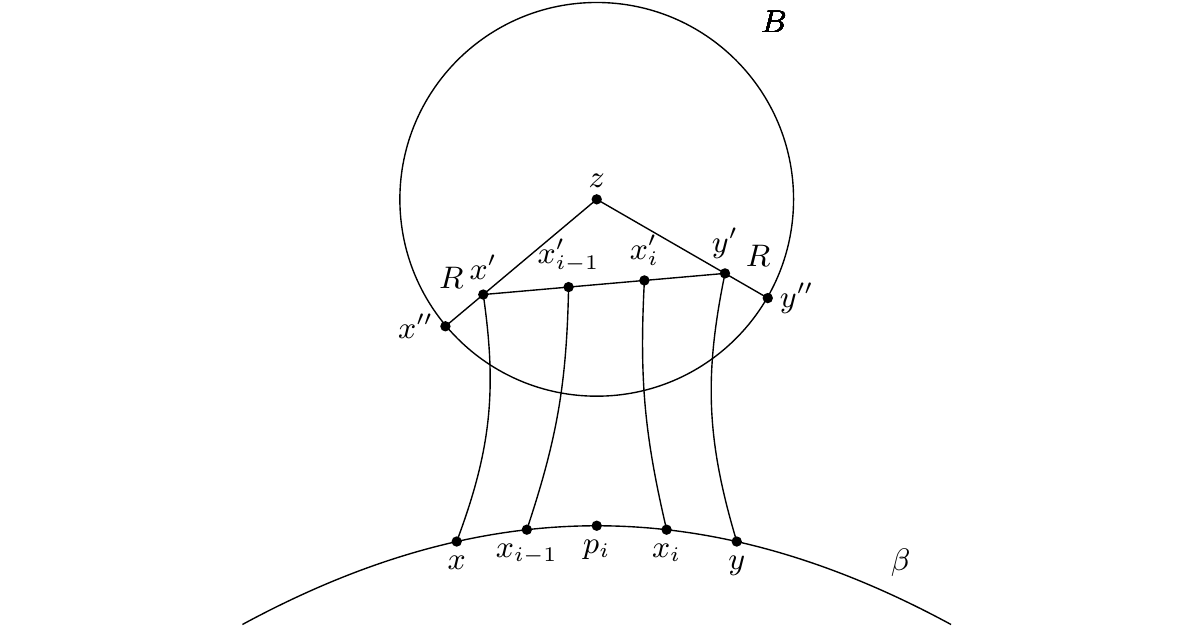}
\caption{}
\label{fig:contracting}
\end{figure}

Each $\gamma_i$ contains some $p_i\in \mathcal P$ at distance $>R$ from the endpoints of $\gamma_i$. Choose $x_1',\ldots, x_{N-1}'$ on $x'y'$ in that
order so that the projection of each $x'_i$ to $\beta$ is~$x_i$. See Figure~\ref{fig:contracting}. Denoting $x_0'=x', x_N'=y', x_0=x, x_N=y$, for $i=1,\ldots, N$
the geodesics $x'_{i-1}x_{i-1}, x'_ix_i$ are at distance $>R$ from~$p_i$. Thus each $x_{i-1}x_ix_i'x_{i-1}'$ is either a
$(p_i,\eps,R)$-quadrilateral, or its sum of the Alexandrov angles is $\leq 2\pi-\eps$. Since
$\angle_{x_i}(x'_{i},x_{i-1})+\angle_{x_i}(x'_{i},x_{i+1})\geq \pi$ and $\angle_{x'_i}(x_{i},x'_{i-1})+\angle_{x'_i}(x_{i},x'_{i+1})\geq \pi$, the
latter possibly would imply that the sum of the Alexandrov angles of $xyy'x'$ is $\leq 2\pi-N\eps\leq 0$, which is a contradiction.
\end{proof}

\begin{proof}[Proof of Theorem \ref{thm:contracting}]
Suppose that the limit points of $g$ are $(\eps,R)$-far. Let $P, U, V$ be provided by Definition~\ref{def:Rsquare}. Let $p\in P$. By possibly passing
to a power of $g$, we can assume $g^n(p)\in U, g^{-n}(p)\in V$ for all $n\geq 1$. Consequently, given $n\geq 1$, for $-(n-1)\leq k\leq n-1$, we have
$g^{n-k}(p)\in U,g^{-n-k}(p)\in V$, and so there is a point $p_k\in P$ on the geodesic $g^{n-k}(p)g^{-n-k}(p)$ that does not lie in a segment of
$g^{n-k}(p)g^{-n-k}(p)$ that is a side of a $(p_k,\eps,R)$-quadrilateral. Denote $D=\mathrm{diam}(P\cup g(P))$. The points
$g^{n-1}(p_{n-1}),g^{n-2}(p_{n-2}),\ldots, g^{-(n-1)}(p_{-(n-1)})$ divide the geodesic $\beta_n=g^n(p)g^{-n}(p)$ into segments of length $\leq D$. In
particular, $\beta_n$ is at Hausdorff distance $\leq D$ from $\{g^{n}(p), g^{n-1}(p),\ldots, g^{-n}(p)\}$.

Let $\beta=\beta_n$ and $\mathcal P=\{g^k(p_k)\}$. By Lemma~\ref{lem:new}, $\beta_n$ is $2R+\big\lceil\frac{2\pi}{\eps}\big\rceil(D+2R)$-contracting
or there is $-(n-1)\leq k\leq n-1$ and a $(g^k(p_k),\eps,R)$-quadrilateral with a side on $\beta_n$. However, the latter possibility would imply that there
is a $(p_k,\eps,R)$-quadrilateral with a side on $g^{-k}(\beta_n)=g^{n-k}(p)g^{-n-k}(p)$, which would be a contradiction.

Thus each $\beta_n$ is $2R+\big\lceil\frac{2\pi}{\eps}\big\rceil(D+2R)$-contracting, as desired.
\end{proof}

\subsection{Construction}

The main result of this section is the following construction of rank~1 isometries, which is a generalisation of \cite{R}*{Main~Thm}.

\begin{prop}
\label{prop:Ruane} Let $f,g$ be isometries of $X$ that have positive limit points $\xi_f^+,\xi_g^+$ that are far and negative limit points
$\xi_f^-,\xi_g^-$ that are far. Then for any neighbourhoods $U_0,V_0$ of $\xi_f^+,\xi_g^+$, for sufficiently large~$n$, the isometry $f^ng^{-n}$ is
vile or has positive translation length and far limit points in $U_0,V_0$.
\end{prop}

To prove Proposition~\ref{prop:Ruane}, we first need the following variant of \cite{Ball}*{Lem~III.3.2}.

\begin{lem}
\label{lem:convergence} Suppose that $\xi,\eta\in \partial_\infty  X$ are far. Let $h_n$ be a sequence of isometries of $X$ such that for some (hence
every) $p\in X$ we have $h_n(p)\to \xi, h^{-1}_n(p)\to \eta$. Then for any neighbourhoods $U_0,V_0$ of $\xi,\eta$, for sufficiently large $n$, the
isometry $h_n$ is vile or has positive translation length and far limit points in $U_0,V_0$.
\end{lem}
\begin{proof}
Suppose that $\xi,\eta$ are $(\eps,R)$-far, and let $P, U, V$ be provided by Definition~\ref{def:Rsquare}. We can assume $U\subset U_0, V\subset V_0$.
By Lemma~\ref{lem:farpi}, we can assume
that for some $p\in P$ and any $u\in U,v\in V$ we have $\angle_p(u,v)>\pi-\frac{\eps}{2}$. Furthermore, we can assume that $U,V$ are basic
neighbourhoods based at $p$ and that for all $n>0$ we have $h_n(B_R(P))\subset U, h^{-1}_n(B_R(P))\subset V$, where $B_R$ denotes the closed
$R$-neighbourhood.

We first justify that for sufficiently large $n$ the isometry $h_n$ has no fixed point $x_n\in X\cup \partial_\infty  X$ outside $U\cup V$. Indeed,
otherwise let $p'\in P$ be a point on the geodesic $h_n(p)h_n^{-1}(p)$ as in Definition~\ref{def:Rsquare}. Since $V$ is based at $p$, and
$h^{-1}_n(B_R(P))\subset V$ whereas $h_n^{-1}(x_n)=x_n\notin V$, we have that the geodesic $ph_n^{-1}(x_n)$ (which is possibly a ray) is disjoint
from $h^{-1}_n(B_R(p'))$. Consequently, the geodesic $h_n(p)x_n$ is disjoint from $B_R(p')$. Analogously, the geodesic $h^{-1}_n(p)x_n$ is disjoint
from $B_R(p')$.

\begin{figure}
\includegraphics[scale=1]{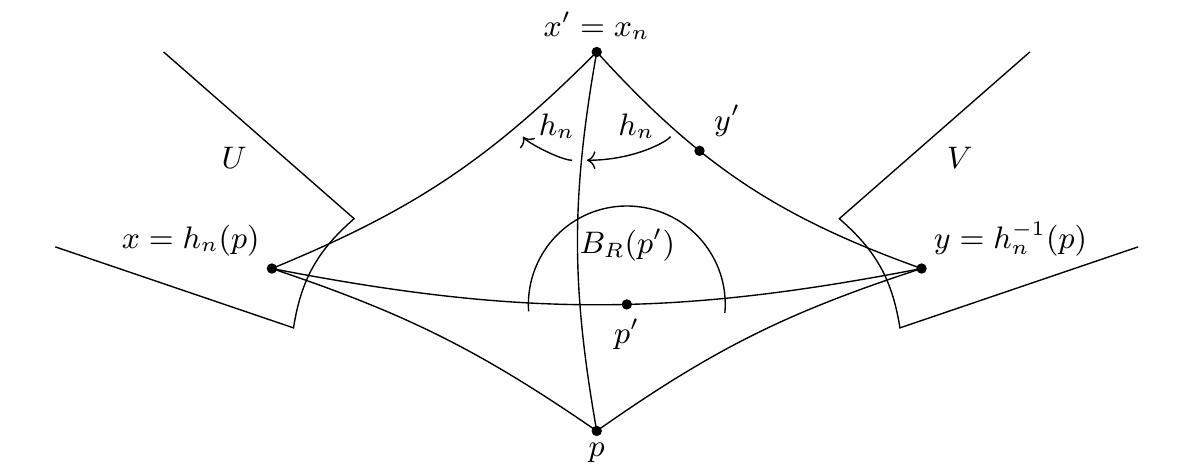}
\caption{}
\label{fig:triangles}
\end{figure}

If $x_n\in X$, then let $x=h_n(p),y=h^{-1}_n(p),x'=x_n$, and let $y'$ be any point on the geodesic $x'y$ distinct from $x',y$. See
Figure~\ref{fig:triangles}. We now compute the Alexandrov angles of the geodesic quadrilateral $xyy'x'$. The triangles $x'xp$ and $x'py$ are
isometric and so $\angle_y(p,x')+\angle_x(x',p)=\angle_p(x,x')+\angle_p(x',y)>\pi-\frac{\eps}{2}$. Consequently, $\angle_y(x,x')+\angle_x(x',y)>
\pi-\frac{\eps}{2} - \angle_y(p,x)-\angle_x(y,p)\geq \pi-\frac{\eps}{2} - \frac{\eps}{2}$. Thus the sum of the Alexandrov angles of $xyy'x'$ at $x$
and $y$ exceeds $\pi-\eps$, whereas at $y'$ it equals~$\pi$. Consequently, $xyy'x'$ is a $(p',\eps,R)$-quadrilateral, which is a contradiction.

If $x_n\in
\partial_\infty  X$, then we take $x'$ sufficiently far on the geodesic ray $h_n(p)x_n$, and we set $y'=h_n^{-2}(x')$. Then, since horoballs are limits
of balls, the Alexandrov angles of $xyy'x'$ at $x'$ and $y'$ are $\geq \frac{\pi}{2}$. Thus an analogous calculation to the one in the case of
$x_n\in X$ leads to a contradiction.

Suppose now that $h_n$ has a fixed point $x_n\in X\cup\partial_\infty  X$ in $V$. Then observe that the geodesic $h_n(p)x_n$ passes through $P$,
which contains $p$. If $x_n\in X$, then for large $n$ this contradicts $d_X\big(h_n(p),x_n\big)=d_X(p,x_n)$. If $x_n\in \partial_\infty  X$ and $x_n$
is the positive fixed point of $h_n$, then for large $n$ this observation contradicts the fact that the Busemann function \cite{BH}*{II.8.17}
satisfies $b_{x_n}(h_n(p))\leq b_{x_n}(p)$. The only remaining possibility is that for large $n$ the isometry $h_n$ is vile or has positive
translation length, with a positive limit point in $U$ and a negative limit point in $V$. Choose neighbourhoods $U_n,V_n\subset X\cup \partial_\infty
X$ of these limits points inside $U,V$. Then $P,U_n,V_n$ satisfy Definition~\ref{def:Rsquare} for the limit points of $h_n$.
\end{proof}

The following is essentially \cite{R}*{Lem~6.3}.

\begin{lem}
\label{lem:Ruane} Let $f$ be an isometry of $X$ with limit points $\xi^+,\xi^-\in \partial_\infty  X$. Let $\eta\in \partial_\infty  X$ with
$\xi^-,\eta$ far. Then there is a neighbourhood $V\subset \partial_\infty  X\cup X$ of $\eta$ such that for every neighbourhood $U^+\subset
\partial_\infty X\cup X$ of $\xi^+$, for sufficiently large $n$ we have $f^n(V)\subset U^+$.
\end{lem}

\begin{proof} Let $P, U^-,V$ be provided by Definition~\ref{def:Rsquare} for $\xi^-,\eta$. Choose a basepoint $x\in X$ such that $f^{-n}(x)$
lie in $U^-$ for all $n\geq 0$. Then for each $v\in V$ and all $n\geq 0$ there is a point of the ball $P$ on the geodesic $f^{-n}(x)v$ (which is
possibly a ray). Consequently, each geodesic $xf^n(v)$ contains a point at a uniform distance from~$f^n(P)$. See Figure~\ref{fig:Ruane}.
This proves that for every neighbourhood $U^+\subset
\partial_\infty  X\cup X$ of $\xi^+$ we have $f^n(V)\subset U^+$ for $n$ sufficiently large.
\end{proof}

\begin{figure}
\includegraphics{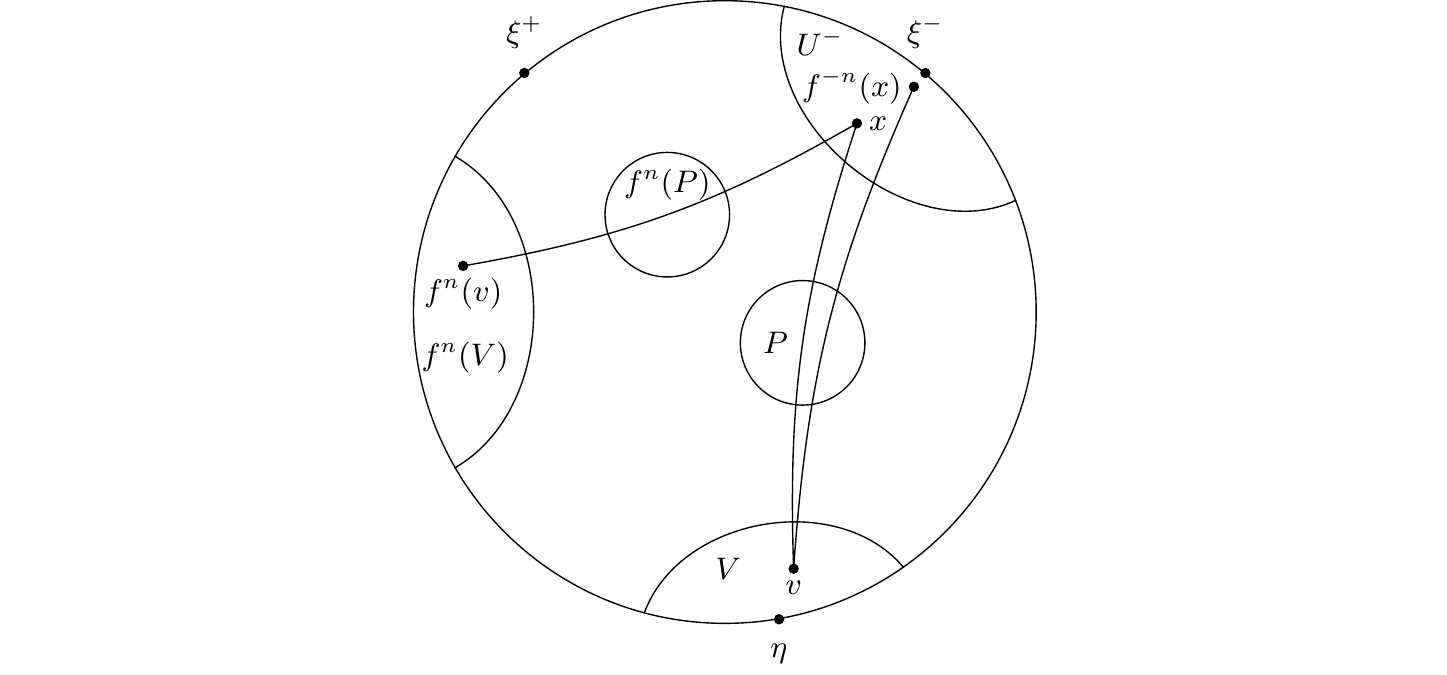}
\caption{}
\label{fig:Ruane}
\end{figure}

\begin{proof}[{Proof of Proposition~\ref{prop:Ruane}}]
Let $U^+$ be a neighbourhood of $\xi_f^+$. Let $V$ be a neighbourhood of $\xi_g^-$ from Lemma~\ref{lem:Ruane}. Choose $x\in X$. Then for $n$
sufficiently large we have $g^{-n}(x)\in V$ and hence by Lemma~\ref{lem:Ruane} for $n$ sufficiently large we have $f^ng^{-n}(x)\in U^+$. Consequently
for $h_n=f^ng^{-n}$ we have $h_n(x)\to \xi_f^+$. Analogously $h^{-1}_n(x)\to \xi_g^+$. By Lemma~\ref{lem:convergence}, for sufficiently large $n$ the
isometry $h_n$ is vile or has positive translation length and far limit points in arbitrary neighbourhoods of $\xi_f^+,\xi_g^+$.
\end{proof}

\section{Building-like space \texorpdfstring{$\X$}{X}}
\label{sec:X}

We recall the construction from \cite{LP2} of a complete $\mathrm{CAT}(0)$ space $\X$ with a natural action of $\T$.  

\subsection{Metric}
\label{sub:X}
Let $\R_{>0}$ be the positive reals and let $\Pi =\R_{>0}^3$. The \emph{weight space} $\nabla$ is the projectivisation of $\Pi$. Let $\Pi^+\subset
\Pi$ consist of $(\alpha_1, \alpha_2, \alpha_3)$ satisfying $\alpha_1\geq \alpha_2\geq \alpha_3$, and let $\nabla^+\subset \nabla$ be the
projectivisation of $\Pi^+$. We equip $\nabla$ and $\nabla^+$ with the topology induced from the projective plane. In particular, $\nabla^+$ is a
topological surface with boundary $\partial \nabla^+$ where $\alpha_1=\alpha_2$ or $\alpha_2=\alpha_3$. Then $\mathrm{int}\, \nabla^+=\nabla^+\setminus \partial \nabla^+$ consists of $[\alpha]$
represented by $\alpha$ with $\alpha_1> \alpha_2>\alpha_3$.

For $\alpha\in \Pi$, we define the following map $\nu_{\id,\alpha}\colon \K[x_1,x_2,x_3]\to \R$. Let $I$ stand for $(m_1,m_2,m_3)\in \N^3$. Namely,
for $P = \sum_{I} c_{I} x_1^{m_1} x_2^{m_2} x_3^{m_3} \in \K[x_1,x_2,x_3]$, we set
$$ \nu_{\id,\alpha}\left(P\right)=\min_{c_I\neq 0}\left( - \sum_{k=1}^3 \alpha_k m_k\right).$$
For $f\in \T$, we define the \emph{monomial valuation} $\nu_{f,\alpha}\colon \K[x_1,x_2,x_3]\to \R$ by $\nu_{f,\alpha}(P)=\nu_{\id,\alpha}(P\circ f)$.
The group $\T$ acts on monomial valuations by $g \cdot \nu (P)=\nu (P\circ g)$. The projective class of $\nu_{f,\alpha}$ depends only on $[\alpha]$
and is denoted by~$\nu_{f,[\alpha]}$. Instead of $[(\alpha_1,\alpha_2,\alpha_3)]$ we shortly write $[\alpha_1,\alpha_2,\alpha_3]$.

We equip the space $\X$ of projective classes of all monomial valuations with the following metric. First, consider the map $\nabla\to \R^3$ sending
each $[\alpha]\in \nabla$ represented by $\alpha$ with $\Pi_{k=1}^3 \alpha_k=1$ to $(\log \alpha_1,\log \alpha_2,\log \alpha_3)$, with image in the
Euclidean plane $\R^2\subset \R^3$ of equation $\log \alpha_1+\log \alpha_2+\log \alpha_3=0$. We equip $\nabla$ (and consequently~$\nabla^+$) with
the Euclidean metric $|\cdot,\cdot|$ pulled back from $\R^2$. 

We view~$\X$ as the quotient (by an equivalence relation $\sim$) of the disjoint union
of \emph{chambers} $\Ap^+_f=\{\nu_{f,[\alpha]} \colon [\alpha]\in \nabla^+\}$. We also consider \emph{apartments} $\Ap_f=\{\nu_{f,[\alpha]} \colon
[\alpha]\in \nabla\}$.
The following proves in particular that the projective classes of monomial valuations $\nu_{\id,[\alpha]}$ and $\nu_{\id,[\alpha']}$ are distinct for $[\alpha]\neq[\alpha']$.

\begin{cor}[\cite{LP2}*{cor 2.5(i)}]
\label{cor:rho} The map $\bigsqcup \Ap^+_f\to \nabla^+$ sending each $\nu_{f,[\alpha]}$ to $[\alpha]$ is well-defined and descends to a map
$\rho_+\colon \X\to \nabla^+$.
\end{cor}

We equip each chamber with the Euclidean metric $|\cdot,\cdot|$ coming from $\nabla^+$. A \emph{chain} is a sequence $(x'_0,x_1\sim x_1',x_2\sim
x_2',\ldots,x_{k-1}\sim x'_{k-1},x_k)$ in $\bigsqcup \Ap^+_f$ such that there are $f_0,\ldots, f_{k-1}\in\T$ with $x_i',x_{i+1}\in \Ap_{f_i}^+$. For
$\overline x,\overline y\in \X$ we define the \emph{quotient pseudo-metric} $d_\X(\overline x,\overline y)$ to be the infimum of the \emph{lengths}
$\sum_{i=0}^{k-1}|x'_i,x_{i+1}|$ of chains with $x_0'\in\overline x$ and $x_k\in\overline y$.

\begin{thm}[\cite{LP2}*{Thm~A}]
\label{thm:cat(0)} $(\X,d_\X)$ is complete and $\mathrm{CAT}(0)$.
\end{thm}

Here are some other basic properties of $(\X,d_\X)$ that we will be using. By $B_\X(\overline{x},\epsilon)$ we denote the radius $\epsilon$ open ball
in $\X$ around $\overline x\in \X$.

\begin{lem}\phantomsection
\label{lem:propertiesX}
\begin{enumerate}[$(i)$]
\item \label{lem:propertiesX:i}
For each $\overline{x}\in X$ there is $\epsilon=\epsilon(\overline{x})>0$ such that $B_\X(\overline{x},\epsilon)$ coincides as a set with
    the union of radius $\epsilon$ open balls around $x$ in $\bigsqcup \Ap^+_f$ for $x\in \overline x$.
\item Each $\Ap^+_f$ is isometrically embedded in $\X$.
\item For each pair $\overline x,\overline y\in \X$, the infimum in the definition of the distance $d_\X(\overline x,\overline y)$ is realised.
\end{enumerate}
\end{lem}
\begin{proof} Item~(i) is \cite{LP2}*{lem~5.5(a)}, and item~(ii) is \cite{LP2}*{lem~5.6(i)}. For item~(iii),
let $\gamma$ be a geodesic from $\overline{x}$ to $\overline{y}$ guaranteed by Theorem \ref{thm:cat(0)}. By the compactness of $\gamma$, we can find
points $\overline{x}_0=\overline{x}, \overline{x}_1,\ldots, \overline{x}_k=\overline{y}$ such that the consecutive balls
$B_\X\Big(\overline{x}_i,\frac{\epsilon(\overline{x}_i)}{2}\Big)$ intersect. Consequently, for each $0\leq i<k$, we have $\overline{x}_i\in
B_\X\big(\overline{x}_{i+1},\epsilon(\overline{x}_{i+1})\big)$ or $\overline{x}_{i+1}\in B_\X\big(\overline{x}_{i},\epsilon(\overline{x}_{i})\big)$.
By item~(i), this shows that $\overline{x}_i,\overline{x}_{i+1}$ have representatives in a common $\Ap^+_{f_i}$. By item~(ii), the length of $\gamma$
coincides with the length of the chain determined by the $\overline{x}_i$.
\end{proof}

\subsection{Point stabilisers}
\label{sec:point_stabilizers}
Let $A< \T$ be the subgroup consisting of the affine transformations of $\K^3$. We will also consider the following subgroups of $\T$:

\begin{multline*}
B=\{(ax_1 +P(x_2,x_3), bx_2+cx_3+d, b'x_2+c'x_3+d'); \\
P\in \K[x_2,x_3], a\neq 0, bc'-b'c\neq 0\},
\end{multline*}
\begin{multline*}
H= \{(ax_1 + P_1(x_2,x_3), bx_2 + P_2(x_3), cx_3 + d);\\
P_1\in \K[x_2,x_3], P_2\in \K[x_3]\ a,b,c\neq 0\},
\end{multline*}
\begin{multline*}
K=\{(ax_1 +bx_2+P_1(x_3), a'x_1+b'x_2+P_2(x_3), cx_3+d);\\P_1,P_2\in \K[x_3], ab'-a'b\neq 0,c\neq 0\}.
\end{multline*}

Let $C=\langle K,H\rangle,$ where $C=K\ast_{K\cap H}H$ by \cite{LP1}*{Prop~3.5}. Finally, let $B'=\langle B,H\rangle$, where $B'=B\ast_{B\cap H}H$,
which corresponds to the splitting of $\mathrm{Aut}(\K^2)$ in \cite{vdk}.

Let $\alpha \in \Pi^+$. We consider the following subgroups $M_\alpha < H, L_\alpha< \GL_3(\K)< \T$.

\begin{enumerate}
\item
If $\alpha_1=\alpha_2=\alpha_3$, then $L_\alpha = \GL_3(\K)$ and $M_\alpha$ is the group of translations, so that $M_\alpha \rtimes L_\alpha=A$.

\item
If $\alpha_1 > \alpha_2> \alpha_3$, then $L_\alpha$ is the group of the diagonal matrices and
$M_\alpha$ is the group of the automorphisms
\[
(x_1, x_2, x_3) \to (x_1 + P_1(x_2,x_3), x_2+P_2(x_3), x_3+d)
\]
where $P_i$ satisfy $\alpha_i \geq -\nu_{\id,\alpha} (P_i)$. In particular, $M_\alpha \rtimes L_\alpha< H$.

\item
If $\alpha_1> \alpha_2=\alpha_3$, then
\begin{align*}
L_\alpha &= \{(a x_1, bx_2 + cx_3, b'x_2 + c'x_3) \}, \\
M_\alpha &= \{(x_1 + P(x_2,x_3), x_2 + d, x_3 + d'); \, \alpha_1 \geq \alpha_2\deg P \}.
\end{align*}
In particular, $M_\alpha \rtimes L_\alpha< B$.

\item
If $\alpha_1=\alpha_2>\alpha_3$, then
\begin{align*}
L_\alpha &= \{(a x_1 + bx_2, a'x_1 + b'x_2, cx_3) \}, \\
M_\alpha &= \{(x_1 + P_1(x_3), x_2 + P_2(x_3), x_3 + d); \, \alpha_1 \geq \alpha_3\deg
P_1,\alpha_3\deg P_2 \}.
\end{align*}
In particular, $M_\alpha \rtimes L_\alpha< K$.
\end{enumerate}

Note that $L_\alpha, M_\alpha$ depend only on the class $[\alpha]$.

\begin{prop}[\cite{LP2}*{prop~3.2}]
\label{prop:stab}
Let $\alpha\in \Pi^+$. We have
\[
\St\big(\nu_{\id,[\alpha]}\big) = M_\alpha \rtimes L_\alpha.
\]
\end{prop}

In particular, each element of the groups $H,B,$ and $K$, has a fixed point in $\Ap^+_\id$.

The following result was proved in \cite{L} for $\K=\C$.
In fact, the proof extends to any field of characteristic zero, see \cite{LamyBook}.

\begin{thm}
\label{thm:k2} Let $\K$ be a field of characteristic zero. Then the polynomial automorphism group $\Aut(\K^2)$ satisfies the strong Tits alternative.
\end{thm}

Below, the \emph{step} of a solvable group is its derived length, that is, the length of its derived series. In particular, non-trivial abelian
groups have step~1.

\begin{rem}\phantomsection
\label{rem:BKTits}
\begin{enumerate}[(i)]
\item The groups $M_\alpha$ are solvable and the groups $L_\alpha$ are linear. Thus, since the strong Tits alternative is preserved under
    semidirect products, by Proposition~\ref{prop:stab} all $\St\big(\nu_{\id,[\alpha]}\big)$ satisfy the strong Tits alternative.
\item Since $M_\alpha$ are solvable of step $\leq 3$, and all virtually solvable subgroups of all $L_\alpha$ contain a triangularisable subgroup
    of uniformly bounded index (see \cite{W} and \cite{KM}*{Thm~21.1.5}), which is solvable of step $\leq 3$, we have that all virtually solvable
    subgroups of all $\St\big(\nu_{\id,[\alpha]}\big)$ contain a solvable subgroup of uniformly bounded index and step $\leq 6$.
\item The group $B'$ has a decomposition into a semidirect product of $\Aut(\K^2)$ (with variables $x_2,x_3$) and the solvable normal subgroup
    consisting of automorphisms of the form $(ax_1 +P(x_2,x_3), x_2,x_3)$. Thus by Theorem \ref{thm:k2}, the group $B'$ also satisfies the strong Tits alternative.
    Similarly, the group $C$ is a semidirect product of the solvable subgroup consisting of $(x_1, x_2,cx_3+d)$ and a normal subgroup that can be
    embedded in the automorphism group $\Aut\big(\K(x_3)^2\big)$ with variables $x_1,x_2$. Consequently, $C$ also satisfies the strong Tits
    alternative.
\end{enumerate}
\end{rem}

Proposition~\ref{prop:stab} can be rephrased as follows.

\begin{cor}[\cite{LP2}*{prop~4.4}] 
\label{cor:fixed_set} Let $f\in \T$ and let $\alpha\in \Pi^+$. Then $\nu_{\id,[\alpha]}$ is fixed by $f$ if and only if $\alpha$ satisfies all the
inequalities of the form
\[
\alpha_{i} \geq \sum_{k} m_{k} \alpha_{k}, 
\]
where $i = 1,2,3$ and $\prod_{k}x_{k}^{m_{k}}$ are the monomials appearing in $f_i$. \end{cor}

\begin{figure}
\includegraphics[scale=1]{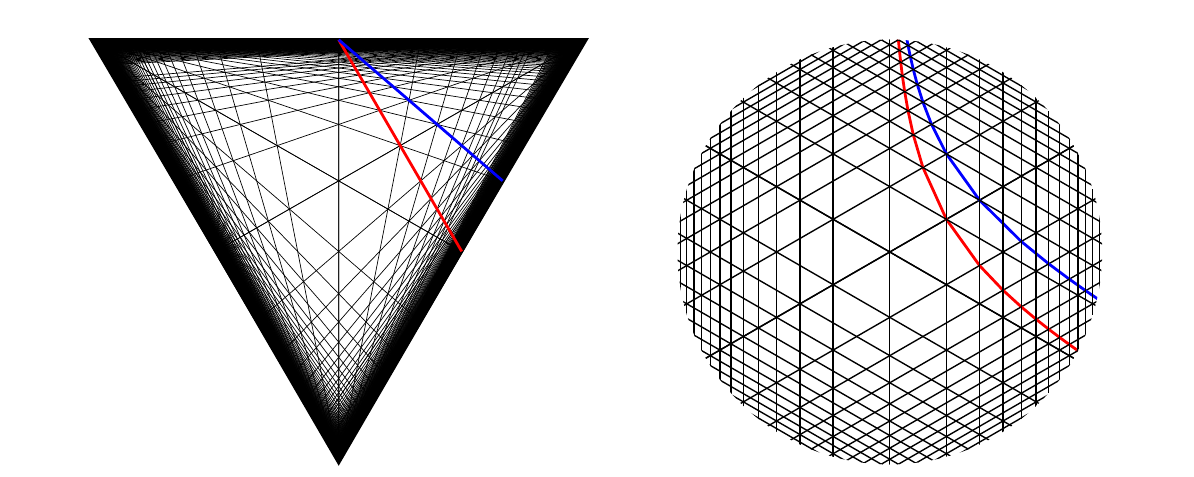}
\caption{}
\label{fig:figure5.1fromoldpaper}
\end{figure}

The subsets of $\nabla^+$ and $\nabla$ given by the equations of the form $\alpha_{1} = m_{2} \alpha_{2}+m_3\alpha_3$, up to
interchanging $\alpha_i$, are \emph{admissible lines}. If additionally $m_2=0$ (or $m_3=0$), then such a line is \emph{principal}.
See
Figure~\ref{fig:figure5.1fromoldpaper} for the set of the admissible lines in $\nabla$ with the projective structure (on the left), and with the
Euclidean metric~$|\cdot,\cdot|$ (on the right). Note that the only admissible lines intersecting the interior of $\nabla^+$ have the precise form $\alpha_1=m_{2} \alpha_{2}+m_3\alpha_3$, where $m_2+m_3\geq 2$, or $\alpha_2=m_3\alpha_3$, where $m_3\geq 2$.

The following Figure~\ref{fig:old_figure4.3} illustrates some possible intersections between $\Ap_{\id}$ and~$\Ap_f$ (and $\Ap^+_{\id}$ and~$\Ap^+_f$
in bold), where $f$ fixes the point $\nu_{\id,[3,2,1]}$.

\begin{rem}
\label{rem:commonrays} Suppose in Corollary~\ref{cor:fixed_set} that $\nu_{\id,[\alpha]}$ is fixed by $f$. If $\alpha_2>\alpha_3$, then for each
$\alpha'=(\alpha_1,\alpha_2,\alpha_3')$ with $\alpha_3'<\alpha_3$, the point $\nu_{\id,[\alpha']}$ is fixed by $f$. If $\alpha_1>\alpha_2$, then for
each $\alpha'=(\alpha'_1,\alpha_2,\alpha_3)$ with $\alpha_1'>\alpha_1$, the point $\nu_{\id,[\alpha']}$ is fixed by $f$.
\end{rem}

\begin{figure}
\includegraphics[scale=1]{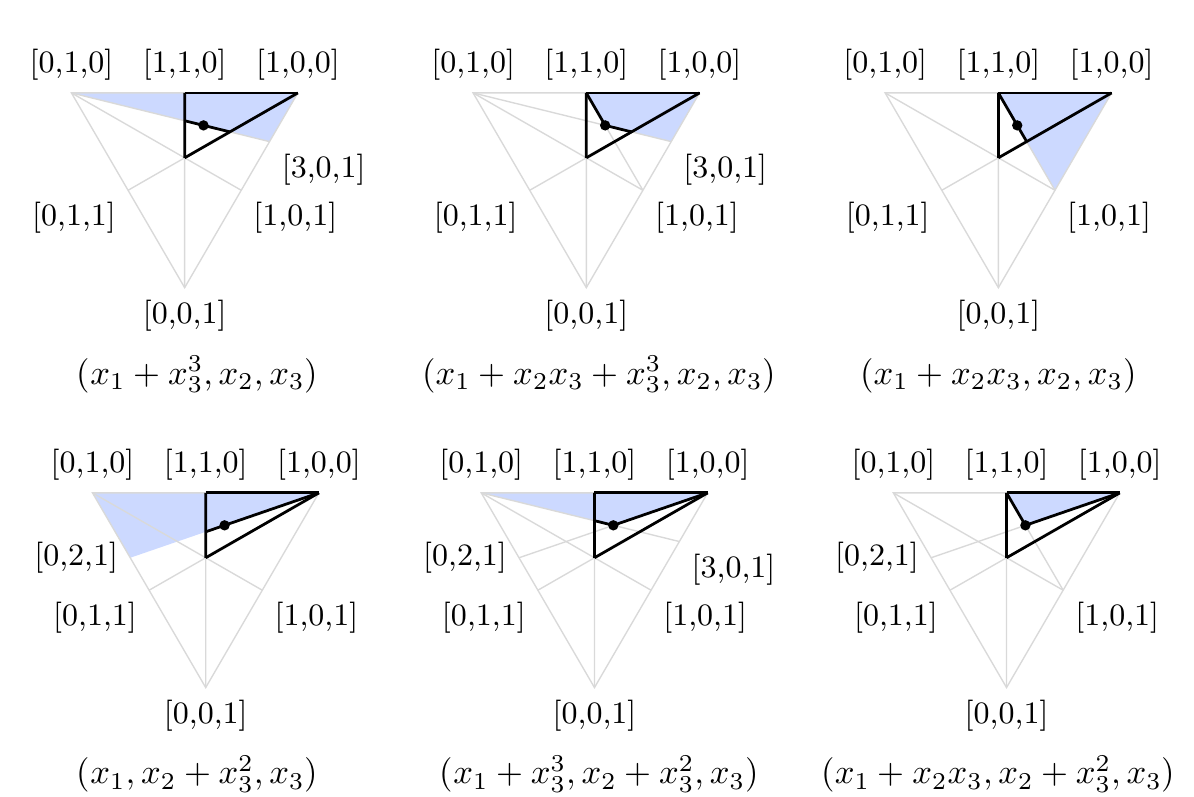}
\caption{Intersections of apartments $\Ap_{\id},\Ap_f$ and chambers $\Ap^+_{\id},\Ap^+_f$ for particular $f$.}
\label{fig:old_figure4.3}
\end{figure}

\subsection{Combinatorial structure on \texorpdfstring{$\X$}{X}}
\label{subs:comb}

A map from a CW-complex $X$ to a CW-complex $Y$ is \emph{combinatorial}, if its restriction to any open cell of $X$ is a homeomorphism onto an open
cell of $Y$. A CW-complex $X$ is \emph{combinatorial}, if the attaching map of each open cell of $X$ is combinatorial for some subdivision of the
sphere. The \emph{star} of a cell~$\sigma$ in a CW-complex $X$ is the union of all the open cells of $X$ containing~$\sigma$ in their closure.

The space $\X$ is equipped with the following structure of a combinatorial CW-complex (which induces a weaker topology than the metric $d_\X$ due to
the lack of local finiteness). Namely, we first equip $\nabla^+$ with the following structure of a combinatorial CW-complex. Its vertices are the
intersection points of distinct admissible lines. Its open edges are the connected components of the complement of the vertices in the admissible lines.
Finally, its open $2$-cells are the connected components of the complement of the admissible lines in~$\nabla^+$. In (the proof of)
\cite{LP2}*{prop~6.3}, we identify~$\X$ with the universal development of a certain complex of groups with underlying complex $\nabla^+$ equipped
with that CW-structure. This equips $\X$ with a CW-complex structure. Equivalently, each open cell of $\X$ is contained in a chamber~$\Ap_f^+$ and it
is the preimage in~$\Ap_f^+$ under~$\rho_+$ of a cell of $\nabla^+$.

\section{Triples of polynomials in two variables}
\label{sec:triple}
In this section we collect several estimates on particular exponents appearing in identities involving triples of polynomials in two variables. This will enable us to study the links of the vertices of $\X$ in Section~\ref{sec:links}.

\begin{lem}
\label{l:trivial}
Let $Q \in \K[y,z]$  be a nonzero homogeneous polynomial of degree $d \ge 0$, let $c,c' \in \K^*$, and let $q \ge 1$ be an integer.
Then the partial derivatives $\partial_y [(cy + c'z)^q Q]$ and $\partial_z [(cy + c'z)^q Q]$
are both nonzero homogeneous polynomials of degree $d + q - 1$ that are multiples of $(cy + c'z)^{q-1}$.
\end{lem}

\begin{proof}
Given a homogeneous polynomial $P \in \K[y,z]$, if $\deg_y P \ge 1$, then we have $\deg \partial_y P = \deg(P) -1$.
This gives the assertion about the degree. The assertion that $(cy + c'z)^{q-1}$ divides the partial derivatives is a straightforward computation. For instance, for the first one:
\[
\partial_y [(cy + c'z)^q Q] = (cy + c'z)^{q-1} [cqQ + (cy + c'z)\partial_y Q].
\qedhere
\]
\end{proof}

\begin{lem}
\label{l:w=1}
Let $R, T, U \in \K[x_2, x_3]$ be nonzero homogeneous polynomials, and let $\ell_1, \ell_2, \ell_3 \in \K[x_2, x_3]$ be pairwise independent linear forms such that we have a relation
\[\ell_1^r R + \ell_2^t T + \ell_3^u U = 0,\]
for some $r,t,u \in \N$.
In particular, the three terms of the sum have the same degree $d = r + \deg R = t + \deg T = u + \deg U$.
Then $\min \{r, t, u \} \le \left\lfloor \frac{2d+1}3 \right\rfloor$.
\end{lem}

\begin{proof}
By a linear change of coordinates we can assume $\ell_1 = x_2$, $\ell_2 = cx_2 + c'x_3$ and $\ell_3 = x_3$, for some $c,c' \in \K^*$,  so that the relation becomes
\[
x_2^r R + (cx_2 + c'x_3)^t T + x_3^u U = 0.
\]
Observe that for an integer $a$, the condition $a \le \left\lfloor \frac{2d+1}3 \right\rfloor$ is equivalent to $a < \frac23 (d+1)$.
For contradiction, assume $r, t, u \ge \frac23 (d + 1)$.
Let $n$ be the unique integer such that  $\tfrac13 (d-2) < n \le \tfrac13 (d+1)$. Then $2n \le \frac23 (d + 1) \le t$.
We have
\[
\deg R = d - r \le d - \tfrac23 (d + 1) = \tfrac13 (d-2) < n,
\]
so $\partial_{x_3}^n [x_2^r R ] = x_2^r \partial_{x_3}^n R = 0$.
Similarly, $\deg U < n$ and $\partial_{x_2}^n [x_3^u U] = 0$.
Taking the partial derivatives $\partial_{x_2}^n\partial_{x_3}^n$ of the expression $E = x_2^r R + (cx_2 + c'x_3)^t T + x_3^u U$,
we obtain
$\partial_{x_2}^n \partial_{x_3}^n E = \partial_{x_2}^n \partial_{x_3}^n [(cx_2 + c'x_3)^t T]$.
Since $2n \le t$, we can apply $2n$ times Lemma \ref{l:trivial} to obtain
$\partial_{x_2}^n \partial_{x_3}^n [(cx_2 + c'x_3)^t T] \neq 0$, hence $E \neq 0$, contradiction.
\end{proof}

We will need the following
weighted versions of Lemma~\ref{l:w=1}.

\begin{lemma}
\label{lem:weighted}
Let $p \ge 1$ be an integer, and let $R, T, U \in \K[x_2, x_3]$ be nonzero weighted homogeneous polynomials with $x_2, x_3$ of respective weights $p, 1$.
Assume that there exist $r,t,u,m\in \N$ such that
\[
\deg R + rp = \deg T + tp = \deg U + up = m.
\]
Suppose that for some distinct $c, c' \in \K^*$ we have
\[
x_2^r R + (x_2 + cx_3^p)^t T + (x_2 + c'x_3^p)^u U = 0.
\]
Let $m = \alpha p + \beta$ be the Euclidean division of $m$ by $p$, and assume $\alpha \ge 2$.
Then $\min \{rp, tp, up \} \le \frac34m$.
\end{lemma}

\begin{rem}
\label{rem:alpha1}
In the case $\alpha = r = t = u = 1$, up to scaling the only possibility for $R, T, U$ satisfying the identity in Lemma~\ref{lem:weighted} is
\begin{align*}
R = (c - c')x_3^\beta, &&
T = c' x_3^\beta, &&
U = -c x_3^\beta.
\end{align*}
\end{rem}

\begin{proof}[Proof of Lemma~\ref{lem:weighted}]
There exist (ordinary) homogeneous polynomials $\widetilde R, \widetilde T, \widetilde U$ satisfying
\begin{align*}
R(x_2,x_3) &= x_3^\beta \widetilde R(x_2,x_3^p), \\
T(x_2,x_3) &= x_3^\beta \widetilde T(x_2,x_3^p), \\
U(x_2,x_3) &= x_3^\beta \widetilde U(x_2,x_3^p).
\end{align*}
By assumption, we have
\[
x_2^r x_3^\beta \widetilde R(x_2,x_3^p) + (x_2 + cx_3^p)^t x_3^\beta \widetilde T(x_2,x_3^p) + (x_2 + c'x_3^p)^u x_3^\beta \widetilde U(x_2,x_3^p) = 0.
\]
Dividing by $x_3^\beta$ and changing variables we get
\[
x_2^r \widetilde R(x_2,x_3) + (x_2 + cx_3)^t \widetilde T(x_2,x_3) + (x_2 + c'x_3)^u  \widetilde U(x_2,x_3) = 0,\]
where each term of the sum is homogeneous of degree $\frac{m-\beta}p = \alpha$.
By Lemma \ref{l:w=1}, we have $\min\{r, t, u\} \le   \left\lfloor \frac{2\alpha+1}3 \right\rfloor$.
Observe that
\begin{equation}
\label{spadesuit}
\max_{\alpha \ge 2} \frac{\left\lfloor \frac{2\alpha+1}3 \right\rfloor }{\alpha} = \frac34 , \text{ realised for } \alpha = 4. \tag{$\spadesuit$}
\end{equation}
We conclude that
\[
\min \{rp, tp, up\} \le  \left\lfloor \frac{2\alpha+1}3 \right\rfloor p \le \frac{\left\lfloor \frac{2\alpha+1}3 \right\rfloor }{\alpha}\alpha p \le \frac34 m.
\qedhere
\]
\end{proof}

\begin{lemma}
\label{lem:weightedplus}
Let $p \ge 1$ be an integer, and let $R, T, U \in \K[x_2, x_3]$ be nonzero weighted homogeneous polynomials with $x_2, x_3$ of respective weights $p, 1$.
Assume that there exist $r,t,u,m\in \N$ such that
\[
\deg R + rp = \deg T + tp = \deg U + u = m.
\]
Suppose that for some $c \in \K^*$ we have
\[
x_2^r R + (x_2 + c x_3^p)^t T + x_3^u U = 0.
\]
Let $m = \alpha p + \beta$ be the Euclidean division of $m$ by $p$, and
assume $\alpha \ge 2$. Then $\min \{rp, tp, u \} < \frac45m$.
\end{lemma}

\begin{rem}
\label{rem:alpha1plus}
In the case $\alpha = r = t =1$, and $u=m$, up to scaling the only possibility for $R, T, U$ satisfying the identity in Lemma~\ref{lem:weightedplus} is
\begin{align*}
R = -x_3^\beta, &&
T = x_3^\beta, &&
U = -c.
\end{align*}
\end{rem}

\begin{proof}[Proof of Lemma~\ref{lem:weightedplus}]
We can assume that $x_3$ does not divide $U$. Write the Euclidean division $m - u = \alpha' p + \beta'$.
Then there exist (ordinary) homogeneous polynomials $\widetilde R, \widetilde T, \widetilde U$ satisfying
\begin{align*}
R(x_2,x_3) &= x_3^\beta \widetilde R(x_2,x_3^p), \\
T(x_2,x_3) &= x_3^\beta \widetilde T(x_2,x_3^p), \\
U(x_2,x_3) &= x_3^{\beta'} \widetilde U(x_2,x_3^p).
\end{align*}
By assumption, we have
\[
x_2^r  x_3^\beta \widetilde R(x_2,x_3^p) + (x_2 + c x_3^p)^t  x_3^\beta \widetilde T(x_2,x_3^p)  + x_3^u x_3^{\beta'} \widetilde U(x_2,x_3^p) = 0.
\]
We see that $u + \beta' - \beta$ must be a multiple of $p$, and writing $u + \beta' - \beta = pu'$ with $u' \ge 0$, we get
\[
x_2^r \widetilde R(x_2,x_3) + (x_2 + c x_3)^t \widetilde T(x_2,x_3) + x_3^{u'} \widetilde U(x_2,x_3) = 0.
\]
By Lemma \ref{l:w=1}, we have $\min\{r, t, u'\} \le  \left\lfloor \frac{2\alpha+1}3 \right\rfloor$.
By definition, we have $u\le pu' + \beta,$ and so multiplying by $p$ we obtain
\[
\min \{rp, tp, u\} \le \left\lfloor \frac{2\alpha+1}3 \right\rfloor p + \beta.
\]
Observe that for any real numbers $0 < c_1 < c_2$, we have $c_1 p + \beta < \frac{c_1+1}{c_2+1} (c_2p + \beta)$.
Applying this observation with $c_1 = \left\lfloor \frac{2\alpha+1}3 \right\rfloor$ and $c_2 = \alpha$, we obtain
\[
\min \{rp, tp, u\} \le \left\lfloor \frac{2\alpha+1}3 \right\rfloor p + \beta < C (\alpha p + \beta) = Cm
\]
for
\[C = \max_{\alpha \ge 2} \frac{\left\lfloor \frac{2\alpha+1}3 \right\rfloor + 1}{\alpha + 1} = \frac45.\]
(The maximum is realised for $\alpha = 4.$)
\end{proof}

\section{Links}
\label{sec:links}

In this section we will study the link of each vertex $\nu$ in $\X$, which is the following metric graph \cite{BB}*{page~176}. Namely, the vertices of
the link correspond to the edges of $\X$ containing $\nu$ and the edges of the link correspond to the corners of the $2$-cells of $\X$ containing
$\nu$. The length of each edge of the link is the angle in $\X$ at the corresponding corner. Note that the link is not finite but has finitely many
possible edge lengths.

Since $\T$ maps any vertex into $\Ap^+_\id$, we will assume $\nu\in \Ap^+_\id$. For example, by \cite{LP2}*{lem~ 3.4} the link of $\nu_{\id,[1,1,1]}$
in $\X$ can be identified with the Bruhat--Tits building of $\GL_3(\K)$. Consequently, each cycle in the link of length $<2\pi +\frac{2\pi}{3}$ has
length exactly $2\pi$ and is the link of $\nu$ in $\Ap_f$ for $f\in A$.

We will now perform a similar classification for other weights. In \cite{LP2}*{prop~7.9, 7.10 and~7.12} we proved that the link of $\nu$ is
$\mathrm{CAT}(1)$, i.e.\ all of its cycles have length $\geq 2\pi$. Our aim is to classify all of its cycles of length $<2\pi+\eps_0$, for some small
$\eps_0>0$ independent of $\nu$.

\subsection{Links of vertices with weight \texorpdfstring{$[m,m,1]$}{[m,m,1]}}

Let $\nu=\nu_{\id,[m,m,1]}$ for an integer $m\geq 2$. Let $\Gamma$ be the link of $\nu$ in $\X$. Let $I$ be the link of $[m,m,1]$ in $\nabla^+$,
which is a simplicial path consisting of $3$ edges of length $\frac{\pi}{3}$ \cite{LP2}*{rem~5.2}, see Figure~\ref{fig:mm1}. We keep the notation
$\rho_+$ for the induced map from~$\Gamma$ to $I$. Let $s,q\in I$ be the directions towards $[1,1,1]$ and towards~$[1,1,0]$, respectively. For $f\in
\St(\nu)$, let $\Gamma_f,\Gamma^+_f\subset \Gamma$ be the cycle and the path, respectively, that are the links of $\nu$ in $\Ap_f,\Ap_f^+$. Let $s_f,q_f\in
\Gamma^+_f$ be the preimages under $\rho_+$ of $s,q,$ respectively.

\begin{figure}
\includegraphics[scale=1]{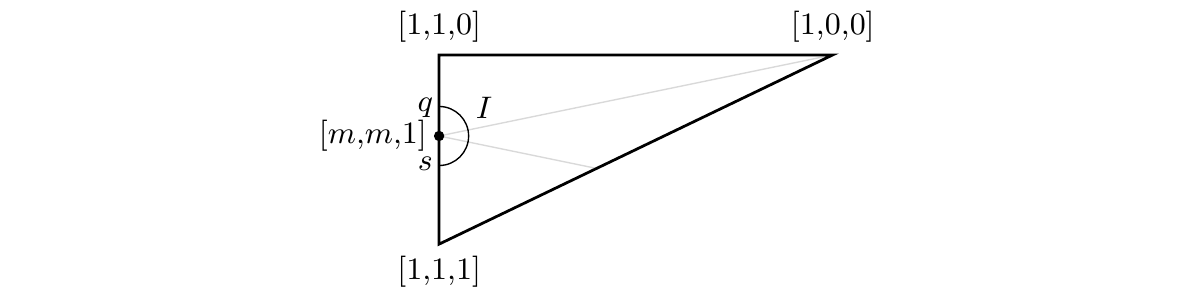}
\caption{}
\label{fig:mm1}
\end{figure}

\begin{prop}
\label{prop:linkpp1}
Each cycle in $\Gamma$ of length $<2\pi +\frac{2\pi}{3}$ has length $2\pi$ and
\begin{itemize}
\item equals $\Gamma_f$, or
\item is the closure of $(\Gamma_f\cup\Gamma_g)\setminus (\Gamma_f\cap \Gamma_g)$, or
\item intersects $\Gamma_f,\Gamma_g, \Gamma_h$ along length $\frac{2\pi}{3}$ paths with midpoints $s_f,s_g,s_h$.
\end{itemize}
\end{prop}

In the proof we will need the following.

\begin{lemma}[\cite{LP2}*{lem~7.7(iii)}]
\label{lem:branch} Each vertex $u\in \Gamma \setminus \rho_+^{-1}\{q,s\}$ is contained in exactly one edge~$e$ such that $\rho_+(e)$ separates
$\rho_+(u)$
    from $q$ in $I$. Moreover, if $u\in \Gamma^+_f$ for some $f\in \Stab(\nu)$, then $e\subset \Gamma^+_f$.
\end{lemma}

\begin{lemma}[\cite{LP2}*{lem~7.8}]
\label{lem:bruhat}
Suppose that $s_f=s_{f'}$ for some $f,f' \in \Stab(\nu)$. Then there exists $f'' \in \Stab(\nu)$ with $\Gamma^+_f,\Gamma^+_{f'}\subset
\Gamma_{f''}$.
\end{lemma}

\begin{proof} [Proof of Proposition~\ref{prop:linkpp1}]
Since $\Gamma$ is bipartite and all of its edges have length $\frac{\pi}{3}$, we only need to classify cycles $\gamma$ of length $2\pi$, each of
which consists of six edges. We orient each edge of $\gamma$, so that its image in $I$ under $\rho_+$ is
oriented from $q$ to $s$. By Lemma~\ref{lem:branch}, $\gamma$ decomposes into $j=2,4,$ or $6$ oriented paths each of which lies in an $\Gamma^+_f$
and ends with $s_f$. By Lemma~\ref{lem:bruhat}, the cases $j=2,4,6,$ correspond to the respective bullet points of the proposition.
\end{proof}

\subsection{Links of vertices with weight \texorpdfstring{$[m,1,1]$}{[m,1,1]}}

Let $m\geq 2$ be an integer.

\begin{rem}
\label{rem:numberline} The weight $[\alpha]=[m,1,1]$ lies on exactly $m+2$ admissible lines, namely $\alpha_2 = \alpha_3$ and $\alpha_1 = a\alpha_2 +
(m-a)\alpha_3$, where $a = 0, \dots, m$. The latter lines contain points $[a,1,0]$, and are principal exactly for $a=0$ and $a=m$. See
Figure~\ref{fig:m11}.
\end{rem}

\begin{figure}
\includegraphics{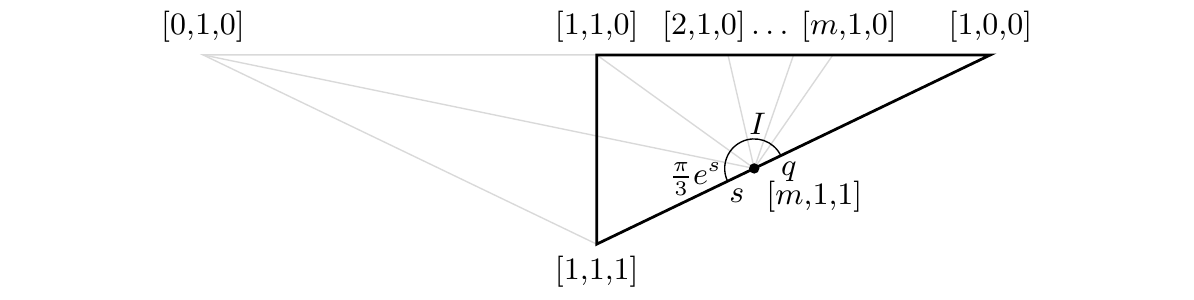}
\caption{}
\label{fig:m11}
\end{figure}

Let $\nu=\nu_{\id,[\alpha]}$ with $[\alpha]=[m,1,1]$. Let $\Gamma_0$ be the link of $\nu$ in $\X$.
We keep the notation $\rho_+$ for the induced map from $\Gamma_0$ to the link $I$ of $[\alpha]$ in $\nabla^+$, which is a simplicial path. Let
$s,q\in I$ be the directions towards $[1,1,1]$ and towards $[1,0,0]$, respectively. Let $e^s\subset I$ be the edge containing $s$. For $f\in
\St(\nu)$, let $\Gamma_f,\Gamma^+_f\subset \Gamma_0$ be the cycle and the path, respectively, that are the links of $\nu$ in $\Ap_f,\Ap_f^+$. Let
$s_f,q_f,e^s_f\subset \Gamma^+_f$ be the preimages under $\rho_+$ of $s,q,e^s,$ respectively. Lemmas~\ref{lem:branch} and~\ref{lem:bruhat} remain
true in this setting, with $\Gamma$ replaced by $\Gamma_0$ (see the first paragraph of the proof of \cite{LP2}*{prop~7.10}).

\begin{prop}
\label{prop:linkp11} There is $\eps_0>0$ independent of $m$ such that each cycle in $\Gamma_0$ of length $<2\pi +\eps_0$
\begin{itemize}
\item equals $\Gamma_f$, or
\item is the closure of $(\Gamma_f\cup\Gamma_g)\setminus (\Gamma_f\cap \Gamma_g)$.
\end{itemize}
\end{prop}

Note that unlike in Proposition~\ref{prop:linkpp1}, here in the second bullet point we can have cycles of length $>2\pi$ arbitrary close to $2\pi$. 

In Example~\ref{exa:cycle} we will see that setting $\eps_0=\frac{2\pi}{3}$ does not suffice.

In the proof, we will need the following. For $f,g\in \Stab(\nu)$, we write $f \sim_{\nu} g$ if $\Gamma^+_f=\Gamma^+_g$. Let $\Stab'(\nu)< \Stab(\nu)$
be the subgroup of elements of form $$(x_1 + P(x_2, x_3), bx_2 + cx_3+d, b'x_2 + c'x_3+d'),$$ which differs from $\Stab(\nu)$ only by requiring the
coefficient at $x_1$ to be equal $1$. Let $N_\alpha< \Stab'(\nu)$ be the subgroup of elements of form $$(x_1 + P(x_2, x_3), x_2+d, x_3+d')$$ with
$\deg P < m$.

\begin{lemma}\phantomsection
\label{lem:normal}
\begin{enumerate}[(i)]
\item For each $f\in \Stab(\nu)$, there is $g\in \Stab'(\nu)$ with $f\sim_\nu g$.
\item Let $h\in N_{\alpha}$. Then $h\sim_\nu \id$.
\item $N_\alpha$ is normal in $\Stab'(\nu)$.
\end{enumerate}
\end{lemma}

\begin{proof}
Part (i) follows from the description of $\Stab(\nu)$ in Proposition~\ref{prop:stab} and the observation \cite{LP2}*{cor~3.3} that for
$l=(ax_1,x_2,x_3)$ we have $l\sim_\nu \id$. Part (ii) follows from Corollary~\ref{cor:fixed_set}. Part (iii) is an easy computation.
\end{proof}

\begin{lemma}
\label{lem:sector} Let $f,g\in \Stab'(\nu)$. If $\Gamma^+_f\neq \Gamma^+_g$ and $\Gamma^+_f\cap \Gamma^+_g\neq \{q_f\}$, then either
\begin{enumerate}[(i)]
\item $\Gamma^+_f\cap \Gamma^+_g=\{s_f,q_f\}$ and $f^{-1}g \in hN_\alpha$ for some $h=(x_1, bx_2 + cx_3, b'x_2 + c'x_3)$ with $b'\neq 0$, or

\item $\Gamma^+_f\cap \Gamma^+_g$ is the path from $q_f$ to the direction mapping under $\rho_+$ to the direction towards $[a,1,0]$, where $0\leq
    a\leq m$ and $f^{-1}g \in hN_\alpha$ for some $$h = (x_1 + \sum_{i = 0}^{a} c_i  x_2^i x_3^{m - i}, bx_2 + cx_3, c'x_3)$$ with $c_a\neq 0$.
\end{enumerate}
\end{lemma}

\begin{proof}
After composing $f$ and $g$ with $f^{-1}$, we can assume $f=\id$. Let $g=(x_1 + P(x_2, x_3), bx_2 + cx_3+d, b'x_2 + c'x_3+d')$. By
Lemma~\ref{lem:normal}(ii,iii), after possibly composing with an element of $N_\alpha$, we can assume that $P$ is homogeneous of degree~$m$, and
$d=d'=0$. If $P$ is trivial, then by Corollary~\ref{cor:fixed_set} we have (i) for $b'\neq 0$ or $\Gamma^+_f=\Gamma^+_g$ for $b'=0$. Analogously, if
$P$ is nontrivial, then by Corollary~\ref{cor:fixed_set} we have~(ii) for $b'=0$ or $\Gamma^+_f\cap \Gamma^+_g=\{q_f\}$ for $b'\neq 0$.
\end{proof}

\begin{lem}
\label{lem:isometry} Let $0\leq \delta \leq 1$. Let $[\alpha],[\alpha']\in \nabla$ and let $c,c'$ be the projective line segments joining
$[\alpha],[\alpha']$ to $[\delta\frac{\alpha_1}{\alpha_2},1,0],[\delta\frac{\alpha'_1}{\alpha'_2},1,0]$, respectively. Then the translation of
$(\nabla,|\cdot|)$ mapping $[\alpha]$ to $[\alpha']$ maps $c$ to $c'$.
\end{lem}
\begin{proof}
In the proof of \cite{LP2}*{lem~5.9} we showed that the reflections in the lines parallel to the principal lines in $(\nabla,|\cdot|)$ extend to
projective maps of $P\R^3$ (preserving~$\nabla$). Thus the translation $\tau$ mapping $[\alpha]$ to $[\alpha']$ maps the projective line through
$[\alpha]$ and $[0,0,1]$ to the projective line through $[\alpha']$ and $[0,0,1]$ and hence it maps $[\frac{\alpha_1}{\alpha_2},1,0]$ to
$[\frac{\alpha'_1}{\alpha'_2},1,0]$. Since $\tau$ preserves $[0,1,0]$ and $[1,0,0]$, as well as the cross-ratio on the projective line through
$[0,1,0]$ and $[1,0,0]$, we conclude that $\tau$ maps $[\delta\frac{\alpha_1}{\alpha_2},1,0]$ to $[\delta\frac{\alpha'_1}{\alpha'_2},1,0]$.
\end{proof}

\begin{proof}[Proof of Proposition~\ref{prop:linkp11}]

Let $\gamma\subset \Gamma_0$ be a cycle of length $<2\pi+\frac{2\pi}{3}$. We orient each edge of $\gamma$, so that its image in $I$ under $\rho_+$ is
oriented from $q$ to $s$. By Lemma~\ref{lem:branch}, $\gamma$ decomposes into an even number $j$ of oriented paths each of which lies in an
$\Gamma^+_f$ and ends with $s_f$. By Remark~\ref{rem:numberline}, each such path contributes at least $\frac{\pi}{3}$ to the length of $\gamma$.
Since $\gamma$ has length $<2\pi+\frac{2\pi}{3}$, we have $j<8$ and so $j=2,4,$ or $6$. The cases $j=2,4$ imply the first two bullet points of the
proposition by
Lemma~\ref{lem:bruhat}.

Suppose now $j=6$, with $\gamma=\gamma_0\cdot \gamma_1^{-1}\cdot \gamma_2\cdots \gamma_5^{-1}$, where each $\gamma_i$ lies in
$\Gamma^+_{f_i}$, for some $f_i\in \St(\nu)$. Note that these paths contribute already $2\pi$ to the length of~$\gamma$, and so for a contradiction
we need to find $\eps_0$ bounding the excess contribution from below. By Remark~\ref{rem:numberline}, points $\rho_+(\gamma_1\cap \gamma_2),
\rho_+(\gamma_3\cap \gamma_4), \rho_+(\gamma_5\cap \gamma_0)$ correspond to directions towards $[a,1,0],[a',1,0],[a'',1,0]$ for some $0\leq
a,a',a''\leq m$. We need to find $\eps_0$ independent of $m$ bounding below at least one of the angles at $[m,1,1]$ between the admissible lines in
these directions and the direction towards $[0,1,0]$. To do that, we will find $\delta>0$ independent of $m$ such that $\max \{a,a',a''\}\geq \delta
m$. We can then assign $\eps_0$ to be the angle at $[m,1,1]$ between the projective lines in the directions towards $[\delta m, 1,0]$ and $[0,1,0]$,
which is independent of $m$ by Lemma~\ref{lem:isometry}.

By Lemma~\ref{lem:normal}(i), we can suppose that all $f_i$ belong to $\Stab'(\nu)$. By
Lemma~\ref{lem:sector}, we have $f_0^{-1}f_1\in h_0N_\alpha, \ldots, f_5^{-1}f_0\in h_5N_\alpha$, where the second and third components of
$h_i=(x_1+P_i,\cdot,\cdot)$ form a linear automorphism in $x_2,x_3$. Furthermore,
\begin{itemize}
\item
for $i=1,3,5,$ the polynomial $P_i$ is homogeneous in $x_2,x_3$ of
degree $m$, and the third component of $h_i$ does not depend on $x_2$, while
\item
for $i=0,2,4,$ we have $P_i=0$ and the third component of $h_i$ depends on $x_2$.
\end{itemize}
By Lemma~\ref{lem:normal}(iii), the group
$N_\alpha$ is normal in $\Stab'(\nu)$ and so by Lemma~\ref{lem:normal}(ii) we have $h_0\cdots h_5\sim_\nu \id$. Thus the first component of
$h_0\cdots h_5$ equals $x_1+P$ for $P$ a polynomial in $x_2,x_3$ of degree $<m$. Since $P_i$ are homogeneous, we have $P=0$. By
Lemma~\ref{lem:sector}, we have that $P_1,P_3,P_5$ are divisible by $x_3^{m-a},x_3^{m-a'},x_3^{m-a''}$, respectively.
Since the first component of $h_0\cdots h_5$
equals~$x_1$, we obtain
$$P_1(\cdot, \ell_1)+P_3(\cdot, \ell_2)+P_5(x_2,x_3)=0,$$
where $\ell_1,\ell_2,x_3$ are pairwise independent linear forms in $\K[x_2,x_3]$.
Now we apply Lemma~\ref{l:w=1} with $$r=m-a,\, \ell_1^rR=P_1, \, t=m-a',\,
\ell_2^tT=P_3,\, u=m-a'',\, x_3^uQ=P_5.$$ 
Thus one of $a,a', a''$ is $\geq m-\big\lfloor\frac{2m+1}{3}\big\rfloor$, and so by \eqref{spadesuit} in Section~\ref{sec:triple} it is $\geq
\frac{m}{4}$. We can then set $\delta=\frac{1}{4}$.
\end{proof}

\begin{exa}
\label{exa:cycle} In the notation of the proof of Proposition~\ref{prop:linkp11}, let $m=2$, let
$$h_4=(x_1,x_3,x_2),\, h_2=(x_1,x_3,x_2+x_3),\, h_0=(x_1,x_2,
x_3-x_2),$$
$$P_5=-x_3^2,\, P_3=x_3^2,\, P_1=x_3(x_3-2x_2).$$ Then the identity $x_2^2-x_3^2+(x_2+x_3)(x_3-x_2)$ implies $h_0\cdots h_5=\id$. The paths
$\gamma_i$ for $i=1,2,3,4$ are of length $\frac{\pi}{3}$, while the paths $\gamma_0,\gamma_5$ are of length $\frac{\pi}{2}$ \cite{LP2}*{lem~5.9}.
Thus $\gamma$ is a cycle in $\Gamma$ of length $2\pi+\frac{\pi}{3}$.
\end{exa}

\begin{rem}
\label{rem:unique_admissible} Let $L_0,L_1$ be any non-principal admissible lines in $\nabla$. Then there is an isometry $(\nabla,|\cdot|)$ mapping
$L_0$ to $L_1$. Indeed, we can assume without loss of generality that the projective intervals $L_0,L_1$ have endpoints $x_0^+,x_1^+$ in the open
projective interval from $[1,0,0]$ to $[0,1,0]$ and  $x_0^-,x_1^-$ in the open projective interval from $[1,0,0]$ to $[0,0,1]$. Let $x_i$ be the
intersection point of the projective line through $[0,0,1]$ and $x_i^+$ and the projective line through $[0,1,0]$ and $x_i^-$. Then as in the proof
of Lemma~\ref{lem:isometry}, the translation of $(\nabla,|\cdot|)$ mapping $x_0$ to $x_1$ maps $x^\pm_0$ to $x^\pm_1$, and hence maps $L_0$ to $L_1$.
\end{rem}

\subsection{Links of vertices with weight in \texorpdfstring{$\mathrm{int}\,\nabla^+$}{int nabla+}}

\begin{figure}
\includegraphics[scale=1]{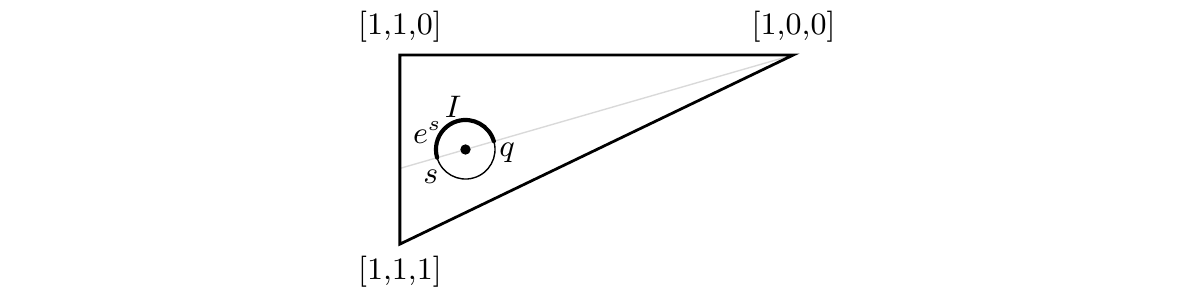}
\caption{}
\label{fig:mp1}
\end{figure}

Let $\nu=\nu_{[\alpha]}$ with $[\alpha]=[\alpha_1,\alpha_2,\alpha_3]\in \mathrm{int}\,\nabla^+$. Let $\Gamma_0$ be the link of $\nu$ in $\X$. We keep
the notation $\rho_+$ for the induced map from $\Gamma_0$ to the cycle $S$ that is the link of $[\alpha]$ in $\nabla^+$. Let $s,q\in S$ be the
directions towards $[0,\alpha_2,\alpha_3]$ or towards $[1,0,0]$, respectively. Let $I\subset S$ be the path defined by
$\frac{\alpha'_2}{\alpha'_3}\geq \frac{\alpha_2}{\alpha_3}$. If $s$ is a vertex of $S$, then let $e^s\subset S$ be the edge containing $s$ lying
inside $I$. See Figure~\ref{fig:mp1}. For $f\in \St(\nu)$, let $\Gamma_f\subset \Gamma_0$ be the cycle that is the link of~$\nu$ in~$\Ap_f$, and let
$s_f,q_f,e^s_f\subset \Gamma_f$ be the preimages under~$\rho_+$ of $s,q,e^s,$ respectively.

\begin{exa}\phantomsection
\label{exa:1}
\begin{enumerate}[(a)]
\item Let
$$h_0=(x_1,x_2-cx_3^p,x_3),\ h_2=(x_1,x_2+cx_3^p,x_3),$$
$$h_1=(x_1+x_2^{a'}P'(x_2,x_3),x_2,x_3),\ h_3=(x_1+x_2^{a}P(x_2,x_3),x_2,x_3),$$
where $c,P,P'\neq 0$. Suppose that we have $h_0\cdots h_3=\id$, and for $i=0,\ldots, 3$, let $f_i=h_0\cdots h_i$, where $f_3=\id$. By
Corollary~\ref{cor:fixed_set}, we obtain a cycle in the link of $\nu$ decomposing into paths in $\Gamma_{f_i}$, examples of which are illustrated
in Figure~\ref{fig:loops}(a), left and right. We proved in \cite{LP2}*{page 40} that we have $p(a+a')\leq m$ and consequently the cycle has
length $\geq 2\pi$ \cite{LP2}*{lem~5.9}. Equality cases include, for example for $m=10p$, the configurations where $c=1$, and $a=10, P=1, a'=0$
(illustrated on the right), or $a=3, P=(x_2+x_3^p)^7, a'=7,$ (illustrated on the left).
\item
Let $m=p+\beta$ with $0<\beta<p$, and let $\alpha=(m,p,1)$. Let
$$h_0=(x_1,x_2-c'x_3^p,x_3),\ h_2=(x_1,x_2+(c'-c)x_3^p,x_3),\  h_4=(x_1,x_2+cx_3^p,x_3),$$
$$h_1=(x_1-bcx_2x_3^\beta,x_2,x_3),\ h_3=(x_1+bc'x_2x_3^\beta,x_2,x_3),\ h_5=(x_1+b(c-c')x_2x_3^\beta,x_2,x_3),$$
where $c,c',c'-c,b\neq 0$. Note that we have $h_0\cdots h_5=\id$. For $i=0,\ldots, 5$, let $f_i=h_0\cdots h_i$, where $f_5=\id$. By
Corollary~\ref{cor:fixed_set}, we obtain a cycle in the link of~$\nu$ decomposing into paths in $\Gamma_{f_i}$ illustrated in
Figure~\ref{fig:loops}(b).
\item As in (b), let $m=p+\beta$ with $0<\beta<p$, and let $\alpha=(m,p,1)$. Let
$$h_0=(x_1-bcx_3^m,x_2-cx_3^p,x_3),\ h_2=(x_1,x_2+cx_3^p,x_3),$$
$$h_1=(x_1+bx_2x_3^\beta,x_2,x_3),\ h_3=(x_1-bx_2x_3^\beta,x_2,x_3),$$
where $c,b\neq 0$. Note that we have $h_0\cdots h_3=\id$. For $i=0,\ldots, 3$, let $f_i=h_0\cdots h_i$, where $f_3=\id$. By
Corollary~\ref{cor:fixed_set}, we obtain a cycle in the link of $\nu$ decomposing into paths in $\Gamma_{f_i}$ illustrated in
Figure~\ref{fig:loops}(c).
\end{enumerate}
\end{exa}

\begin{prop}
\label{prop:linkmp1} There is $\eps_0>0$ independent of $\alpha$ such that each cycle in $\Gamma_0$ of length $<2\pi +\eps_0$
\begin{itemize}
\item equals $\Gamma_f$, or
\item is the closure of $(\Gamma_f\cup\Gamma_g)\setminus (\Gamma_f\cap \Gamma_g)$, or
\item is a cycle in Example~\ref{exa:1}(a), or
\item is the cycle in Example~\ref{exa:1}(b), or
\item is the cycle in Example~\ref{exa:1}(c),
\end{itemize}
up to a translation by an element in $\St(\nu)$ in the last three bullet points.
\end{prop}

\begin{figure}
\includegraphics{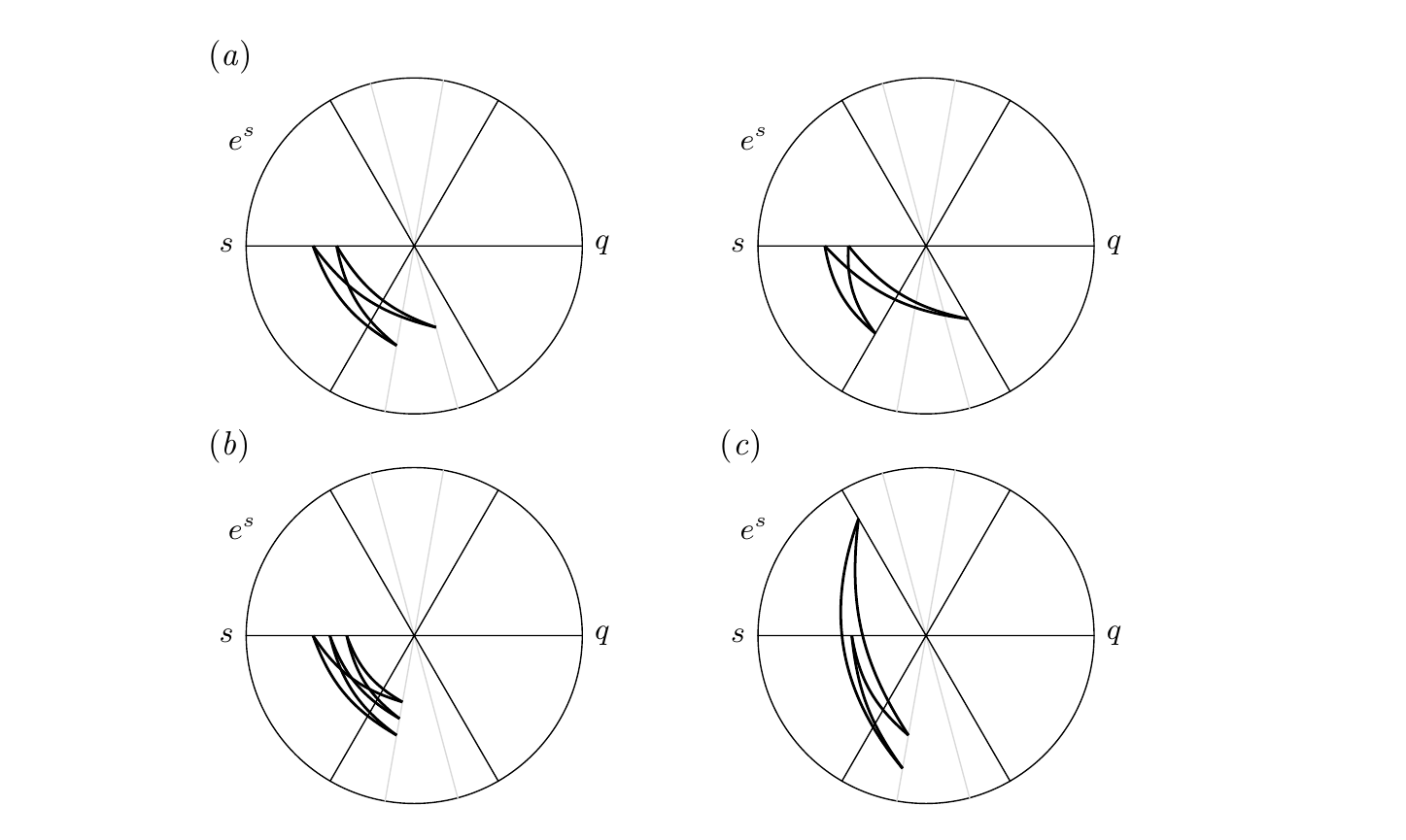}
\caption{}
\label{fig:loops}
\end{figure}

\begin{proof}
Let $\gamma\subset \Gamma_0$ be a cycle of length $<2\pi+\frac{2\pi}{3}$. We orient each edge of $\gamma$ with interior not containing an $s_f$ or a
$q_f$, so that its image in $S$ under $\rho_+$ is oriented from~$q$ to $s$. We divide each edge of $\gamma$ containing an $s_f$ (respectively, $q_f$)
into two edges oriented towards $s_f$ (respectively, away from $q_f$). As in the previous subsections, $\gamma$ decomposes into an even number $j$ of
oriented paths each of which ends with an $s_f$. Each such path is contained in an $\Gamma_g$ and contributes at least $\frac{\pi}{3}$ to the length
of $\gamma$. Since $\gamma$ has length $<2\pi+\frac{2}{3}\pi$, we have $j<8$ and so $j=2,4,$ or $6$. The case $j=2$ implies one of the first two
bullet points of the proposition. In the remaining cases we have that $\rho_+(\gamma)$ is disjoint from $q$.

Consider now the case $j=4$, with the paths of $\gamma$ oriented towards $s_f$ and $s_g$. In the second paragraph of the proof of
\cite{LP2}*{prop~7.12} we observed that if $s$ is a vertex, then $e^s_f$ is the only edge of $\Gamma_0$ containing $s_f$ whose image under $\rho_+$
lies in~$I$. Thus if $\gamma$ contains $e^s_f$, then the restriction of $\rho_+$ to $\gamma$ is locally injective at $s_f$. Since $\rho_+(\gamma)$ is
disjoint from $q$, it follows that $\gamma$ contains also $e_g^s$. This implies the second bullet point of the proposition. The same conclusion
applies when $s$ is not a vertex. If $\gamma$ contains neither $e^s_f$ nor $e_g^s$, then it is obtained by concatenating paths in four distinct
$\Gamma_f$, each of which contains $s_f$ as an endpoint. By \cite{LP2}*{page~39}, this brings us to Example~\ref{exa:1}(a).

Finally, we suppose $j=6$, with the paths of $\gamma$ oriented towards $s_f,s_g,$ and $s_h$. Note that the six paths contribute already $2\pi$ to the
length of $\gamma$, and so for a contradiction we need to find $\eps_0$ bounding the excess contribution from below. Suppose first that $\gamma$
contains none of $e^s_f,e^s_g,e^s_h$, as in Figure~\ref{fig:loops}(b). Since $s$ is a vertex, by \cite{LP2}*{rem~4.9} we have $[\alpha]=[m,p,1]$ for
some integers $m>p>1$. We have a decomposition $\gamma=\gamma_0\cdot \gamma_1^{-1} \cdots \gamma_5^{-1}$, where each $\gamma_i$ lies in
$\Gamma_{f_i}$, for some $f_i\in \St(\nu)$. By \cite{LP2}*{rem~4.10}, points $\rho_+(\gamma_1\cap \gamma_2), \rho_+(\gamma_3\cap \gamma_4),
\rho_+(\gamma_5\cap \gamma_0)$ correspond to directions towards $[m-pa,0,1],[m-pa',0,1],[m-pa'',0,1]$ for some $0\leq a,a',a''\leq \big\lfloor
\frac{m}{p} \big\rfloor.$

For $\frac{m}{p}\geq 2$, we will find
$\eps_0$ bounding below at least one of the angles between the admissible lines in these directions and the direction towards $[0,0,1]$. These angles are
equal to the opposite ones: the angles between the admissible lines in the directions towards $[a,1,0],[a',1,0],[a'',1,0]$ and the direction towards
$[m,p,0]$. Thus by Lemma~\ref{lem:isometry}, we will find appropriate $\eps_0$ as long as there exists $\delta<1$ independent of $m$ such that $\min
\{a,a',a''\}\leq \delta \frac{m}{p}$.

By \cite{LP2}*{lem~4.11(i)}, we can suppose that all $f_i$ lie in
$M_\alpha$. By \cite{LP2}*{lem~4.12}, we have $f_0^{-1}f_1\sim_\nu h_0,\ldots, f_5^{-1}f_0\sim_\nu h_5$, where $f\sim_\nu g$ means
$\Gamma_f=\Gamma_g$, and $h_i$ have the following form:
$$h_0=(x_1,x_2-c'x_3^p,x_3),\ h_2=(x_1,x_2+c''x_3^p,x_3),\  h_4=(x_1,x_2+cx_3^p,x_3),$$
$$h_1=(x_1+P_1,x_2,x_3),\ h_3=(x_1+P_3,x_2,x_3),\ h_5=(x_1+P_5,x_2,x_3).$$
Here $c,c',c''\neq 0$ and $P_i(x_2,x_3)$ are homogeneous polynomials of weighted degree~$m$ with $x_2,x_3$ of weights $p,1$. By
\cite{LP2}*{lem~4.11(ii,iii)}, the elements $\sim_\nu \id$ of $M_\alpha$ are exactly the elements of $M_\alpha$ of form
$(x_1+P(x_2,x_3),x_2+Q(x_3),x_3)$, where $-\nu(P)<m, \deg Q=-\nu(Q)<p$, and they form a normal subgroup $N_\alpha<M_\alpha$. Consequently, we have
$h_0\cdots h_5\sim_\nu \id$ and so $h_0\cdots h_5$ has the above form. Since $P_i$ are homogeneous of degree $m$, $P$ is also homogeneous of degree
$m$ and so we have $c''=c'-c,P=0$ and
$$P_1(x_2+c'x_3^p,x_3)+P_3(x_2+cx_3^p,x_3)+P_5(x_2,x_3)=0.$$
By \cite{LP2}*{lem~4.12(i)}, we have that $P_1,P_3,P_5$ are divisible by $x_2^a,x_2^{a'},x_2^{a''}$, respectively. Thus we can apply
Lemma~\ref{lem:weighted} with $r=a'',t=a',u=a$, and set $\delta=\frac{3}{4}$.

For $\frac{m}{p}<2,$ if one of $a,a',a''$ equals $0$, then $\delta=\frac{3}{4}$ works as well. Otherwise, we have $a=a'=a''=1$, which by Remark~\ref{rem:alpha1} brings us exactly to Example~\ref{exa:1}(b).

Since $\rho_+(\gamma)$ is disjoint from $q$, it passes through $e^s$ an even number of times. Thus it remains to consider the case where $\gamma$
contains $e^s_f,e^s_g$, but not $e^s_h$, as in Figure~\ref{fig:loops}(c). Then we can assume $\gamma_0\cdot \gamma_1^{-1}\subset \Gamma_{f_0},
\gamma_2\cdot \gamma_3^{-1}\subset \Gamma_{f_1}, \gamma_4\subset \Gamma_{f_2},\gamma_5\subset \Gamma_{f_3}$. By \cite{LP2}*{rem~4.10}, points
$\rho_+(\gamma_1\cap \gamma_2), \rho_+(\gamma_3\cap \gamma_4), \rho_+(\gamma_5\cap \gamma_0)$ correspond to the directions towards
$[a,1,0],[m-pa',0,1],[m-pa'',0,1]$ for some $0\leq a,a',a''\leq \big\lfloor \frac{m}{p} \big\rfloor$. By Lemma~\ref{lem:isometry}, for
$\frac{m}{p}\geq 2$ we will find appropriate $\eps_0$ as long as there exists $\delta<1$ independent of $m$ such that $\min
\{\frac{m}{p}-a,a',a''\}\leq \delta \frac{m}{p}$.

By \cite{LP2}*{lem~4.12}, we have $f_0^{-1}f_1\sim_\nu h_0,\ldots, f_3^{-1}f_0\sim_\nu h_3$, where
$$h_0=(x_1+P_0,x_2+c'x_3^p,x_3),\ h_2=(x_1,x_2+cx_3^p,x_3),$$
$$h_1=(x_1+P_1,x_2,x_3),\ h_3=(x_1+P_3,x_2,x_3),$$
where $c,c'\neq 0$ and $P_i(x_2,x_3)$ are homogeneous polynomials of weighted degree $m$ with $x_2,x_3$ of weights $p,1$. The identity $h_0\cdots
h_3\sim_\nu \id$ yields $c'=-c$ and
$$P_0(x_2+cx_3^p,x_3)+P_1(x_2+cx_3^p,x_3)+P_3(x_2,x_3)=0.$$
By \cite{LP2}*{lem~4.12(i,iii)} we have that $P_0,P_1,P_3$ are divisible by $x_3^{m-pa},x_2^{a'},x_2^{a''}$, respectively. Thus we can apply
Lemma~\ref{lem:weightedplus} with $r=a'', t=a', u=m-pa$, and set $\delta=\frac{4}{5}$.

For $\frac{m}{p}<2,$ if one of $a',a''$ equals $0$ or $a=1$, then $\delta=\frac{4}{5}$ works as well. Otherwise, we have $a'=a''=1$, and $a=0$ which
by Remark~\ref{rem:alpha1plus} brings us exactly to Example~\ref{exa:1}(c).
\end{proof}

\section{Diagrams}
\label{sec:diagrams}

In the entire section, \textbf{let $X$ be a simply connected combinatorial $2$-complex} (see Section~\ref{subs:comb} for the definition). A
\emph{disc} (resp.\ \emph{half-plane}) \emph{diagram} $D$ is a combinatorial complex homeomorphic to a $2$-disc (resp.\ a half-plane). A
combinatorial map $\phi\colon D\to X$ from a disc (resp.\ half-plane) diagram to $X$ is a \emph{disc} (resp.\ \emph{half-plane}) \emph{diagram in
$X$}. We say that $\phi$ is \emph{reduced} if it is a local embedding at open edges. The restriction of $\phi$ to $\partial D$ is the \emph{boundary}
of~$\phi$.

\begin{rem}
\label{rem:vk}
Given an embedded combinatorial closed path $\gamma$ in $X$, there exists a reduced disc diagram
$\phi\colon D\to X$ with boundary~$\gamma$. 
This is essentially van Kampen's lemma \cite{vK}, see \cite{JS}*{Lem~1.6} for a complete proof for the case of $X$ a simplicial complex, which implies the case of arbitrary $X$.
\end{rem}

\subsection{Relative disc diagrams}

Often, we will consider the following diagrams that are in general not combinatorial.

\begin{defin}
\label{def:frilled} Let $D$ be a disc diagram with a subcomplex $I\subseteq \partial D$. 
A map $\phi\colon
D\to X$ is a \emph{disc diagram in $X$ relative to~$I$} (or, simply,  a \emph{relative disc diagram in~$X$}), if:
\begin{itemize}
\item the restriction of $\phi$ to any open cell~$\sigma_D$ of $D$ is an embedding into an open cell
$\sigma_X$ of $X$,
\item this restriction is a homeomorphism onto~$\sigma_X$ if the closure $\overline \sigma_D$ is disjoint
from~$I$,
\item
if $\sigma_D$ is an edge not contained in~$I$, then $\sigma_X$ is an edge.
\end{itemize}
Analogously we define a \emph{half-plane diagram in~$X$ relative to $I=\partial D$}, or, simply, a \emph{relative half-plane diagram in~$X$}. Again,
a relative disc or half-plane diagram $\phi$ is \emph{reduced} if it is is a local embedding at open edges.
\end{defin}

\begin{lem}
\label{lem:discexists} Let $\gamma$ be an embedded closed path in $X$, intersecting $X^1$, and that can be subdivided into finitely many segments,
the interior of each of which lies in a single open cell of~$X$. Then there exists a reduced relative (to $\partial D$) disc diagram $\phi\colon D\to
X$ with boundary~$\gamma$.
\end{lem}

In the proof we will need:

\begin{rem}
\label{rem:cells} Let $\sigma$ be an open $2$-disc, and let $g\colon \sigma\to \sigma$ be a local embedding (hence a local homeomorphism) that is
proper, i.e.\ that extends to a continuous map on the one-point compactification ${S}^2$ of $\sigma$ mapping the point ${S}^2\setminus \sigma$ to
itself. Then $g$ is a covering map, and hence a homeomorphism. This implies the following.

Let $\overline{\sigma}=\sigma\cup \partial \sigma$ be a closed $2$-disc, with $\partial \sigma$ a concatenation of paths $\beta_1\cdots \beta_{k}$.
Let $g'\colon \sigma\to \sigma$ be a local embedding that extends to a continuous map $\overline g'$ from $\overline \sigma$ to ${S}^2$ mapping only the endpoints
of $\beta_j$ to ${S}^2\setminus \sigma$. Suppose that $\overline g'$ is a local embedding at each interior point of each $\beta_j$ and that the restriction of $\overline g'$ to the interior of each~$\beta_j$ is an embedding. Then the map $g'$ is an embedding. Indeed, we can extend $g'$ to~$g$ above by attaching to $\sigma$ additional discs
along $\beta_{j}$ and mapping them under $g$ to the discs in $S^2$ bounded by $\overline g'_{|\beta_{j}}$.
\end{rem}

\begin{proof}[Proof of Lemma~\ref{lem:discexists}]
We subdivide $X$ along $\gamma$. The vertices of the new combinatorial structure on $X$ are contained in the original $1$-skeleton. By Remark~\ref{rem:vk}, there is a reduced disc diagram $\phi\colon D\to X$ with boundary $\gamma$, with respect to the new combinatorial structure. Assume that $D$ has the
minimal number of $2$-cells among all such disc diagrams.

We will now discuss the original combinatorial structure on $X$. First, we remove from the combinatorial structure of $D$ all the edges $e$ of $D$
outside $\partial D$ with $\phi(e)\not\subset X^1$.

Second, let $v$ be an interior vertex of~$D$ with $\phi(v)\notin X^0$, hence $\phi(v)\in X^1\setminus X^0$.
After possibly cutting and reglueing
along distinct edges starting at $v$ with the same image under $\phi$ (after which $\phi$ is still reduced by the minimality assumption) we can
assume that $\phi$ a local embedding at $v$. Consequently, there are exactly two edges of~$D$ containing $v$, and they combine to an edge embedding
under $\phi$ in an edge of $X$. We then remove $v$ from the combinatorial structure of $D$. We repeat this procedure for all such $v$.

The map $\phi$ is a local embedding on the resulting cells of $D$. Note that each cell of~$D$ is simply connected since otherwise it would have
$\partial D$ as a boundary component, which would contradict the assumption that $\gamma$ intersected $X^1$. By Remark~\ref{rem:cells}, we have that
$\phi$ is a disc diagram relative to $\partial D$.
\end{proof}

Note that in this article we will only be applying Lemma~\ref{lem:discexists} to $\gamma$ satisfying Theorem~\ref{thm:disc_embeds},
where the cutting and reglueing procedure above is not necessary.

\begin{defin}
\label{def:Euclideanmetric} Suppose that $X$ has a \emph{piecewise smooth
Euclidean metric}. This means that each $2$-cell of $X$ is equipped with the Riemannian metric of a subset of the Euclidean plane bounded by a
piecewise smooth closed path, with nonzero interior angles, and that these Riemannian metrics agree on commmon edges. We also assume that for each
vertex~$v$ of~$X$, all the cells of $X$ containing $v$ have only finitely many possible isometry types, and so in particular $X$ is complete (see
e.g.\ \cite{NOP}*{\S1}).

We define the links of the vertices in $X$ as in Section~\ref{sec:links}. Occasionally, we will discuss the links of points that are not vertices. The link of a point $x$ in the interior of an edge $e$ has two vertices corresponding to the two components  of $e\setminus x$,  and edges of length $\pi$ corresponding to the $2$-cells containing~$e$. The link of a point in the interior of a $2$-cell is a circle of length $2\pi$.

Let $\phi\colon D\to X$ be a reduced disc or half-plane diagram relative to $I$. If the restriction of $\phi$ to each edge $\sigma_D\subset I$ with $\sigma_X$ a $2$-cell is a geodesic in~$\sigma_X$, then we can pull back the
piecewise smooth Euclidean metric from~$X$ to a \emph{degenerate} piecewise smooth Euclidean metric on $D$. This means that we allow the angles
in~$D$ to be zero.

Furthermore, if $X$ is $\mathrm{CAT}(0)$, then by \cite{NOP}*{\S2}, we have that $D$ satisfies the local conditions of~\cite{BaBu}*{Thm~7.1} implying
$\mathrm{CAT}(0)$ in the non-degenerate case. Since the only angles in $D$ that are possibly equal zero are formed by pairs of edges one of which
lies in $\partial D$ and has zero geodesic curvature, it is easy to find an embedding of $D$ as a convex subspace of $D$ with a (non-degenerate)
piecewise smooth Euclidean metric also satisfying the conditions of~\cite{BaBu}*{Thm~7.1}. Consequently, $D$ is $\mathrm{CAT}(0)$ with both metrics.

We define the link of a vertex of such $D$ as in the non-degenerate case, but we do not include the edges corresponding to the zero angles and we do not include their vertices corresponding  to the edges in $\partial D$.

We say that a disc or half-plane diagram $D$ is \emph{locally Euclidean} if each interior point of $D$ has a neighbourhood isometric to a Euclidean
disc.
\end{defin}

\begin{rem}
\label{rem:isometric} Suppose that $X$ as in Definition~\ref{def:Euclideanmetric} is $\mathrm{CAT}(0)$.
\begin{enumerate}[(i)]
\item If $\phi \colon D\to X$ is a reduced relative disc diagram that is locally Euclidean, then $\phi$ is a local isometric embedding at each
    interior point of $D$ (see e.g.\ the second to last paragraph of the proof of \cite{OP}*{Lem~4.1}). In particular, if $D$ is isometric to a
    convex subset of $\R^2$, then by \cite{BH}*{II.4.14} $\phi$ is an isometric embedding, and so $\phi(D)$ is convex.
\item Let $D^\pm\subset X$ be embedded relative disc diagrams, and let $R$ be a connected component of $D^+\cap D^-$. Suppose that $D^\pm$ are
    isometric to convex subsets of~$\R^2$. Then by~(i) we have that $R$ is convex in each $D^\pm$.
\item Generalising~(i), let $R_1,R_2\subset X$ be embedded relative disc diagrams isometric to convex subsets of $\R^2$. Suppose that for each point $x$ of $R_1\cup R_2$, the set~$\mathcal U$ of the directions at $x$ of the geodesics in $R_1\cup R_2$ has diameter $\leq \pi$ in the path metric induced on $\mathcal U$ from the link of $x$. (We refrain from calling $\mathcal U$ the link of $R_1\cup R_2$, since the piecewise Euclidean metric on $R_1\cup R_2$ might be degenerate and $R_1\cup R_2$ is not a disc diagram.)
We claim that then $R_1\cup R_2$ is locally convex (hence convex) in $X$.

 Indeed, this follows from the fact that
    for $x_i\in R_i\setminus R_{3-i}$, the geodesic between~$x_1$ and $x_2$ in $R_1\cup R_2$, which passes through $x\in R_1\cap R_2$, must be a
    geodesic in $X$, since otherwise the geodesic triangle $x_1xx_2$ would have Alexandrov angle $<\pi$ at $x$. By the diameter hypothesis and \cite{NOP}*{Lem~2.1}, the
    shortest path between the directions corresponding to $xx_1$ and $xx_2$ in the link of $x$ in $X$ would have to be contained in $\mathcal U$, contradicting the assumption that $x_1xx_2$ is a geodesic in $R_1\cup R_2$.
\end{enumerate}
\end{rem}

\subsection{Flat lemma}

A \emph{flat in $X$} is an isometrically embedded Euclidean plane in $X$.

\begin{lem}
\label{lem:circle} Suppose that $X$ as in Definition~\ref{def:Euclideanmetric} is $\mathrm{CAT}(0)$. Then for each circle~$\omega$ of length $2\pi$ isometrically embedded in $\partial_\infty X$ with the angle metric, there is a flat $F\subset X$ with
$\partial_\infty F=\omega$.
\end{lem}

Note that for $X$ proper, Lemma~\ref{lem:circle}
is \cite{Leeb}*{Prop~2.1}.

\begin{proof}
Choose a basepoint $x_0\in X$ and seven geodesic rays $r_0, r_1,\ldots, r_6\colon [0,\infty)\to X$ starting at $x_0$ representing consecutive points
on $\omega$ at distance $\frac{2\pi}{7}$ in $\partial_\infty X$. Furthermore, for each integer $n\geq 0$, and $i=0,\ldots, 6$, let $\gamma^n_i$ be the geodesic in $X$ between
$r_i(n)$ and $r_{i+1}(n)$, where $r_7=r_0$.

Consider a closed path $\gamma$ of the form $r_{i+1}[n+1,n]\cdot (\gamma_i^n)^{-1}\cdot r_i[n,n+1]\cdot \gamma_i^{n+1}$, for some $i=0,\ldots, 6$,
and $n\geq 0$. Let $B_i^n$ be the closed ball of radius $1+|\gamma^i_n|/2$ 
centred at the midpoint of $\gamma^i_n$. Since $X$ is $\mathrm{CAT}(0)$, we have that $B_i^n$ is convex. Since $B_i^n$
contains $r_{i+1}[n+1,n]\cdot (\gamma_i^n)^{-1}\cdot r_i[n,n+1]$, it contains the entire~$\gamma$. By \cite{OP}*{Rem~2.4}, we have that~$\gamma$ is
contained in the union of finitely many cells of~$X$. Thus there is a union $K_i^n$ of finitely many cells of~$X$ intersecting~$B_i^n$, such that
$\gamma$ is contained in $K_i^n$ and contractible in $K_i^n$.

By \cite{BH}*{II.9.8(4)}, for each $i$ we have $\frac{|\gamma^n_i|}{n}\to 2\sin \frac{\pi}{7}<1$. Consequently, the distance from $B_i^n$ to
$x_0$ becomes arbitrary large as $n\to \infty$. In particular, each vertex of $X$ belongs to only finitely many $K^n_i$ and so the union $K$ of all
$K_i^n$ is proper. Since $X$ is $2$-dimensional, $K$ is locally $\mathrm{CAT}(0)$ by \cite{NOP}*{\S2} and \cite{BaBu}*{Thm~7.1}.

Let $\widetilde{K}\to K$ be the universal cover of $K$, which is a complete proper $\mathrm{CAT}(0)$ space. Let $\widetilde x_0\in \widetilde{K}$ be
a lift of $x_0$, and for $i=0,\ldots, 6,$ let $\widetilde r_i$ be the lift of $r_i$ to $\widetilde{K}$ starting at $\widetilde x_0$. Since we
included $K_i^n$ in $K$, the geodesics $\gamma_i^n$ lift as well to geodesics $ \widetilde \gamma_i^n$ connecting $\widetilde r_i(n)$ and $\widetilde
r_{i+1}(n)$. Thus by \cite{BH}*{II.9.8(4)} we obtain that consecutive $\widetilde r_i$ represent points $[\widetilde r_i]\in \partial_\infty
\widetilde{K}$ at distance $\frac{2\pi}{7}$. Note that since $\widetilde{K}\to K\to X$ does not increase distances, we have that the distance between
each $[\widetilde r_i]$ and $[\widetilde r_{i+2}]$ is $\frac{4\pi}{7}$. Consequently, the geodesics joining consecutive~$[\widetilde r_i]$
in~$\partial_\infty \widetilde{K}$ form a locally (hence globally, since $\partial_\infty  \widetilde{K}$ is $\mathrm{CAT}(1)$) isometrically
embedded circle $\omega_K\subset \partial_\infty \widetilde{K}$ of length $2\pi$. By \cite{Leeb}*{Prop~2.1}, there is a flat $F\subset \widetilde{K}$
with boundary~$\omega_K$. By Remark~\ref{rem:isometric}(i), $F\to\widetilde{K}\to K\to X$ is an isometric embedding. The boundary of the image of $F$
in~$X$ contains $[r_i]$, and thus coincides with~$\omega$.
\end{proof}

\subsection{Overlaps and unions}
\label{subs:unions} Let $\phi^\pm \colon D^\pm\to X$ be two embedded (and hence reduced) relative disc diagrams. Suppose that we have a connected
subset $\beta\subset \phi^+(D^+)\cap \phi^-(D^-)$. The \emph{overlap of $\phi^\pm$ at~$\beta$} is the connected component $R$ of $\phi^+(D^+)\cap \phi^-(D^-)$
containing $\beta$. Denote $R^\pm=(\phi^\pm)^{-1}(R)\subseteq D^\pm$.

Let $\widehat D$ be the space obtained from the disjoint union of $\overline{D^+\setminus R^+}$ and $\overline{D^-\setminus R^-}$ by identifying the
subsets $P^+\subseteq \mathrm{fr}\, R^+$ (the topological boundary of $R^+$ in $D^+$) and $P^-\subseteq \mathrm{fr}\, R^-$ consisting of points with
the same image under $\phi^+$ and~$\phi^-$. In other words, we have $P^\pm=\mathrm{fr}\, R^\pm\cap (\phi^\pm)^{-1}\phi^\mp(\mathrm{fr}\, R^\mp)$.

The \emph{union of $\phi^\pm$ at $\beta$} is the map $\widehat \phi\colon \widehat D\to X$ whose restriction to
$\overline{D^\pm\setminus R^\pm}$ equals~$\phi^\pm$. 
If each pair $(\overline{D^\pm\setminus R^\pm},P^\pm)$ is \emph{standard}, i.e.\ homeomorphic to the pair consisting of the closed unit disc and the
closed unit half-circle, then $\widehat D$ is homeomorphic to a closed disc and $\widehat \phi$ is a local embedding but possibly not an embedding.
See Figures~\ref{fig:cutR1} and~\ref{fig:cutRnew}.

\begin{figure}
\includegraphics{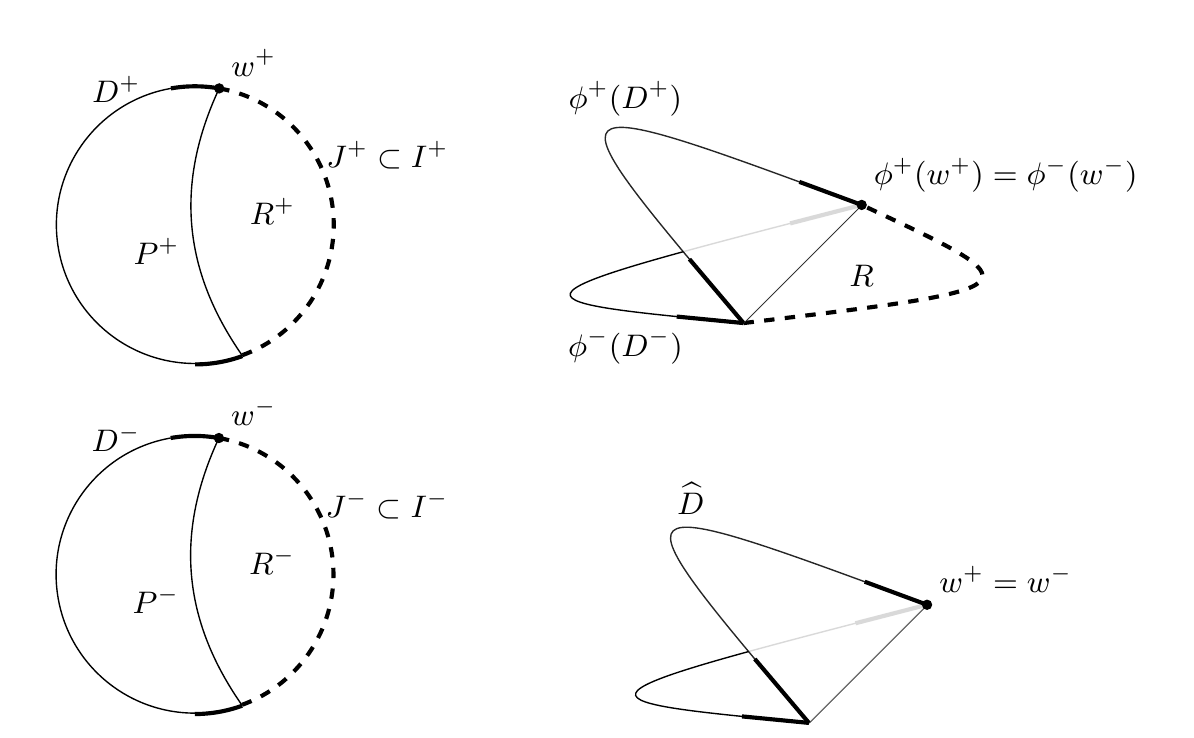}
\caption{A first possibility for $\phi^+(D^+),\phi^-(D^-)$, and $\widehat D$, where $P^\pm=\mathrm{fr}\, R^\pm$. Examples of the sets $I^\pm$ discussed in Lemma~\ref{lem:union} are shown in bold.}
\label{fig:cutR1}
\end{figure}

Note that the following lemma is trivial for $\phi^\pm$ combinatorial.

\begin{lemma}
\label{lem:union} Let $\phi^\pm \colon D^\pm\to X$ be embedded disc diagrams relative to $I^\pm\subset \partial D^\pm$. Let $\beta$ be a connected component of
$\phi^+(\partial D^+)\cap \phi^-(\partial D^-)$ with $J^\pm=(\phi^\pm)^{-1}(\beta)\subseteq I^\pm$. Let $\widehat \phi\colon \widehat D\to X$ be the union of~$\phi^\pm$ at $\beta$. Suppose that each pair $(\overline{D^\pm\setminus R^\pm},P^\pm)$ is
standard, that $\widehat \phi(\partial \widehat D)$ intersects $X^1$, and that the following condition $(*)$ holds for $w^\pm$ the first or the last
points of $P^\pm$:

We have $\phi^+(w^+)=\phi^-(w^-)\in X^1$, or, for the maximal subpaths $f^\pm$ in $\partial (D^\pm\setminus R^\pm)\setminus P^\pm$ incident
to~$w^\pm$ mapped under $\phi^\pm$ to the same open $2$-cell of $X$ as $w^\pm$, the path $\phi^+(f^+)\cdot \phi^-(f^-)$ is embedded in $X$.

Then $\widehat D$ can be given a combinatorial structure such that $\widehat \phi$ is a disc diagram relative to the closure $\widehat I$ of the
union of $I^\pm\setminus J^\pm$.
\end{lemma}

\begin{figure}
\includegraphics{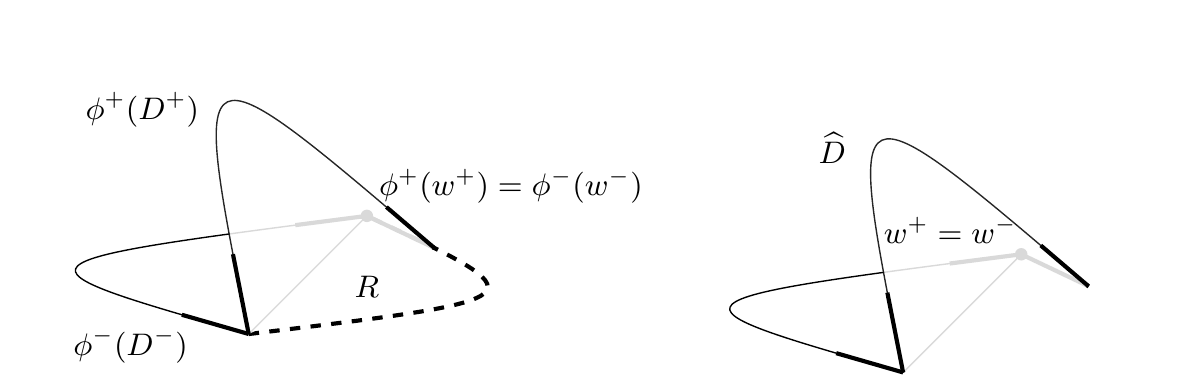}
\caption{A second possibility for $\phi^+(D^+),\phi^-(D^-)$, and $\hat D$, where $P^+\subsetneq\mathrm{fr}\, R^+$. Here a subset of $I^-$ becomes contained in $\overline{D^+\setminus R^+}\subset \widehat D$.}
\label{fig:cutRnew}
\end{figure}

\begin{proof}
We start with the combinatorial structure on each $D^\pm$ subdivided along $\mathrm{fr}\, R^\pm\subset D^\pm$, and at the endpoints of $P^\pm$. Let
$\widehat D_0$ be the disc diagram obtained from glueing $D^\pm$ along $J^\pm$, and let $\widehat \phi_0\colon \widehat D_0\to X$ be the map with
restrictions $\phi^\pm$ to~$D^\pm$. Note that $\partial \widehat D_0$ is the closure of the union of $\partial D^\pm\setminus J^\pm$ and contains~$\widehat I$.
Furthermore, $\widehat D$ and $\widehat\phi$ are obtained from $\widehat D_0$ and $\widehat\phi_0$ by cancelling the $2$-cells of~$R^\pm$, which
gives a natural embedding $\widehat I\subseteq \partial\widehat D$. Let $\widehat J$ be the subset of $J^\pm$ remaining in~$\widehat D$.

Note that the restrictions of $\phi^\pm$ to $\overline{D^\pm\setminus R^\pm}$, which are also restrictions of $\widehat \phi$, are disc diagrams in
$X$ relative to the intersection of their boundary with $\widehat J\cup \widehat I$. 
This follows from the fact that the remaining part of $\partial \overline{D^\pm\setminus R^\pm}$ is mapped under~$\phi^\pm$ into~$X^1$ by the
definition of $R^\pm$. However, $\widehat \phi$ might no longer be a relative disc diagram since the cells of $\widehat J$ become interior in
$\widehat D$. To correct this, we remove from the combinatorial structure of $\widehat D$ the edges and the interior vertices contained in $\widehat J$ that are not mapped
under $\widehat \phi$ into $X^1$.

After this correction, the restriction $g'$ of the map $\widehat \phi$ to each open $2$-cell $\sigma_{\widehat D}$ of~$\widehat D$ has still image in
a single open cell of~$X$. Note that $\sigma_{\widehat D}$ is simply connected since otherwise it would have $\partial \widehat D$ as a boundary component,
which would contradict the assumption that $\widehat \phi(\partial \widehat D)$ intersected $X^1$. Since $\phi^\pm$ were relative disc diagrams, by
the definition of $R^\pm$ we have that $g'$ is a local embedding. Since $\phi^\pm$ were embeddings, we can apply Remark~\ref{rem:cells} to $g'$,
possibly using condition $(*)$ if one of $\beta_{j}=f^+\cdot f^-$. 
This implies that $g'$ is an embedding, and even a
homeomorphism for $\overline \sigma_{\widehat D}$ disjoint from $\widehat I$.

Finally, we remove from the combinatorial structure of $\widehat D$ the interior vertices that are not mapped under $\widehat \phi$ into
$X^0$, and possibly the endpoints of $P^\pm$ if they do not belong to $\widehat I$ and are not mapped under $\widehat \phi$ into $X^0$. Then $\widehat \phi$ becomes a disc diagram
relative to $\widehat I$.
\end{proof}

\subsection{Curvature}

Let $D$ be a disc diagram with a degenerate piecewise smooth Euclidean metric. For a vertex $v\in D$, the \emph{angle} at $v$ is the sum of the
Riemannian angles at $v$ of all $2$-cells of $D$ containing $v$. We analogously define \emph{angles} at $v$ between a pair of edges containing $v$.
If $v\in \int \, D$, then its \emph{curvature} $\kappa(v)$ is $2\pi$ minus its angle. If $v\in
\partial D$, then its \emph{boundary curvature} $\kappa_\partial(v)$ is $\pi$ minus its angle. If $e$ is an edge of $D$, then its \emph{curvature}
$\kappa(e)$ is the integral of its geodesic curvature in the (one or two) $2$-cells of $D$ containing it. In particular, if $e$ is a Riemannian
geodesic in the $2$-cells containing it, then $\kappa(e)=0$. From adding up the classical Gauss--Bonnet formula for the $2$-cells of $D$, we obtain
the following (see e.g.\ \cite{BaBu}*{\S2.3}).

\begin{thm}
\label{thm:GB} Let $D$ be a disc diagram. Then
$$\sum_{v\in \int D}\kappa(v)+ \sum_{v\in \partial D}\kappa_\partial(v)+\sum_{e\subset D}\kappa(e)=2\pi.$$
\end{thm}


We have the following consequence of \cite{Stad}*{Thm~2}. For the reader's convenience, we include the proof in our $2$-dimensional setup. A similar
proof was earlier provided in \cite{Bader}*{Thm~12.1} in the case where each edge $e$ of $X$ is a Riemannian geodesic in the $2$-cells
containing it.

\begin{thm}
\label{thm:disc_embeds} Suppose that $X$ is $\mathrm{CAT}(0)$. Let $\gamma\colon S^1\to X$ be an embedded closed path consisting of $3$
or $4$ geodesic segments. Then each reduced relative disc diagram $\phi\colon D\to X$ with boundary $\gamma$ is embedded.
\end{thm}

\begin{proof}
\noindent \textbf{Step~$1$.} $\phi$ is a local embedding.

\smallskip

Let $V\subset D$ denote the set of the (at most $4$) endpoints of the geodesic segments in~$\partial D$. Since $X$ is $\mathrm{CAT}(0)$, each vertex
$v\in \int \,D$ and each edge $e\not\subset \partial D$ has nonpositive curvature. Since $\gamma$ is locally geodesic outside $V$, each $e\subset
\partial D$ or $v\in
\partial D\setminus V$ has also nonpositive (boundary) curvature. Finally, each vertex of $V$ has curvature $\leq \pi$. (In fact, it is $<\pi$ or one
of the edges containing this vertex has negative curvature.)
Let $\mathcal K$ denote the sum of all
these curvatures.

Let $e\neq f$ be oriented edges starting at a vertex $v$ of $D$, and suppose $\phi(e)=\phi(f)$. If $v\in \int\, D$, then, since $X$ is
$\mathrm{CAT}(0)$, both angles at $v$ between $e$ and $f$ are $\geq 2\pi$, so the angle at $v$ is $\geq 4\pi$. Thus $\kappa(v)\leq -2\pi$, and so
$\mathcal K< 4\pi-2\pi$, which contradicts Theorem \ref{thm:GB}. If $v\in \partial D\setminus V$, then let $\theta_1+\theta_2+\theta_3$ be the angle
at $v$, decomposed to indicate the angles at $v$ between $\partial D, e,f,$ and $\partial D$ again. Since $X$ is $\mathrm{CAT}(0)$, we have
$\theta_2\geq 2\pi$. Since $v\in
\partial D\setminus V$, we have $\theta_1+\theta_3\geq \pi$. Consequently, $\kappa_\partial(v)\leq -2\pi$, which is again a contradiction. Finally, if
$v\in V$, then $\kappa_\partial(v)\leq -\pi$, and so $\mathcal K<3\pi - \pi$, contradiction.

\smallskip

\noindent \textbf{Step~$2$.} $\phi$ is an embedding.

\smallskip

Suppose that $\beta\subset D$ is an embedded combinatorial path from $w_1$ to $w_2$ such that the restriction of $\phi$ to $\beta$ is an embedding
except that $\phi(w_1)=\phi(w_2)$. Let $\psi \colon D'\to X$ be a reduced relative disc diagram with boundary $\phi_{|\beta}$, guaranteed by Lemma~\ref{lem:discexists}. Suppose that among
such $\beta$ and $\psi$, the disc diagram $D'$ has the minimal possible number of $2$-cells.

Let $v'\in \partial D'\setminus \psi^{-1}\phi(w_1)$ be a vertex with $\kappa_\partial(v')>0$ and let $e',f'$ be the edges of~$\partial D'$ containing
$v'$. Let $e,v,f\subset \beta$ be such that $\phi(e)=\psi(e'),\phi(v)=\psi(v'),\phi(f)=\psi(f')$. By the minimality assumption, the images of the
restrictions of $\phi$ and $\psi$ to the open neighbourhoods of $e,e'$ in any $2$-cells containing $e,e'$, are distinct, and the same property holds
for $f,f'$. Thus, if $v\in \int\, D$, then since $X$ is $\mathrm{CAT}(0)$, both angles in~$D$ between $e$ and $f$ are $\geq \pi+\kappa_\partial(v')$,
and so $\kappa(v)\leq -2\kappa_\partial(v')$. If $v\in
\partial D\setminus V$, then $\theta_2\geq \pi+\kappa_\partial(v')$ and $\theta_1+\theta_3\geq \kappa_\partial(v')$. Consequently, $\kappa_\partial(v)\leq
-2\kappa_\partial(v')$. Finally, if $v\in V$, then $\kappa_\partial(v)\leq -\kappa_\partial(v')\leq\pi-2\kappa_\partial(v')$.

Let now $e'\subset \partial D'$ be an edge with $\kappa(e')>0$, and let $e\subset \beta$ be such that $\phi(e)=\psi(e')$. Then
$\phi(e)$ is not a geodesic and so $e\not\subset
\partial D$. Moreover, since $X$ is $\mathrm{CAT}(0)$, we have $\kappa(e)\leq -2\kappa(e')$.

By Theorem \ref{thm:GB} applied to $D',$ the sum of all the curvatures of edges $e'\subset \partial D'$ and vertices $v'\in \partial D'\setminus
\psi^{-1}\phi(w_1)$ is at least $\pi$. By the previous paragraph, each positive contribution $\kappa_\partial(v')$ or $\kappa(e')$ to that sum
decreases the maximal possible curvature of $v,e$ in $D$ by at least $2\kappa_\partial(v')$ or $2\kappa(e')$. Thus $\mathcal K<4\pi-2\pi$, which is a
contradiction.
\end{proof}

\subsection{Frillings}

The following operation will be needed to discuss the limits of relative disc diagrams.

\begin{defin}
\label{def:frilling} Let $D$ be a half-plane diagram. A \emph{frilling} of $D$ is a combinatorial complex $D^*$ obtained from $D$ via the following
subdivision and quotient map $D\to D^*$. First, we subdivide some $2$-cells of $D$ along new edges (and their vertices) each of which has exactly one
endpoint in $\partial D$. Then, we quotient some of the edges intersecting $\partial D$ to vertices, including all the new edges. Finally, we
collapse some of the $2$-cells.

Here a \emph{collapse} of an open $2$-cell $\sigma\subset D$, such that $\overline {\sigma}\cap \partial D$ is a nontrivial path~$\tau$, is a
deformation retraction of $D$ onto $D\setminus (\sigma\cup \mathrm{int}\ \tau)$.
\end{defin}

\begin{figure}
\includegraphics{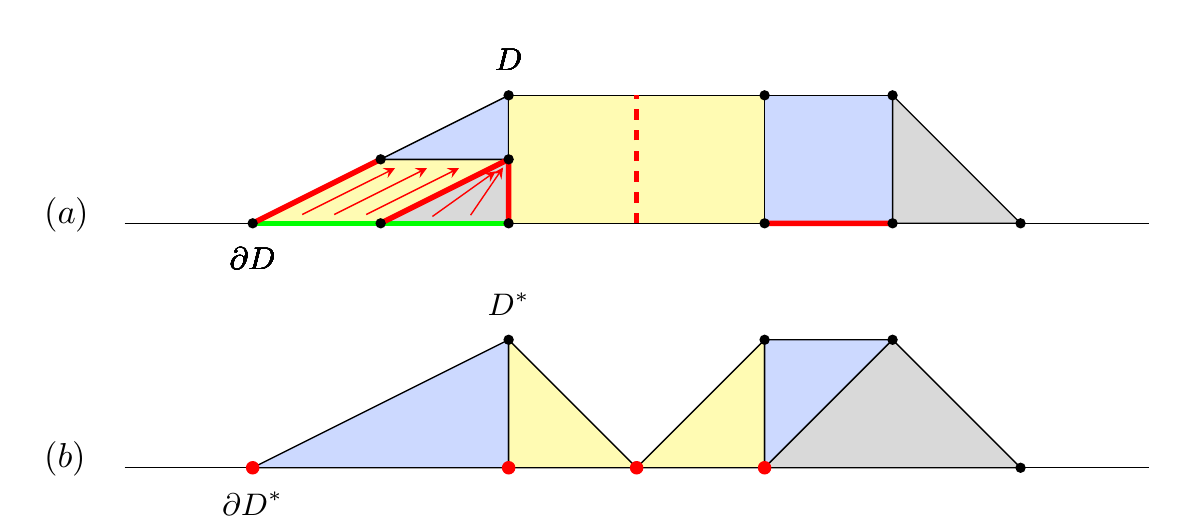}
\caption{An example of a frilling.}
\label{fig:frilling}
\end{figure}

See Figure~\ref{fig:frilling}(a), where we subdivide $D$ along the dashed edge, we quotient this edge along with the other red edges, and then we
collapse the two indicated $2$-cells~$\sigma$, where the paths $\tau$ are indicated in green. The resulting frilling $D^*$ is illustrated in Figure~\ref{fig:frilling}(b).

\begin{rem}
\label{rem:half-plane}  We denote by $\partial D^*$ 
the image of
$\partial D$ in $D^*$. Note that if the preimage in $\partial D$ of each point of $\partial D^*$ is compact, then $D^*$ is a half-plane diagram $D'$ wedged at points
of $\partial D'$ with combinatorial complexes that are contractible (which may exist only if $\partial D^*$ is not an embedded line) or homeomorphic
to $2$-spheres. In our applications in further sections we will always have $D'=D^*$.
\end{rem}

\begin{constr}
\label{constr:conv}
Let $D$ be a half-plane diagram. Suppose that $D$ is an increasing union of disc diagrams $D_n$ with $I_n=D_n\cap \partial D$. We then say that
$(D_n,I_n)$ \emph{converge} to $(D,\partial D)$.

Suppose that additionally we have relative disc diagrams $\phi_n\colon D_n\to X$ satisfying the assumptions (i--iii) below. In the construction below
we define their \emph{limit} $\phi^*\colon D^* \to X$, where $D^*$ is a particular frilling of $D$. The reason for passing to the frilling is that it
will turn out that if the hypothesis of Remark~\ref{rem:half-plane} holds, then the restriction $\phi'$ of $\phi^*\colon D^* \to X$ to $D'$ is a
relative half-plane diagram.

For future reference, whenever we say that $\phi^*\colon D^* \to X$ (or $\phi'\colon D' \to X$) is the \emph{limit} of $\phi_n$, then in particular
we mean that $(D_n,I_n)$ can be identified (by combinatorial isomorphisms) with subcomplexes of a half-plane diagram $D$ converging to $(D,\partial
D)$, that $\phi_n$ satisfy the assumptions (i--iii), and that $D^*$ (or $D'$) is the particular frilling described below.

The assumptions that we make on $\phi_n$ are the following.
\begin{enumerate}[(i)]
\item \label{constr:conv:i}
The attaching map of each cell in $X$ is injective, and the boundary of each~$\phi_n$ intersects $X^1$ at least twice (which implies
    that the attaching map of each cell in $D_n$ is injective, since otherwise the boundary of a $2$-cell would contain $\partial D$ and would map into $X\setminus X^1$ except for one point).
\item \label{constr:conv:ii}
For each cell $\sigma_D$ in $D$, the cell $\sigma_X$ in Definition~\ref{def:frilled} is the same for all $\phi_n$ such that $\sigma_D$ lies
    in~$D_n$.
\item \label{constr:conv:iii}
For each $2$-cell $\sigma_D$ whose closure intersects $\partial D$, the
closed paths $\phi_{n|\partial \sigma_D}$ converge to a
    closed path $\phi_{\sigma_D}\colon \partial \sigma_D\to \overline\sigma_X$. Moreorover, for $\phi^*_{\sigma_D}\colon \partial\sigma_D^*\to
    \overline\sigma_X$, obtained from $\phi_{\sigma_D}$ by quotienting to points the edges of $\partial \sigma_D$ on which $\phi_{\sigma_D}$ is
    constant, we have that
    \begin{itemize}
\item the restriction of $\phi^*_{\sigma_D}$ to the subset of $\partial\sigma_D^*$ that is the image of $\partial \sigma_D\cap \partial D$ is
    an embedding, and
\item $\phi^*_{\sigma_D}(\partial\sigma_D^*)$ bounds a finite (possibly empty) union of open discs in $\sigma_X$.
\end{itemize}
\end{enumerate}

Then the \emph{limit} $\phi^*\colon D^* \to X$ is the following map whose domain $D^*$ is the following frilling $D^*$ of $D$. To start, on each cell
$\sigma_D$ of $D$ disjoint from $\partial D$ (which is thus a cell of~$D^*$) we define~$\phi^*$ to be a homeomorphism onto $\sigma_X$. On each
remaining edge~$\sigma_D$ of $D$ not contained in~$\partial D$, we define~$\phi^*$ to be a homeomorphism onto the subset of the edge~$\sigma_X$
bounded by the limits of $\phi_n(\partial \sigma_D)$. In the case where these two limits coincide, we quotient~$\sigma_D$ to a vertex in $D^*$. 

Consider now a $2$-cell $\sigma_D$ of $D$ whose closure intersects $\partial D$. If $\phi^*_{\sigma_D}$ is an embedding on the entire
$\partial\sigma_D^*$, then we define~$\phi^*$ on~$\sigma_D$ to be a homeomorphism onto the subset of the $2$-cell~$\sigma_X$ bounded by
$\phi^*_{\sigma_D}(\partial\sigma_D^*)$. Then the $2$-cell $\sigma_D$ remains a $2$-cell of~$D^*$, and we 
identify $\partial\sigma_D^*$ with its
boundary, and $\phi^*_{\sigma_D}$ with the restriction of~$\phi^*$ to its boundary. Otherwise, we denote by $\tau_i,\tau_i'$ the maximal subpaths of
$\partial \sigma_D\cap\partial D,\overline{\partial \sigma_D\setminus\partial D}$ with the same image under $\phi_{\sigma_D}$, for $\tau_i$ not an endpoint of a
connected component of $\partial \sigma_D\cap\partial D$. We
\begin{itemize}
\item
subdivide the $2$-cell $\sigma_D$ along edges starting (resp.\ ending) at an endpoint of~$\tau_i$ (resp.\ the corresponding endpoint of~$\tau_i'$),
\item quotient these edges to vertices, and
\item collapse the resulting $2$-cells bounded by nontrivial $\tau_i,\tau'_i$.
\end{itemize}
This allows us to extend $\phi^*$ to an embedding on each of the $2$-cells of $D^*$ resulting from~$\sigma_D$ and on the image in $\partial\sigma_D^*$ of each edge of $\partial \sigma_D\cap \partial D$ that was not removed during a collapse.
Consequently, if the hypothesis of
Remark~\ref{rem:half-plane} holds, then the restriction $\phi'$ of $\phi^*\colon D^* \to X$ to $D'$ is a relative half-plane diagram.
\end{constr}

See Figure~\ref{fig:frilling2}, where on the top we illustrate a part of the complex $X$, and where $\phi_1(D_1)$ is located above the straight line
$\phi_1(I_1)$ in (a). Choosing $\phi_n$ so that $\phi_n(I_n)$ converges to the straight line in (b), we have that $(D_n,I_n)$ converge to
$(D,\partial D)$ from Figure~\ref{fig:frilling}(a), with limit $\phi^*=\phi'$ whose domain is the frilling~$D^*$ from Figure~\ref{fig:frilling}(b).

\begin{figure}
\includegraphics{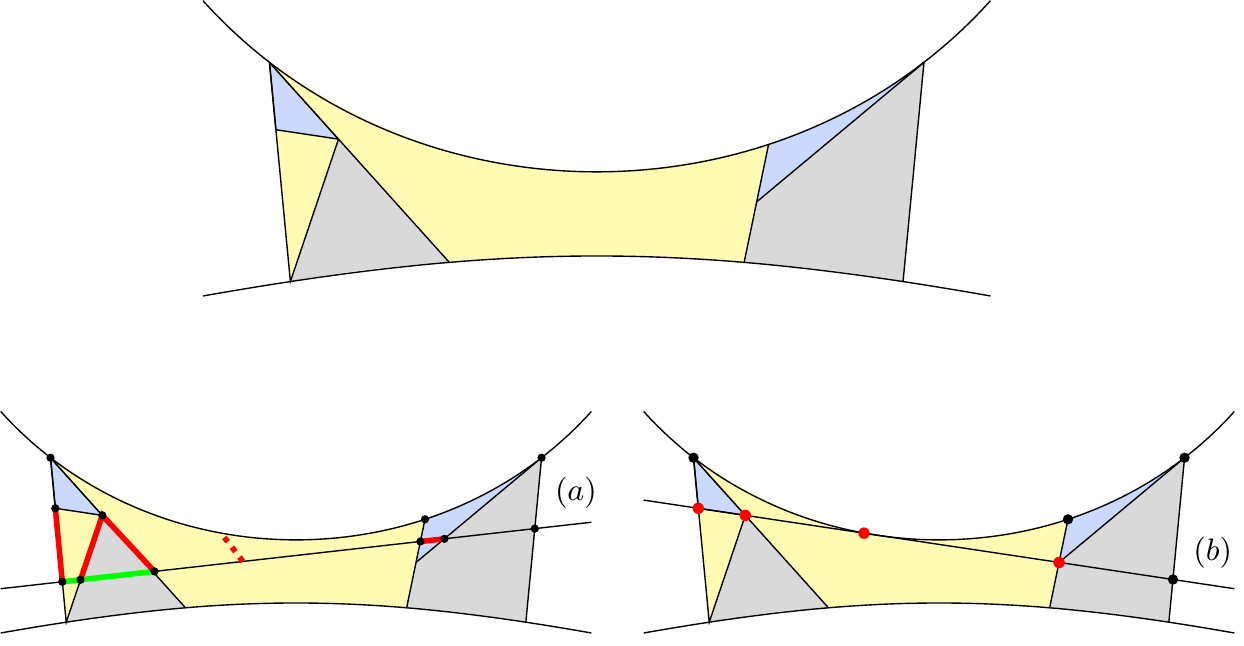}
\caption{}
\label{fig:frilling2}
\end{figure}

\section{Diagrams in \texorpdfstring{$\X$}{X}}
\label{sec:diagramsinX}

In this section we describe the additional structure of the relative disc or half-plane diagrams $D\to X$ for $X=\X$ or $X=\nabla^+$. 

Note that by Lemma~\ref{lem:propertiesX}(iii), each geodesic segment~$\gamma$ in~$\X$ passes through only finitely many chambers. Since admissible
lines form a locally finite set in $(\nabla^+, |\cdot,\cdot|)$ \cite{LP2}*{lem~4.1}, and bound convex half-spaces \cite{LP2}*{lem~5.1}, each geodesic
segment in $\nabla^+$ intersects the union of admissible lines in a union of a finite set of points and a finite set of subsegments. Thus there are
finitely many connected components of intersections of $\gamma$ with the open cells of $\X$, and Lemma~\ref{lem:discexists} can be applied to piecewise
geodesic closed paths in $\X$. Furthermore, the attaching maps of the cells in $\X$ and $\nabla^+$ are injective as required in Construction~\ref{constr:conv}\eqref{constr:conv:i}.

\subsection{Limits in \texorpdfstring{$\nabla^+$}{nabla+}}

\begin{rem}
\label{rem:converge_nabla} Suppose in Construction~\ref{constr:conv} that $X=\nabla^+$ and all $\phi_n$ are the compositions of reduced relative disc
diagrams $\widetilde \phi_n$ in $\X$ with~$\rho_+$. Furthermore, suppose that $\partial D$ is a concatenation $I_a\cdot I_b\cdot I_c$ and we have
$\epsilon>0$ and a positive function $\delta\colon I_a\cup I_c\to \R$ such that
\begin{itemize}
\item the restrictions of $\widetilde \phi_n$ to the intersections of $I_a,I_b,I_c$ with $I_n$ are geodesics concatenated at an Alexandrov angle $\geq
    \epsilon$, and
\item for each $n$, and each $t\in I_{a}\cap I_n$, we have $d_\X \big(\widetilde \phi_n(t),\widetilde \phi_n(I_c\cap I_n)\big)\geq \delta(t)$,
    and the same condition holds if we interchange $I_a$ with $I_c$.
\end{itemize}
Then:

\begin{enumerate}
\item[--] For each $2$-cell $\sigma_D$, the restriction of $\phi_n$ to $\partial \sigma_D\cap \partial D$ converges to a disjoint union of
    finitely many (possibly constant) piecewise geodesics. Thus the maps~$\phi^*_{\sigma_D}$ satisfy the assumption of
    Construction~\ref{constr:conv}\eqref{constr:conv:iii}. Furthermore, the preimage of each point under $\partial D\to \partial D^*$ is connected.

\item[--] All $D_n$ are locally $\mathrm{CAT}(0)$, and so $D^*$ is locally $\mathrm{CAT}(0)$. Thus, by Theorem~\ref{thm:GB}, we have that $D^*$
    contains no subcomplexes homeomorphic to $2$-spheres.

\item[--] Consequently $D'=D^*, \phi'=\phi^*$, and $\partial D'$ is a path. Note that $\phi'_{|\partial D'}$ is obtained from $\phi_{|\partial D}$ by
    quotienting to points the segments on which $\phi_{|\partial D}$ is constant.

\item[--] Since there is a uniform bound on the number of $2$-cells traversed by a geodesic in $\X$ contained in the star of a vertex $v$ of $\X$
    with given $\rho_+(v)\in \nabla^+$ \cite{OP}*{Claim in \S2}, the preimage of each point under $\partial D\to \partial D'$ is compact, and so
    Remark~\ref{rem:half-plane} applies. We conclude that $\partial D'$ is a line and $D'$ is a half-plane diagram.
\end{enumerate}
\end{rem}

\begin{defin}
\label{def:bordered} Let $\phi'\colon D'\to \nabla^+$ be a relative half-plane diagram that is the limit of~$\phi_n$ that are the compositions with
$\rho_+$ of reduced relative disc diagrams $\widetilde \phi_n\colon D_n\to \X$. Let $J\subset\partial D$ be homeomorphic with $[0,\infty)$.

Let $r\colon J\to \X$ be a geodesic ray and suppose that the restrictions of $\widetilde \phi_n$ to $\partial D_n\cap J$ coincide with the
restrictions of~$r$. We then say that $\phi'$ is \emph{bordered} by $r$. This terminology is justified since then
the sequence $\phi_n$ is constant on $J$, and so the map $D\to D'$ is the identity on $J$. We may thus
identify $J$ with a subset of $\partial D'$, and the restriction of $\phi'$ to~$J$ equals $\rho_+\circ r$.

Let $r\colon J\to \X$ be a geodesic ray disjoint from the vertex set $\X^0$ and transverse to the $1$-skeleton $\X^1$. Suppose that the restrictions
of $\widetilde \phi_n$ to $J\cap \partial D_n$ are geodesics that converge to~$r$. We then say that $\phi'$ is \emph{weakly bordered} by $r$. In that
case, we can arrange $\phi'$ to be bordered by $r$ after modifying $\widetilde \phi_n$ on the cells intersecting $J$.
\end{defin}

Below we describe how does the operation of passing to the limit of relative disc diagrams behave with respect to the operation of the union.

\begin{rem}
\label{rem:limit_union} Suppose that $(D^\pm_n,I^\pm_n)$ converge to half-plane diagrams $(D^\pm,\partial D^\pm)$, with frillings $D^{*\pm}=D^{'\pm}$
and $\phi^{'\pm}\colon D^{'\pm}\to \nabla^+$ the limits of the compositions $\phi^\pm_n$ of embedded relative disc diagrams $\widetilde
\phi^\pm_n\colon D_n^\pm\to \X$ with $\rho_+$. Let $J^\pm\subset
\partial D^\pm$ be rays with $\beta_n=\widetilde \phi^+_n(J^+\cap \partial D_n^+)=\widetilde \phi^-_n(J^-\cap \partial D^-_n)$ a connected
component of the intersection of $\widetilde \phi^\pm_n(\partial D^\pm_n)$. Let $\widehat \phi_n\colon \widehat D_n\to \X$ be the unions
of~$\widetilde \phi_n^+$ and~$\widetilde \phi_n^-$ at~$\beta_n$ with overlaps $R_n$.

After passing to subdiagrams and a subsequence, 
for all $n$ we have $(\widetilde \phi^\pm_n)^{-1}(R_n)=R^\pm\cap D^\pm_n$ for some
fixed $R^\pm\subseteq D^\pm$. Let $\widehat D$ be the space obtained from the disjoint union of $\overline{D^\pm\setminus R^\pm}$ by identifying the
subsets $P^\pm\subseteq \mathrm{fr}\, R^\pm$ consisting of points with the same image under $\widetilde \phi^\pm_n$ (this does not depend on $n$).

Suppose that each pair $(\overline{D^\pm\setminus R^\pm}, P^\pm)$ is homeomorphic to the pair consisting of the closed upper half-plane in $\R^2$ and
the negative real axis, so that we can assume that each $(\overline{D_n^\pm\setminus R_n^\pm}, P_n^\pm)$ is standard. Suppose that the condition
$(*)$ of Lemma~\ref{lem:union} holds for $P_n^\pm$, and so $\widehat \phi_n$ are disc diagrams relative to appropriate $\widehat I_n$.

Then $\widehat D$ has a combinatorial structure such that $(\widehat D_n,\widehat I_n)$ converge to $(\widehat D,\partial \widehat D)$. Moreover, we
have a frilling $\widehat D\to \widehat D'$ and a relative half-plane diagram $\widehat \phi'\colon \widehat D'\to\nabla^+$ that is the limit of
$\rho_+\circ\widehat \phi_n$. Note that $\widehat D\to
\widehat D'$ coincides with $D^\pm\to D^{'\pm}$ on the cells contained entirely in $D^+$ or $D^-$, and $\widehat \phi'$ coincides with $\phi^{'\pm}$
on their quotients.

Furthermore, if $\phi'^\pm$ are both bordered by a geodesic ray with domains $J^\pm$, and $R^\pm=J^\pm$, then each cell $\sigma$ of $\widehat D'$  with $\overline \sigma$ disjoint from $\partial
\widehat D'$, considered as a cell of $\widehat D$, is a union of cells~$\sigma_\sharp$ of~$D^\pm$, and $\widehat \phi'_{|\sigma}=\widehat
\phi_{|\sigma}$ is the union of $\phi^{'\pm}_{|\sigma_\sharp}=\phi^{\pm}_{|\sigma_\sharp}$.
\end{rem}

\begin{rem}
\label{rem:bizarreunion} More generally, in Remark~\ref{rem:limit_union} we can suppose that (see Figure~\ref{fig:bizarre}):
\begin{itemize}
\item $P^\pm$ is a (possibly empty or infinite) union of bounded paths $P^{k\pm}$ and an unbounded path $P^{\infty\pm}$, pairwise disjoint except
    possibly at the endpoints, and
\item $\overline{D^-\setminus R^-}$ is a union of disc diagrams $D^{k}$, and a half-plane diagram $D^{\infty}$, pairwise disjoint except possibly
    at the endpoints of $P^{k-}$.
\end{itemize}
We require that
\begin{itemize}
\item each  $(\overline{D^+\setminus R^+}, P^{\infty+}),(D^\infty, P^{\infty-})$ is homeomorphic to the pair consisting of the closed upper
    half-plane in $\R^2$ and the negative real axis, and

\item each $(\overline{D^+\setminus R^+}, P^{k+})$ is homeomorphic to the pair consisting of the closed upper half-plane in $\R^2$ and a bounded
    interval in its boundary, and
\item each $(D^{k}, P^{k-})$ is standard.
\end{itemize}

\begin{figure}
\includegraphics{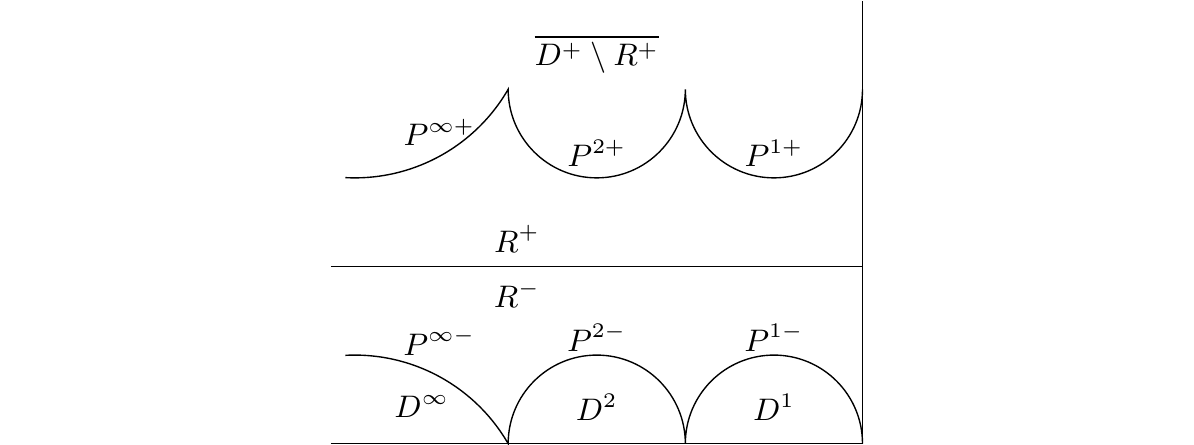}
\caption{}
\label{fig:bizarre}
\end{figure}

We can then assume that $(\overline{D_n^\pm\setminus R_n^\pm}, P_n^\pm)$ have analogous structure, where $\overline{D_n^+\setminus R_n^+}$ and
$D_n^\infty$ are disc diagrams. Suppose that the endpoints of $P^{\pm}_n$ are mapped under~$\phi'^\pm_n$ into the $1$-skeleton of $\nabla^+$. Then
$\widehat \phi_n$ are disc diagrams relative to appropriate $\widehat I_n$ and the last two paragraphs of Remark~\ref{rem:limit_union} hold as well.
\end{rem}

\subsection{Folding locus}

\begin{defin}
\label{def:locus} Let $\phi\colon D\to \nabla^+$ be a relative disc or half-plane diagram. We say that $\phi$ is \emph{$\X$-reduced} if at a
neighbourhood of each vertex of $D$ it is the composition of a reduced relative disc diagram in $\X$ and the map $\rho_+$. 

The \emph{folding locus} $L\subset D$ of $\phi$ is the closure of the union of open edges $e$ in $D\setminus\partial D$ that
\begin{itemize}
\item do not map to $\partial\nabla^+$ under $\phi$, but
\item at which $\phi$ is not a local embedding.
\end{itemize}
If such an edge $e$ is mapped under $\phi$ into a principal line, then we \emph{orient} $e$ in the direction opposite to a vertex of $\nabla$ (see
Figure~\ref{fig:orient}).
\end{defin}

For example, since in Remark~\ref{rem:limit_union} the relative disc diagrams $\widehat \phi_n$ are local embeddings, we have that each $\rho^+\circ \widehat \phi_n$, and consequently $\widehat \phi'$, is $\X$-reduced.

\begin{figure}
\includegraphics{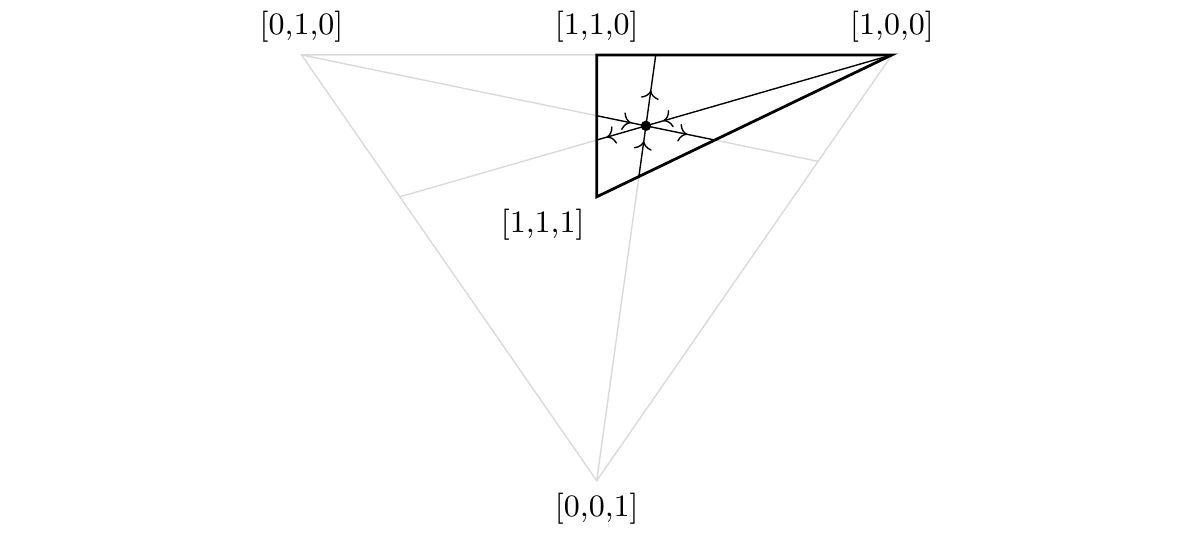}
\caption{}
\label{fig:orient}
\end{figure}

From Propositions~\ref{prop:linkpp1}, \ref{prop:linkp11}, and~\ref{prop:linkmp1} we deduce the following, where $\eps_0$ is the minimum of the
constants in Propositions~\ref{prop:linkp11} and~\ref{prop:linkmp1}.

\begin{cor}
\label{cor:consise} Let $\phi\colon D\to \nabla^+$ be an $\X$-reduced relative disc or half-plane diagram and let $v\in D\setminus\partial D$ be a vertex of
the folding locus $L\subset D$. Suppose that the angle at $v$ in $D$ is $<2\pi+\eps_0$. Then the intersection of $L$ with the star of $v$ in $D$ has
one of the forms in Figure~\ref{fig:concise}.

The angles at $v$ in (a),(b),(e) and the two top angles in (f) are the obvious multiples of $\frac{\pi}{3}$. Both angles in (c) are $>\pi$ and in (d)
they are $\geq \pi$. The two bottom angles in~(f), (g), the two top angles in~(h), as well as all the angles in~(i) are $>\frac{\pi}{3}$.
The two bottom angles in (h) are $>\frac{2\pi}{3}$. The sum of the top and bottom angle at each side of (g) is $\geq \pi$.
\end{cor}

\begin{figure}
\includegraphics{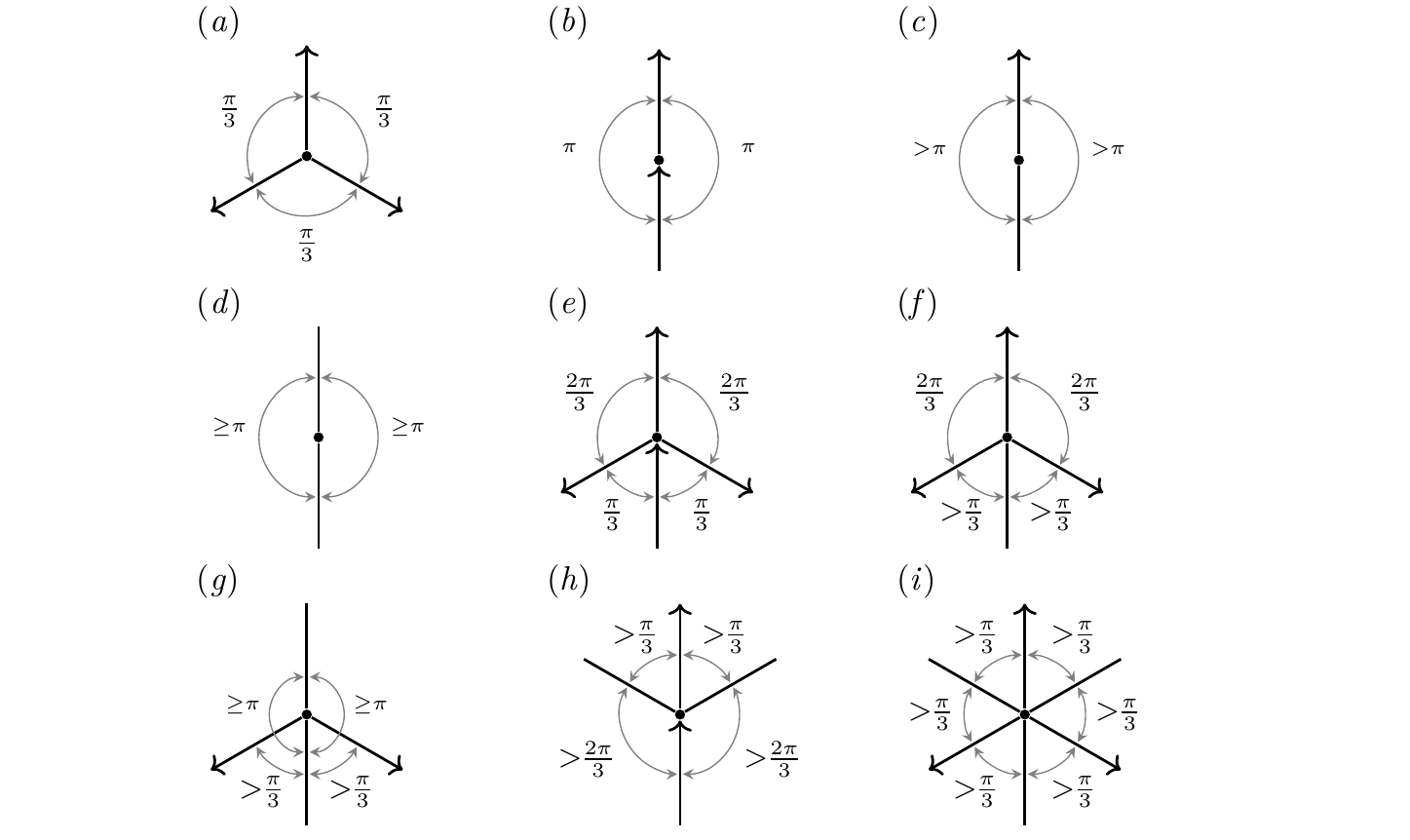}
\caption{}
\label{fig:concise}
\end{figure}

\begin{proof} The third and second bullet point in Proposition~\ref{prop:linkpp1} correspond to~(a) and~(b), respectively. The second bullet point in Propositions~\ref{prop:linkp11}
and~\ref{prop:linkmp1} corresponds to~(b),~(c), and~(d). The third bullet point in Proposition~\ref{prop:linkmp1} corresponds to~(e),~(f), and~(g).
The fifth and fourth bullet points in Proposition~\ref{prop:linkmp1} correspond to~(h),~(i), respectively.
\end{proof}

\begin{rem}\phantomsection
\label{rem:straight}
\begin{enumerate}[(i)]
\item If in Corollary~\ref{cor:consise} we have an edge $e$ of $L$ containing $v$ that is oriented towards $v$ (resp.\ non-oriented but not as in
    Figure~\ref{fig:concise}(h)), then there is an edge $e'$ of $L$ containing~$v$ at angle $\geq \pi$ from $e$ that is oriented from $v$ (resp.\
    or non-oriented).
\item Thus if at each interior vertex of $D$ the angle is $<2\pi+\eps_0$, then each oriented edge of $L$ is contained in a sequence of oriented
    edges of $L$ forming an oriented geodesic ray or segment ending at $\partial D$. Furthermore, each non-oriented edge~$e$ of $L$ is contained
    in a geodesic of $L$ whose each end is unbounded or lies at~$\partial D$ or at $v$ in Figure~\ref{fig:concise}(h).
\item Let $c\subset \nabla^+$ be the (typically non-geodesic) ray that is a parametrisation of an admissible line
    $\alpha_1=m_2\alpha_2+m_3\alpha_3$. Then there is a subray $c'$ of~$c$ such that any other admissible line in $\nabla^+$ with unbounded
    intersection with $\alpha_1\geq m_2\alpha_2+m_3\alpha_3$ is disjoint from $c'$ or has form $\alpha_2=m\alpha_3$.
\item Consequently, if in part (i) in Figure~\ref{fig:concise}(b,c,d) we have $\phi(e)\subset c'$, with $\phi(v)$ separating $\phi(e)$ from the unbounded end of $c'$, then
    $\phi(e')$ is the consecutive edge on $c'$, and we have Figure~\ref{fig:concise}(b) or~(d).
\end{enumerate}
\end{rem}

\begin{defin}
\label{def:u}For any point $[\alpha]\in \nabla^+$, let $S_{[\alpha]}$ denote its unit tangent space in the Euclidean plane $(\nabla,|\cdot,\cdot|)$,
which can be identified with the link of $[\alpha]$ in~$\nabla$. Furthermore, let $S^\angle_{[\alpha]}$ denote the quotient of $S_{[\alpha]}$ by the
action of the order $6$ dihedral group generated by the reflections in the lines through $[\alpha]$ parallel to the principal lines. Parallel
transporting in $(\nabla,|\cdot,\cdot|)$, we can identify all $S_{[\alpha]}$ and all $S^\angle_{[\alpha]}$, which we denote then shortly $S$ and
$S^\angle$. The direction towards (resp.\ away from) the vertices of $\nabla$ determines the \emph{principal point} (resp.\ \emph{antiprincipal
point}) of $S^\angle$. The union of these two points is denoted $\partial S^\angle\subset S^\angle$.

Consider now an $\X$-reduced relative disc or half-plane diagram $\phi\colon D\to \nabla^+\subset \nabla$ with folding locus $L\subset D$. For $x\in D$, let $S_x$
be its link in $D$, let $\phi_x\colon S_x\to S_{\phi(x)}$ be the map induced by $\phi$ between the links, and let $\phi^\angle_x$ be the composition
of $\phi_x$ with $S_{\phi(x)}\to S^\angle_{\phi(x)}$.

Suppose now that $D$, with the degenerate piecewise Euclidean metric pulled back from $\nabla^+$, is a convex subset of the Euclidean plane $\R^2$. For $x\in D$, let $\mathbb{S}_x$ be its unit tangent space in~$\R^2$.
Parallel transporting, we can identify all $\mathbb{S}_x$, which we denote then as $\mathbb{S}$. For $x\notin L\cup \phi^{-1}(\partial \nabla^+)$,
the map~$\phi_x$ extends to an isometry $\phi_x\colon \mathbb{S}=\mathbb{S}_x\to S_{\phi(x)}=S$.
Since $D\subset \R^2$, all edges of $L$ are mapped
under~$\phi$ to principal lines, and hence their directions are mapped to $\partial S^\angle$ in $S^\angle$. Thus it can be easily seen that
$\phi^\angle_x$ does not depend on $x$, and we can denote it $\phi^\angle$.
\end{defin}

Below, a \emph{positively} (resp.\ \emph{negatively}) \emph{oriented} geodesic ray in the folding locus consists of edges that are all oriented
towards (resp.\ away from) the unbounded end of the ray.

\begin{lemma}
\label{lem:noncompact} Let $\phi\colon D\to \nabla^+$ be an $\X$-reduced relative half-plane diagram with interior vertex angles $<2\pi+\eps_0$.
Let $r$ be a positively oriented geodesic ray in the folding locus $L\subset D$.
\begin{enumerate}[(i)]
\item Then $\phi(r)$ is unbounded.
\item Assume additionally that $r$ is a connected component of $L$. Then there is $M$ depending only on the first edge of $\phi(r)$, such that for the
    subray $r'$ obtained from $r$ by removing its initial length $M$ segment, we have that $\phi(r')$ is contained in a principal line.
\end{enumerate}
The same statements hold for $r$ negatively oriented.

Finally, if $r$ is a geodesic ray in $D$ disjoint from $L$, with the tangent direction of~$r$ at some (hence each) point~$x$ of~$r$ sent to $\partial
S^\angle$ under the map $\phi^\angle_x$ from Definition~\ref{def:u}, then $\phi(r)$ is unbounded as well.
\end{lemma}
\begin{proof}
If $r$ is positively oriented, then each edge of $\phi(r)$ is directed away from $[1,0,0],[0,1,0],$ or $[0,0,1]$. The change of the direction might
occur only at $\partial \nabla^+$ or at $\phi(v)$ for a vertex $v$ described in Figure~\ref{fig:concise}(e,h).
We have the following possibilities, as illustrated in Figure~\ref{fig:noncompact}:

\begin{figure}
\includegraphics{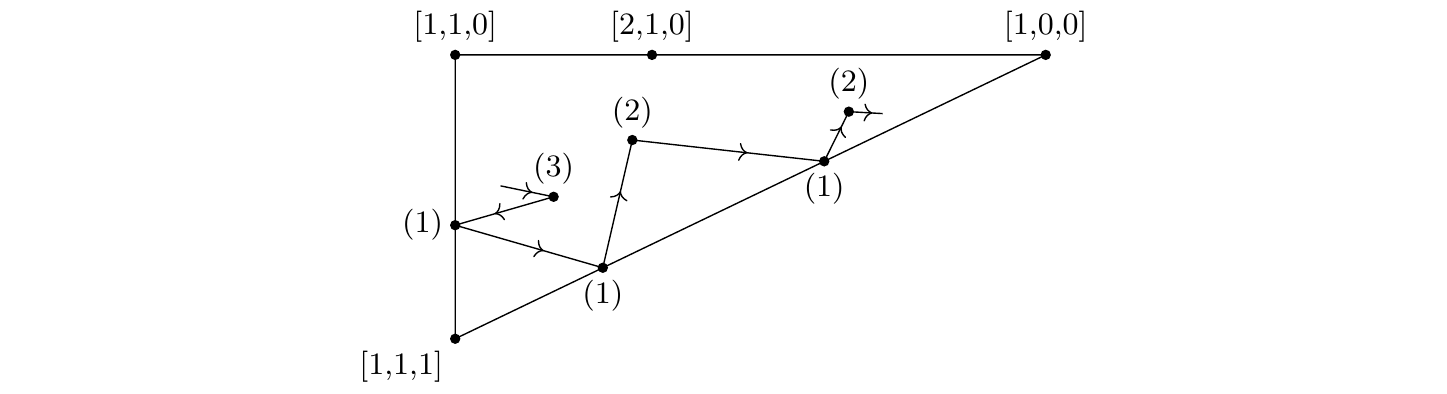}
\caption{Image of a positively oriented ray in $\nabla^+$. Numbers (1), (2), (3) correspond to the cases in the proof of Lemma \ref{lem:noncompact}.}
\label{fig:noncompact}
\end{figure}

\begin{enumerate}[(1)]
\item
At $\partial \nabla^+$, the direction
away from $[1,0,0]$ changes to the direction away from $[0,1,0]$ or the direction away from $[0,1,0]$ changes to the direction away from $[0,0,1]$.
\item
At $\phi(v)$ as in Figure~\ref{fig:concise}(e) (which cannot occur under the additional assumption in~(ii)), the direction away from $[0,0,1]$
changes to the direction away from $[0,1,0]$. In that case $\phi(v)$ has $\frac{\alpha_1}{\alpha_2}$ an integer $\geq 2$.
\item
At $\phi(v)$ as in Figure~\ref{fig:concise}(h) (which also cannot occur in~(ii)), the direction away from $[0,1,0]$
changes to the direction away from $[1,0,0]$. In that case $\phi(v)$ has $\frac{\alpha_1}{\alpha_2}<2$.
\end{enumerate}

Moving away from $[0,1,0]$ increases $\frac{\alpha_1}{\alpha_2}$. Consequently, for a subray $r'\subset r$ each edge of $\phi(r')$ is directed away
from $[0,1,0]$ (which eventually does not happen in~(ii)) or $[0,0,1]$. Thus the orthogonal projection w.r.t.\ $|\cdot,\cdot|$ of~$\phi(r)$ onto the
principal line $\alpha_2=\alpha_3$ diverges to $\infty$.

In $r$ is negatively oriented, for a subray $r'\subset r$ each edge of $\phi(r')$ is
directed towards $[0,1,0]$ (which eventually does not happen in~(ii)) or $[1,0,0]$. Thus the orthogonal projection w.r.t.\ $|\cdot,\cdot|$ of~$\phi(r)$ onto the principal line $\alpha_1=\alpha_2$ diverges to~$\infty$.

The last assertion is proved similarly.
\end{proof}

Below, a \emph{non-oriented component} $\LL$ of the folding locus $L$ is a connected component of the union of the open non-oriented edges of $L$ and
the vertices in Figure~\ref{fig:concise}(d,g).

\begin{lem}
\label{lem:curved_edges} There is a continuous function $M\colon \big(0,\frac{\pi}{3}\big)\to \R$ such that the following holds. Let $\phi\colon D\to
\nabla^+$ be an $\X$-reduced relative half-plane diagram with interior vertex angles $<2\pi+\eps_0$. Let $\LL\subset L$ be a non-oriented component,
let $x\in \LL$, and let ${u}\in S_x$ be the direction of an edge of $\LL$. If $\phi^\angle_x({u})\in S^\angle$ is at angle $\eps$ from the principal
point, then the connected component of $\LL\setminus x$, to which $u$ is tangent 
\begin{itemize}
\item
has length $<M(\eps)$ and ends at $\partial D$, or 
\item
contributes $\leq
-\eps$ to the sum of the curvatures of the interior edges and the interior vertices of~$D$.
\end{itemize}
\end{lem}
\begin{proof} Consider the arc-length parametrisation $\LL\colon [0,T)\to D$ (possibly $T=\infty$) of the union of $x$ and the connected component of $\LL\setminus x$ tangent to $u$,
with $\LL(0)=x$. Let ${u}(t)\in S_{\LL(t)}$ be the direction of $\LL[t,T)$. In the case where $T<\infty$ and $\lim_{t\to T}\LL(t)$ is a vertex in the
interior of $D$, which we call~$\LL(T)$, we have the configuration of Figure~\ref{fig:concise}(c,f,h,i). For (c,f,i), we define additionally
${u}(T)\in S_{\LL(T)}$ to be the direction of the edge $e$ of~$L$ opposite at $\LL(T)$ to the non-oriented edge of $\LL$ (Remark~\ref{rem:straight}(i)). For (h) we denote by $e$
the edge outside $L$ indicated in Figure~\ref{fig:h} whose direction ${u}(T)\in S_{\LL(T)}$ has $\phi^\angle_{\LL(T)}\big({u}(T)\big)$ principal.

\begin{figure}
\includegraphics{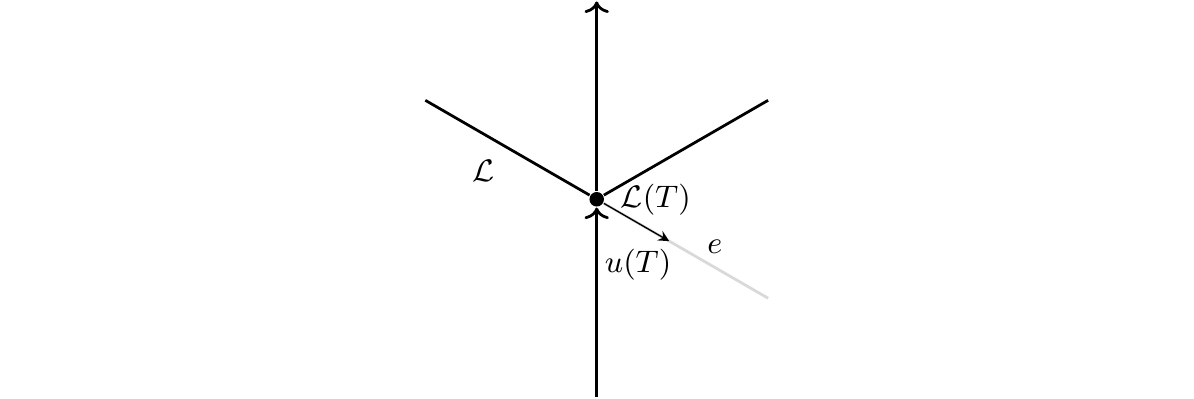}
\caption{}
\label{fig:h}
\end{figure}

We claim that $t\to \phi^\angle_{\LL(t)}\big({u}(t)\big)\in S^\angle$ is a decreasing function, where the interval~$S^\angle$ of length $\frac{\pi}{3}$ is parametrised by the angle from the principal point. Furthermore, the sum of the integral of the geodesic
curvature and the exterior angles at each side of an interval $\LL(t_1,t_2)$ (resp. $\LL(t_1,t_2]$, including possibly the exterior angle at $\LL(T)$ of $\LL\cdot e$ for $t_2=T$), is at least as large as the angle
between~$\phi^\angle_{\LL(t_1)}\big(u(t_1)\big)$ and $\lim_{t\nearrow t_2}\phi^\angle_{\LL(t)}\big({u}(t)\big)$ (resp.\
$\phi^\angle_{\LL(t_2)}\big({u}(t_2)\big)$).

Indeed, for $\LL(t_1,t_2)$ a single edge this follows from the
definition of geodesic curvature and the position of non-principal admissible lines in $\nabla$. 
Furthermore, for $t\in (t_1,t_2]$ with $\LL(t)$ a vertex, we have the configuration of Figure~\ref{fig:concise}(d,g) for $t< T$ or
Figure~\ref{fig:concise}(c,f,h,i) for $t=T$. Except for (h,i), in all these cases the angle between $\lim_{t'\nearrow t} \phi^\angle_{\LL(t')}\big({u}(t')\big)$ and
$\phi^\angle_{\LL(t)}\big({u}(t)\big)$ is the exterior angle at $\LL(t)$ of $\LL\cdot e$. In case~(i), and on one side in case (h) the exterior angle is even larger. This justifies the claim.

Let now $\LL'\colon (-\infty,+\infty)\to \nabla$ be any non-principal admissible line, all of which are isometric by
Remark~\ref{rem:unique_admissible}. Let $K$ be the path in $\LL'$ of points where the direction of $\LL'$ maps under $S\to S^\angle$ to the interval
$\big[\frac{\eps}{2}, \eps\big]$. Let $M=M(\eps)$ be the length of~$K$.

Without loss of generality, we can assume that $\phi(x)$ is the starting point of $K$. Furthermore, by the monotonicity in the claim
and by Remark~\ref{rem:unique_admissible}, consecutive edges of $\phi\big(\LL[0,T)\big)$ can be mapped by the translations of $\nabla$ to disjoint (except at
the endpoints) edges of~$\LL'$, ordered consistently. Consequently, if $\LL$ has length $\geq M$, then the minimal subpath~$K'$ of $\LL'$ containing
all of these edges has also length $\geq M$. Then $K'$ contains a point where the direction of $\LL'$ maps under $S\to S^\angle$ into $[0,\frac{\eps}{2}]$.
By the claim, the sum of the curvatures at each side of $\LL$ is $\leq -\frac{\eps}{2}$, as desired. If $\LL$ has length $<M$ but ends in the interior of $D$, then the image under $S\to S^\angle$ of the direction
of $\phi(e)$ above is the principal point and so by the claim the sum of the curvatures at each side of $\LL$ is $\leq -\eps$.
\end{proof}

\section{Rays}
\label{sec:rays}

In this section, we discuss particular geodesic rays in $\X$ that represent some points in $\partial_\infty \X$ fixed by elements conjugate into $C$ or $B'$.

\subsection{Principal rays}

Let $a\geq 1$. An \emph{$a$-principal ray} is a geodesic ray in a chamber~$\Ap^+_f$ that is the reparametrisation of the curve $\nu_{f,[ta, t,1]}$,
where $t\geq t_0$ for some $t_0\geq 1$. An \emph{$a$-antiprincipal ray} is a geodesic ray in a chamber $\Ap^+_f$ that is the reparametrisation of
the curve $\nu_{f,[t, a,1]}$, where $t\geq t_0$ for some $t_0\geq a$. A geodesic ray is \emph{principal} (resp.\ \emph{antiprincipal}) if it is an
$a$-principal (resp.\ $a$-antiprincipal) ray for some $a\geq 1$. See Figure~\ref{fig:rays}.

\begin{figure}
\includegraphics{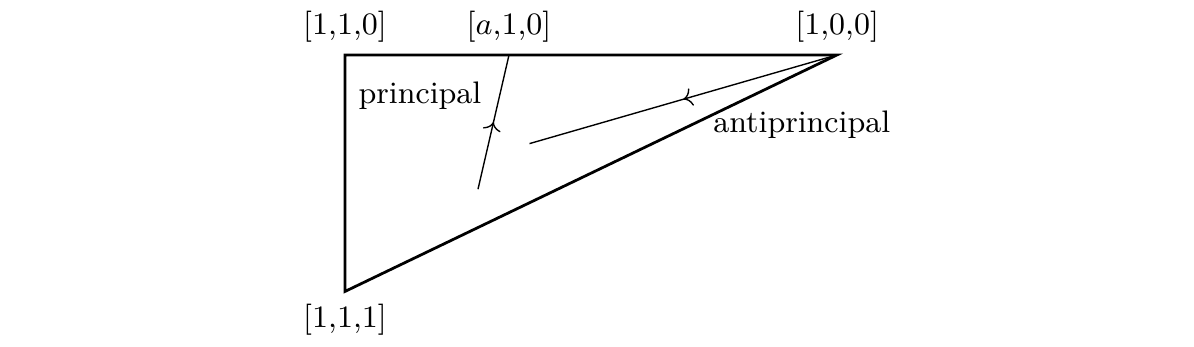}
\caption{Principal and antipricipal rays in a chamber $\Ap^+_f$.}
\label{fig:rays}
\end{figure}

Principal rays $r,r'$ are \emph{elementarily parallel}, if there is a chamber containing them both in which $r,r'$ are parallel with respect to the Euclidean metric $|\cdot,\cdot|$. 
This generates an equivalence relation on the set of principal rays, and we say that $r,r'$ are \emph{parallel} if
they are in the same equivalence class. 
A class of parallel rays determines a point $\zeta$ in the visual boundary $\partial_\infty  \X$. We call
such a point $\zeta$ \emph{principal}. Analogously we define \emph{antiprincipal} points in $\partial_\infty \X$.

\begin{lem}
\label{lem:easystab} Let $\zeta\in \partial_\infty  \X$ be the principal \parent{resp.\ antiprincipal} point that is the class of the principal \parent{resp.\
antiprincipal} rays in $\Ap^+_\id$. Then the stabilizer $\St(\zeta) < \T$ equals $C$ \parent{resp.\ $B'$}.
\end{lem}

We need a preparatory result.

\begin{lem}
\label{lem:asymptoticparallel} Let $r$ be a principal \parent{resp.\ antiprincipal} ray in $\Ap^+_\id$, and let $r'$ be any geodesic ray in~$\X$
asymptotic to $r$. Then $r'$ has a subray that is principal \parent{resp.\ antiprincipal} and parallel to $r$.
\end{lem}

\begin{proof} Suppose first that $r\colon [0,\infty)\to \Ap^+_\id$ is principal.
Let $d>0$ be such that $d_\X\big(r(t),r'(t)\big)\leq d$ for all $t\geq 0$. After passing to subrays of $r$ and $r'$, we can assume that $\rho_+(r)$
is at distance $>d$ from the set $\{\alpha_2=\alpha_3\}$ in the metric $|\cdot, \cdot|$. Let $\overline{x}_0=r(0), \overline{x}_1,\ldots,
\overline{x}_k=r'(0)$ be points guaranteed by Lemma~\ref{lem:propertiesX}(iii) such that for each $i=0,\ldots, k-1$ both $\overline{x}_i$ and
$\overline{x}_{i+1}$ lie in a common chamber $\Ap_{f_i}^+$, and $\sum_{i=0}^{k-1}\big|\rho_+(\overline{x}_i), \rho_+(\overline{x}_{i+1})\big|=d$.
Thus all $\rho_+(\overline{x}_i)$ lie outside $\{\alpha_2=\alpha_3\}$. Hence, by Remark~\ref{rem:commonrays}, the principal rays from
$\overline{x}_i$ in $\Ap_{f_i}^+$ and $\Ap_{f_{i-1}}^+$ coincide. Consequently~$r$ and~$r'$ are parallel, as desired. If $r$ is antiprincipal, the
argument is analogous.
\end{proof}

\begin{proof}[Proof of Lemma~\ref{lem:easystab}]
Recall from Section~\ref{sec:point_stabilizers} that the group $C$ is generated by $K$ and~$H$. Applying Corollary~\ref{cor:fixed_set}, we obtain the
following. If $f\in K$, then~$\Ap^+_\id$ and~$\Ap^+_f$ share a $1$-principal ray. If $f\in H$, then $\Ap^+_\id$ and $\Ap^+_f$ share an $a$-principal
ray for $a$ sufficiently large (and thus they have parallel $1$-principal rays). In particular, $\St(\zeta)$ contains~$C$. Conversely, by
Proposition~\ref{prop:stab}, if for $f\in \T$ the chambers~$\Ap^+_\id$ and~$\Ap^+_f$ share a principal ray, then $f$ belongs to $H$ or $K$. Thus to
show that for $f\in \St(\zeta)$ we have $f\in C$ it remains to justify that $\Ap^+_\id$ and $\Ap^+_f$ have parallel principal rays. This follows from
applying Lemma~\ref{lem:asymptoticparallel} to $r'=f(r)$. The proof of the second assertion is analogous.
\end{proof}

\subsection{H-generic rays}

\begin{constr}\phantomsection
\label{rem:B'rank2} 
\begin{enumerate}[(1)]
\item \label{831}
Let $f$ be an element of $B'$ not conjugate into $B$ or $H$. Then there is an isometrically embedded Euclidean half-plane $\mathbf{H}$ in $\X$, invariant under~$f$.
Indeed, by Corollary~\ref{cor:fixed_set}, for $g,g'\in
B'$ with $g^{-1}g'\in B\setminus H$, the chambers $\Ap^+_g,\Ap^+_{g'}$ intersect along a $1$-antiprincipal ray, and for $g^{-1}g'\in H\setminus B$
they intersect along a region whose boundary contains an
$a$-antiprincipal ray for some $a\geq 2$. 
After possibly conjugating~$f$, we can assume that its normal form in $B'=B\ast_{B\cap H}H$ is $h_1b_1\cdots h_nb_n$. For $i=1,\ldots, n$, let $r'_i,
r_i$ be the above antiprincipal rays for $g=h_1\cdots b_{i-1}, g'=gh_i$ and $g=h_1\cdots h_i,g'=gb_i$, respectively. Choose a geodesic~$\sigma$ in
$(\nabla,|\cdot,\cdot|)$ perpendicular to the line $\alpha_2=\alpha_3$ that intersects the images under $\rho_+$ of all $r_i,r_i'$. Let
$r_0=f^{-1}(r_n)$. For $i=1,\ldots, n$, consider the geodesic $\sigma_i$ in $\Ap^+_{g}$ for $g=h_1\cdots b_{i-1}$ (respectively, $\sigma_i'$
in~$\Ap^+_{g'}$ for $g'=gh_{i}$) that joins $r_{i-1}$ to $r'_i$ (respectively, $r'_i$ to~$r_i$) and maps into $\sigma$ under~$\rho_+$. Then the union
of the concatenation $\sigma_1\cdot\sigma_1'\cdots \sigma_n\cdot\sigma_n'$ and its translates under~$\langle f \rangle$ is an axis of $f$ that
bounds an isometrically embedded Euclidean half-plane~$\mathbf{H}$ in~$\X$, which is the union of the half-strips bounded by $\sigma_i,\sigma_i'$ and
subrays of $r_i,r_i'$ (Remark~\ref{rem:isometric}(i)).

\item \label{832}
We have a following similar construction of an invariant convex subset $\mathbf{H}$ for an element $f\in C$ that is not conjugate into $K$ or $H$. By
Corollary~\ref{cor:fixed_set}, for $g,g'\in C$ with $g^{-1}g'\in K\setminus H$, the chambers $\Ap^+_g,\Ap^+_{g'}$ intersect along a $1$-principal
ray, and for $g^{-1}g'\in H\setminus K$ they intersect along a region whose boundary contains a curve projecting under $\rho_+$ inside the (possibly
non-principal) admissible line $\alpha_1=m_2\alpha_2+m_3\alpha_3$ for some $m_2\geq 1, m_2+m_3\geq 2$.
After possibly conjugating~$f$, we can assume that its normal form in $C=K\ast_{K\cap H}H$ is $h_1k_1\cdots h_nk_n$. For $i=1,\ldots, n$, let $c_i,
r_i$ be the above curves and principal rays for $g=h_1\cdots k_{i-1}, g'=gh_i$ and $g=h_1\cdots h_i,g'=gk_i$, respectively. Choose a
geodesic~$\sigma$ in $(\nabla,|\cdot,\cdot|)$ perpendicular to the line $\alpha_1=\alpha_2$ that intersects the images under $\rho_+$ of all
$c_i,r_i$. Let $r_0=f^{-1}(r_n)$. For $i=1,\ldots, n$, consider the geodesic $\sigma_i$ in $\Ap^+_{g}$ for $g=h_1\cdots k_{i-1}$ (respectively,
$\sigma_i'$ in~$\Ap^+_{g'}$ for $g'=gh_{i}$) that joins $r_{i-1}$ to $c_i$ (respectively, $c_i$ to $r_i$) and maps into~$\sigma$ under~$\rho_+$. Let
$R_i,R_i'$ be the Euclidean half-strips bounded by $\sigma_i,\sigma_i'$ and the principal rays issuing from their endpoints. Let $R'_0=f^{-1}(R'_n)$.
Note that for $i=1,\ldots, n,$ the half-strips $R_{i-1}', R_{i}$ intersect along the $1$-principal rays by which they are bounded, and $R_i, R'_i$
intersect along the regions bounded by the other principal rays and the curves inside $c_i$. Thus the union of all $R_i,R_i'$ and their translates
under $\langle f \rangle$ form a subset~$\mathbf{H}$ that is, by Remark~\ref{rem:isometric}(iii), locally convex, and hence convex in~$\X$.
\end{enumerate}
\end{constr}

We have the following immediate consequence of Construction~\ref{rem:B'rank2}.

\begin{cor}\phantomsection
\label{cor:not_rank_1}
\begin{enumerate}[(1)]
\item
Let $f$ be an element of $B'$ not conjugate into $B$ or $H$. Then $f$ is loxodromic not of rank~$1$. Furthermore, the translation length $|f|$ is the sum of the distances from $r_{i-1}$ to $r'_i$, and from $r'_i$ to~$r_i$, for $i=1,\ldots, n$, in the notation of Construction~\ref{rem:B'rank2}\eqref{831}.
\item
Let $f$ be an element of $C$ not conjugate into $K$ or $H$. Then $f$ is loxodromic not of rank~$1$, if all $c_i$ are principal rays, or parabolic, otherwise. 
Furthermore, $|f|$ is the sum of the distances \parent{that might be zero} from $r_{i-1}$ to $c_i$, and from $c_i$ to~$r_i$, for $i=1,\ldots, n$, in the notation of Construction~\ref{rem:B'rank2}\eqref{832}.
\end{enumerate}
\end{cor}

\begin{exa}
The map $f=(x_2,x_1+x_2x_3, x_3)\in C$ is parabolic of translation length $0$. The map $f=(x_2,x_1+x^2_2x_3, x_3)\in C$ is parabolic with positive translation length.
\end{exa}

\begin{defin} 
\label{def:generic}
Let $f$ be an element of $B'$ not conjugate into $B$ or $H$ or an element of $C$ not conjugate into $K$ or $H$, with positive translation length. Let $\xi\in \partial_\infty  \X$ be a
limit point of $f$. A geodesic ray $r\colon J\to \mathbf{H}$ representing~$\xi$, for~$\mathbf{H}$ as in Construction~\ref{rem:B'rank2}, is
\emph{$\mathbf{H}$-generic} if it is disjoint from the vertex set $\X^0$ and transverse to the $1$-skeleton $\X^1$. The consecutive points $l_k\in J$
at which $r$ passes from one chamber ~$\Ap^+_{g}$ to another are called \emph{transition points}.
\end{defin}

\begin{rem}
\label{rem:generic_exist}

Let $f$ be an element of $B'$ not conjugate into $B$ or $H$ or an element of $C$ not conjugate into $K$ or $H$, with positive translation length, and let $\mathbf{H}$ be as in
    Construction~\ref{rem:B'rank2}. 
    Let $\zeta\in \partial_\infty \X$ be the antiprincipal or principal point fixed by $B'$ or~$C$, respectively.
   Let $\xi\in \partial_\infty \X$ be a limit point of $f$. Note that since $\mathbf{H}$ is convex and $f$-invariant, each geodesic ray starting
   in~$\mathbf{H}$ and representing~$\xi$ is contained in~$\mathbf{H}$. By \cite{CL}*{Thm 1.1} (see also \cite{D}*{thm~6.1}), we have that~$\xi$ is at
   Tits distance~$\pi$ in $\partial_\infty \mathbf{H} \subset \partial_\infty \X$ from the other limit point of~$f$, and so
   $\angle(\xi,\zeta)=\frac{\pi}{2}$. In particular $\xi\neq
   \zeta$.

\begin{enumerate}[(i)]
\item An $\mathbf{H}$-generic ray exists. Indeed, choose a ray $s$ representing $\zeta$ contained in~$\mathbf{H}$ and disjoint from all the
    curves~$c_i$. Let $\mathcal R$ be the set of geodesic rays $r$ representing~$\xi$ starting at the points of~$s$. Note that distinct rays
    $r,r'\in\mathcal R$ are disjoint, because they are contained in a subspace of $\mathbf{H}$ with no edges of degree $3$, and because $r\subset
    r'$ would imply $\xi=\zeta$. Since there are uncountably many points in~$s$, there exists $r\in \mathcal R$ disjoint from the vertices of
    $\X$ in $\mathbf{H}$. Similarly, since the edges of $\X$ have geodesic curvature of constant sign in each $2$-cell, if two rays in~$\mathcal
    R$ are
tangent to a given edge of $\X$ intersecting $\mathbf{H}$, then we have one of the configurations in Figure~\ref{fig:atmost2}. The bottom configuration is not
allowed
since the subspace bounded by the two rays and the edge, which is CAT(0) by the same argument as in Definition~\ref{def:frilled}, would have two
asymptotic geodesic rays starting at the same point.
Consequently, at most two rays in~$\mathcal R$ are tangent to a given edge of $\X$. Thus, there exists a ray in~$\mathcal R$ transverse
to~$\X^1$ and disjoint from $\X^0$.

\item
    By \cite{BH}*{II.9.8(2)}, the Alexandrov angle at $r(t)$ between $r$ and the geodesic ray representing
    $\zeta$
    converges to~$\frac{\pi}{2}$ and so the same holds at the intersection with the curves $c_i$. 

\item Consequently, using \cite{BH}*{II.9.3}, if at least one $c_i$ projects to a non-principal admissible line in $\nabla^+$, then $\rho_+(r)$
    leaves every compact set of $\nabla^+$.

\item Conversely, suppose that all $c_i$ project to principal lines in $\nabla^+$. Then $\mathbf{H}$ is isometric to the Euclidean half-plane and
    all above intersection angles are~$\frac{\pi}{2}$. Consequently, we have $f(r)\subset r$.

\item By (ii), the distances in~$J$ between the consecutive transition points of $r$ are uniformly bounded from above.
\end{enumerate}
\end{rem}

\begin{figure}
\includegraphics{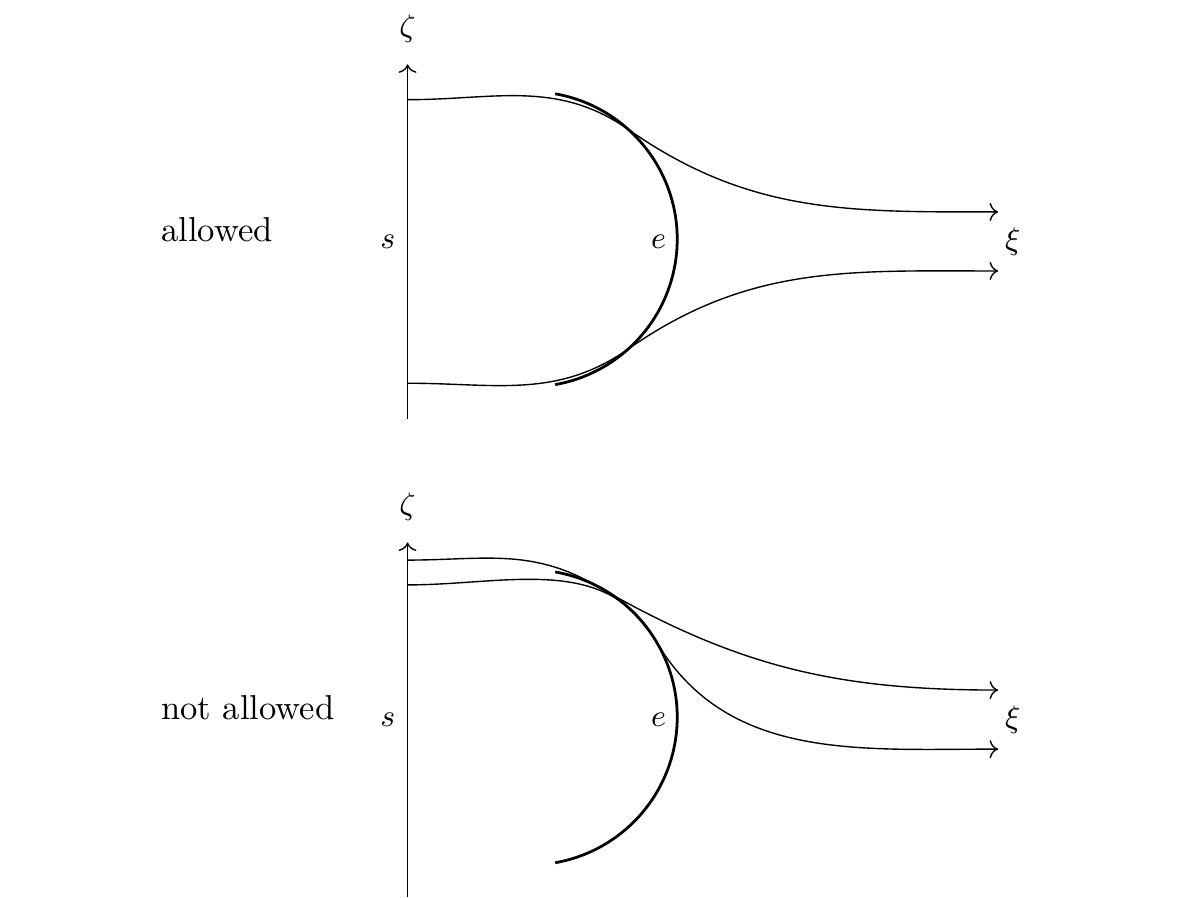}
\caption{The solid curve $e$ represents an edge to which two distinct rays in $\mathcal{R}$ are tangent.}
\label{fig:atmost2}
\end{figure}

\begin{lem}
\label{lem:generic} Let $\widetilde\phi_n \colon D_n\to \X$ be a sequence of reduced relative disc diagrams. Suppose that
$\phi_n=\rho_+\circ \widetilde \phi_n$ have the limit half-plane diagram $\phi'\colon D'\to \nabla^+$ with folding locus $L$. Furthermore, assume
that $\phi'$ is bordered by an $\mathbf{H}$-generic ray $r\colon J\to \X$. Suppose that the sum of the curvatures
of the vertices of $J\subset \partial D'$ is $\geq -\eps_0$ from Corollary~\ref{cor:consise}.

Then the transition points $l_k$ are contained in single edges of $L$ at angle to $J$ converging to~$\frac{\pi}{2}$. See Figure~\ref{fig:lk}.

Furthermore, for any $l\neq l_k$ in $L\cap J$, the angles between the edges of $L$ at $l$ and~$J$ are~$<\frac{\eps_0}{2}$.
\end{lem}

\begin{figure}
\includegraphics{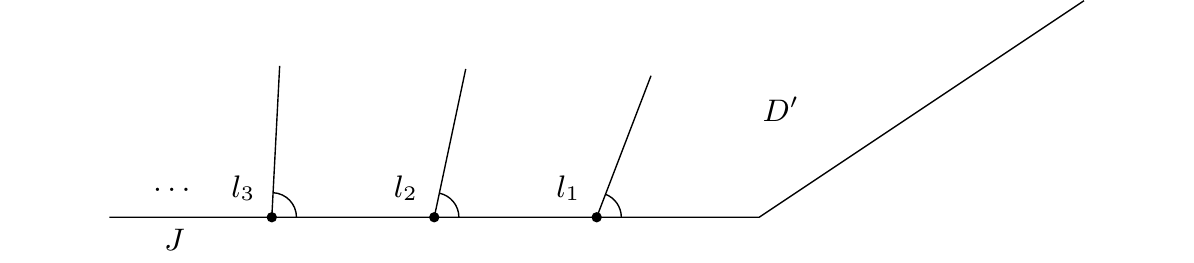}
\caption{}
\label{fig:lk}
\end{figure}

\begin{proof}
Let $l=l_k$ be a transition point. Then $l\in L$. Moreover, since $x=r(l)$ is not a vertex of $\X$, there is an edge~$e$ in $\X$ containing $x$ in
its interior. Since $e$ lies in a translate of $c_i,r_i$, or $r_i'$, the Alexandrov angles between~$r$ and such $e$ converge with~$k$ to $\frac{\pi}{2}$ by
Remark~\ref{rem:generic_exist}(ii).

Let $e_1,e_2$ be the connected components of $e\setminus x$. If the point $l$ is not contained in a single edge $f$ of $L$, since all such edges map
under $\widetilde \phi_n$ into $e_1$ or $e_2$, we have that the limit inferior (as $k\to \infty$) of the curvature of such $l$ is $\leq -\pi$, which is
a contradiction. Similarly, the angles between $J$ and $f$ have to coincide with the Alexandrov angles between $r$ and~$e$.

For the last assertion, let $x=r(l)$, where $l$ is not a transition point. If $l\in L$, then $l$ is
contained in two edges $f_1,f_2$ mapping under $\widetilde \phi_n$ into $e_1$ and $e_2$. The angles that $f_i$ make with $J$ are equal, so if they are
$\geq\frac{\eps_0}{2}$, then, since the angle between $f_1$ and $f_2$ is $\pi$, this contradicts the assumption that the curvature of $l$ is
$>-\eps_0$. See Figure~\ref{fig:lk2}.
\end{proof}

\begin{figure}
\includegraphics{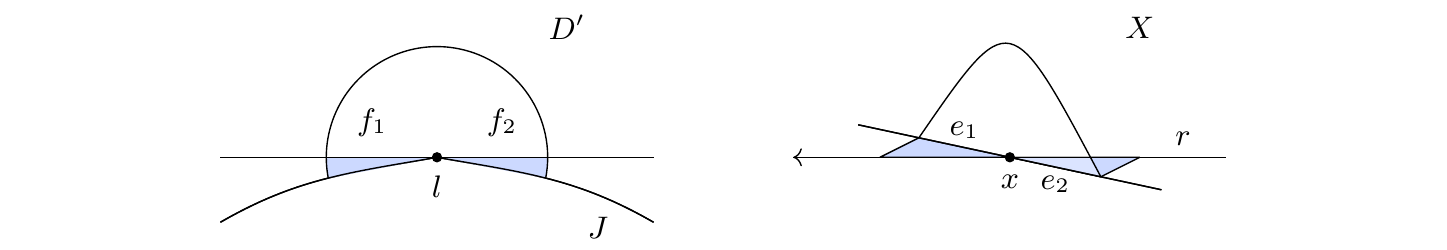}
\caption{}
\label{fig:lk2}
\end{figure}

\begin{prop}
\label{prop:diagram->fixedpoint} Under the assumptions of Lemma~\ref{lem:generic}, suppose additionally that the sum of the curvatures
of the interior edges and the interior vertices of $D'$ as well as
of the vertices of $J$ is $\geq -\eps_0$.

\begin{enumerate}[(1)]
\item
Suppose that $l_k$ are contained in connected components $L_k$ of
$L$ that are asymptotic geodesic rays. Then there is $k$ and a subray $L'_k\subset L_k$ such that

\begin{itemize}
\item $\phi'_{|L'_k}$ is an embedding into an admissible line diverging from $\{\alpha_2=\alpha_3\}$ \parent{resp.\ $\{\alpha_1=\alpha_2$\}}, and
\item for each $l'_k\in L'_k$, for $n$ sufficiently large, the class in $\partial_\infty \X$ of the unique principal \parent{resp.\ antiprincipal} ray in $\X$ starting at $\widetilde \phi_n(l_k')$ is fixed by $f$ from Definition~\ref{def:generic}.
\end{itemize}

\item Alternatively, suppose that $L_k$ are geodesic segments, whose length becomes
arbitrarily large when $k\to \infty$, ending at points $v_k\in \partial D'$ where $v_kv_{k+1}\subset \partial D'$ is geodesic and $\angle_{v_k}(L_k,v_{k+1})+\angle_{v_{k+1}}(L_{k+1},v_k)\geq \pi$. 
Then there is $k$ such that for all $m\geq k$ we have that
\begin{itemize}
\item $\phi'(v_m)$ lies outside $\{\alpha_2=\alpha_3\}$ \parent{resp.\ $\{\alpha_1=\alpha_2\}$}, and
\item
for $n$ sufficiently large, the class in $\partial_\infty \X$ of the unique principal \parent{resp.\ antiprincipal} ray in $\X$ starting at $\widetilde \phi_n(v_m)$ is fixed by $f$.
\end{itemize}
\end{enumerate}
\end{prop}
\begin{proof}
We focus on the assertion (1). Suppose first that $f\in C$ and there exists~$c_i$ in Construction~\ref{rem:B'rank2} that projects to a
non-principal admissible line in~$\nabla^+$. By Remark~\ref{rem:generic_exist}(iii), the projection $\rho_+(r)$ leaves every compact set $K$
of~$\nabla^+$. Choose~$K$ to contain $c\setminus c'$ for $c=\rho_+(c_i)$ and $c'$ as in Remark~\ref{rem:straight}(iii). Let $c_i'$ be the
preimage in~$c_i$ of $c'$ under $\rho_+$. Then we can consider $c_k',c_m'$ defined as two consecutive translates of~$c_i'$ intersected by $r$ in points
$r(l_k),r(l_m)$. Let $c''_k,c''_m$ be the unbounded connected components of $c'_k\setminus r(l_k), c'_m\setminus r(l_m)$. By \cite{BH}*{II.9.3}
applied in the subspace of $\mathbf{H}$ bounded by $c''_k, c''_m$ and $r(l_k)r(l_m)$, we then have
$\angle_{r(l_m)}\big(c''_m,r(l_{k})\big)+\angle_{r(l_{k})}\big(c''_{k},r(l_{m})\big)<\pi$.

Similarly, since $D'$ is $\mathrm{CAT}(0)$ and $L_k,L_m$ are asymptotic, we have $\angle_{l_m}(L_m,l_{k})+\angle_{l_{k}}(L_{k},l_{m})<\pi$.
Consequently, without loss of generality the first edge of $L_k$ is mapped under each $\widetilde \phi_n$ to~$c''_k$. By
Remark~\ref{rem:straight}(iv), we have $\widetilde \phi_n(L_k)\subset c''_k$, as desired.

Second, suppose that all $c_i$ in Construction~\ref{rem:B'rank2} project to principal lines in~$\nabla^+$. Then we have $f(r)\subset r$ by
Remark~\ref{rem:generic_exist}(iv). Note that since $J$ intersects all~$L_k$ perpendicularly, and all $L_k$ are asymptotic, by
\cite{BH}*{II.9.3(2)}, the regions between consecutive $L_k$ are Euclidean half-strips. Cutting $D$ along $L_1$, we can assume that $D'=D$ is a
quadrant in $\R^2$ with one boundary ray $J$, and $\phi'=\phi$. By Remark~\ref{rem:isometric}(i), we can assume that all $\widetilde \phi_n$ are
embeddings. Let $f_D\colon D\to D$ be the translation that acts on $J$ as $f$ does on $r$.

We apply Remark~\ref{rem:limit_union} with $D^+=D^-=D, \phi^{+}=\phi, \phi^{-}=f\circ \phi$ and $J^+=f_D(J),J^-=J$. By
Remark~\ref{rem:isometric}(ii), we have that $R^\pm$ are convex in $D$. Note that $R^+=f_D(R^-)$.

The path $P^+$ is asymptotic to a line in $D$
determining a direction in $\partial_\infty  D$ at angle $\theta$ from the direction of $J$. Consider first the case where $\theta=0$. Then there is
an infinite subcurve $\gamma\subset P^+$ with unit tangent directions mapped under $\phi^\angle$ \parent{see Definition~\ref{def:u}} arbitrarily close
in $S^\angle$ to the midpoint of $S^\angle$.

We claim that a neighbourhood $N$ of $\gamma$ in $\overline{D^+\setminus R^+}$ 
can be
identified with a subset of $\widehat D=\widehat D'$ 
on which $\phi^+,\widehat \phi$ coincide.
If
$\gamma=J^+$, then the claim follows from the last paragraph of Remark~\ref{rem:limit_union}. Otherwise, $\gamma$ maps to the $1$-skeleton of
$\widehat D$ under $\overline{D^+\setminus R^+}\subset \widehat D$, so $N$ can be taken as a union of cells of $\widehat D$ contained entirely in $D^+$ and the claim follows from the
second to last paragraph of Remark~\ref{rem:limit_union}. 

\begin{figure}
\includegraphics{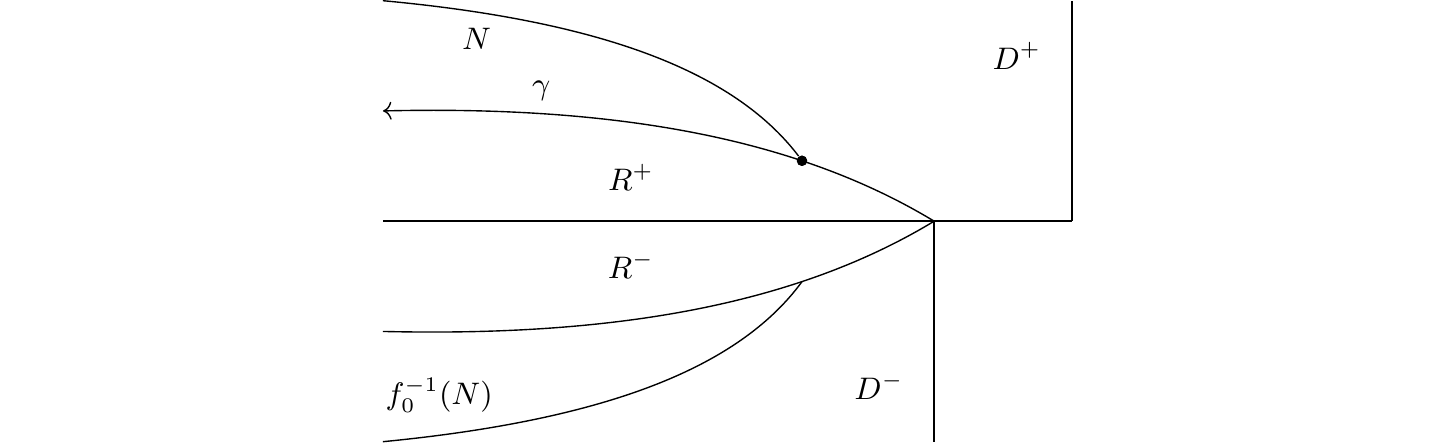}
\caption{}
\label{fig:N}
\end{figure}

We keep the notation $\gamma,N$ for their image in $\widehat D$. Since the function $\phi^\angle$ is invariant under translation, and applying the
claim to the neighbourhood $f^{-1}_D(N)$ of $f^{-1}_D(\gamma)$ in $\overline{D^-\setminus R^-}$, we have that the symmetry of $N\cup
f^{-1}_D(N)\subset \widehat D$ interchanging $N$ and~$f^{-1}_D(N)$ preserves $\widehat\phi^\angle$. See Figure~\ref{fig:N}. Since the unit tangent
directions of $\gamma$ are not mapped under $\phi^\angle$ to $\partial S^\angle$, we conclude that $\gamma$ is contained in the folding
locus~$\widehat L$ of~$\widehat D$. This contradicts Lemma~\ref{lem:curved_edges}.

Second, consider the case where $\theta>0$. We treat $L_k$ as rays $L_k\colon [0,\infty)\to D$ parametrised by arc-length. By
Lemma~\ref{lem:noncompact}\parent{ii}, there is $M$ such that each $\phi\big(L_k[M,\infty)\big)$ is contained in a principal line. Furthermore, since
$l_k$ are at distance uniformly bounded from above (Remark~\ref{rem:generic_exist}(v)), we can assume that the geodesic path $L_k(M)L_{k+1}(M)$
in~$D$ maps under $\phi$ outside $\{\alpha_2=\alpha_3\}$ (resp.\ $\{\alpha_1=\alpha_2\}$).

Since $\theta>0$, there are $m$ and $n$ with $l'_m=L_m(M)$ contained in $R^+_n$. Then for $l'_k=f^{-1}_D(l'_m)$, and
$x_k=\widetilde\phi_n(l'_k),x_m=\widetilde\phi_n(l'_m)$, we have $f(x_k)=x_m$. Consequently, $f$ maps the unique principal (resp.\ antiprincipal) ray
starting at $x_k$ to the one starting at $x_m$. These rays are parallel in $\X$ since $l'_k,l'_m$ are connected by a path in~$D_n$ mapping under
$\phi_n$ to a piecewise geodesic path outside $\{\alpha_2=\alpha_3\}$ (resp.\ $\{\alpha_1=\alpha_2\}$), and by Remark~\ref{rem:commonrays}.

The proof of assertion (2) is identical, with the following modifications. In the beginning of the proof, to justify
$\angle_{l_m}(L_m,l_{k})+\angle_{l_{k}}(L_{k},l_{m})<\pi$, we use the inequality in the hypothesis. Later, $D$ might be no longer a quadrant but a
possibly smaller subset of $\R^2$. The last difference is that now $L_k$ are bounded, but using the assumption that $L_k$ become longer than $M$ when
$k\to \infty$, we can still find $k$ such that for each $m\geq k,$ for sufficiently large $n$, we have that $l'_m=L_m(M)$ belongs to~$R^+_n$.
\end{proof}

\section{Classification}
\label{sec:classification}

In this section we finish the classification of parabolic isometries and loxodromic isometries not of rank 1, for the action of $\T$ on $\X$.
In particular, as a converse to Corollary \ref{cor:not_rank_1}, we show that such isometries always belong to the conjugates of the groups $B', C$ introduced in Section~\ref{sec:point_stabilizers}. 

\subsection{Parabolic isometries}

\begin{thm}
\label{thm:classify_isometries_parabolic} If $f\in \T$ is parabolic, then $f$ belongs to a conjugate of $C$. Furthermore, such a conjugate is unique.
\end{thm}
\begin{proof}
Let $\X_l\subset \X$ be the subspace of points $x$ with $d_\X\big(x,f(x)\big)\leq l$. We first prove that for each compact set $K\subset \nabla^+$
there is $l>|f|$ with $\rho_+(\X_l)\cap K=\emptyset$. Indeed, otherwise, there are $x_n\in \X_{|f|+\frac{1}{n}}$ with $\rho_+(x_n)\in K$. After
passing to a subsequence, we can assume $\rho_+(x_n)\to [\alpha]\in K$. Choose $\epsilon$ so that each $[\alpha']\in \nabla_+$ with
$|[\alpha'],[\alpha]|<\epsilon$ lies in the star of the cell containing $[\alpha]$ in its interior. Choose~$n$ such that
$\frac{1}{n}<\frac{\epsilon}{3}$ and $|\rho_+(x_n),[\alpha]|<\frac{\epsilon}{3}$. Then there is $y\in \X$ with $\rho_+(y)=[\alpha]$ and
$|x_n,y|<\frac{\epsilon}{3}$. Since $f$ does not have a fixed point and acts without inversions on $\X$, we have $f(y)\neq y$ and $|y,f(y)|\geq
\epsilon$. This contradicts $|y,f(y)|\leq |y,x_n|+|x_n,f(x_n)|+|f(x_n),f(y)|<\frac{\epsilon}{3}+\frac{\epsilon}{3}+\frac{\epsilon}{3}$.

Let $\nabla^+_l\subset \nabla^+$ be the set of points at distance $\leq l$ from $\partial \nabla^+$. We claim that $\rho_+(\X_l)\subset \nabla^+_l$.
Indeed, for each $x\in \X_l$, by Lemma~\ref{lem:propertiesX}(iii), there is a chain of length~$\leq l$ from $x$ to $f(x)$. If $|\rho_+(x),\partial
\nabla^+|> l$, then all the points in that chain map under~$\rho_+$ outside $\partial \nabla^+$. Consequently, by Corollary~\ref{cor:fixed_set},
the chambers containing the points of that chain have nonempty intersection, which is fixed pointwise by $f$, contradiction. 

Now, let $K$ be the $4|f|$ neighbourhood of $[1,1,1]$ in $\nabla^+$, and choose $2|f|>l>|f|$ with $\rho_+(\X_l)\cap K=\emptyset$. Note that $K$
separates $\nabla^+_l$ into two connected components, which are contained in the $l$-neighbourhoods of the sets $\{\alpha_1=\alpha_2\}$, or
$\{\alpha_2=\alpha_3\}$, respectively. Since $\X_l$ is convex, hence connected, it is mapped under $\rho_+$ into one of these two connected
components.

Consider first the former possibility. Let $x,f(x)\in \X_l$ and assume without loss of generality $x\in \Ap^+_\id$. Then the minimal length chain from
$x$ to $f(x)$ in $\X_l$ maps under~$\rho_+$ outside the set $\{\alpha_2=\alpha_3\}$. 
Then by Remark~\ref{rem:commonrays}, the principal rays in $\Ap^+_\id, \Ap^+_f$ starting at $x$ and at $f(x)$ are parallel.
By Lemma~\ref{lem:easystab}, we obtain $f\in C$, as desired. If $\X_l$ maps under $\rho_+$ to the other
connected component of $\nabla^+_l\setminus K$, we analogously obtain $f\in B'$. However, this contradicts Corollary~\ref{cor:not_rank_1}(1).

For the remaining assertion, consider parabolic $f\in C\cap gCg^{-1}$ for some $g\in \T$. Let $r\colon [0,\infty)\to \X$ be the geodesic ray starting
at $r(0)=x\in \X_l$ that represents the same point $\zeta$ in $\partial_\infty \X$ as the principal rays in $\Ap^+_g$. Since $f\in gCg^{-1}$, we have
$f(\zeta)=\zeta$. Consequently, for each $t\geq 0$, we have we have $d_\X\big(r(t),f(r(t))\big)\leq d_\X\big(r(0),f(r(0))\big)\leq l$, and so
$r[0,\infty)\subset \X_l$. By Lemma~\ref{lem:asymptoticparallel}, we have that a subray of $r$ is parallel to a principal ray in $\Ap^+_g$.
Consequently, there is a chamber $\Ap^+_h$ with $h\in C$ with $h$ simultaneously in~$gC$, and so $g\in C$.
\end{proof}

\begin{prop}
\label{prop:parabolicdynamics} Let $f\in C$ be parabolic with translation length zero and let $\zeta\in \partial_\infty  \X$ be the class of the
principal rays in $\Ap^+_{\id}$. Then $\zeta$ is the positive \parent{and the negative} limit point of~$f$.
\end{prop}

In the language of Section~\ref{sec:rank1}, Proposition~\ref{prop:parabolicdynamics} states that $f$ is not vile.

In the proof we will use the following. Let $a>0$. An \emph{$a$-sector} is the $\mathrm{CAT}(0)$ space obtained by glueing two Euclidean half-strips
$[0,\infty) \times [0,R]$ along copies of the set of points $(t,y)$ defined by $y\geq f(t)$ for some convex function $f\colon [0,\infty)\to [0,R]$
satisfying $\frac{1}{c}{e^{-at}}\leq f\leq ce^{-at}$ for some $c> 0$. We call the union of the two copies of the ray $[0,\infty)\times \{0\}$ the
\emph{frontier} of the sector.
An \emph{$a$-quasicusp} is the locally $\mathrm{CAT}(0)$ space glued of finitely many $a$-sectors by cyclically identifying them along their frontiers.

\begin{lem}
\label{lem:quasiH}
The universal cover $\mathbf{H}$ of an $a$-quasicusp 
is quasi-isometric with a horodisc $\mathbf{H}'$ in the rescaled hyperbolic plane $\frac{1}{a}\H^2$. Furthermore, the quasi-isometry conjugates the
deck transformation group of $\mathbf{H}$ with a group of parabolic isometries of $\mathbf{H}'$. \end{lem}

Note that we could replace $\frac{1}{a}\H^2$ by $\H^2$ in the statement by composing with the obvious quasi-isometry between
$\frac{1}{a}\H^2$ and $\H^2$, but a constant $a$ will appear naturally in the proof of Proposition~\ref{prop:parabolicdynamics}.

\begin{proof}
Let $f'$ be a parabolic isometry preserving $\mathbf{H}'$. Let $r$ be the image of a geodesic ray in $\mathbf{H}'$ starting on the horosphere
bounding $\mathbf{H}'$. It suffices to show that the union of $f'^n(r)$, over all $n\in \Z$, with the metric induced from $\mathbf{H}'$, is
quasi-isometric with the union of all the lifts to $\mathbf{H}$ of all $[0,\infty)\times \{0\}$ and all $\{(t,f(t))\}$ in all the sectors of a
quasicusp, with the metric induced from $\mathbf{H}$. To do that, it suffices to show that this metric on $r\cup f'(r)$ is bilipschitz to the metric
on the union of $[0,\infty)\times \{0\}$ and
$\{(t,f(t))\}$ in $[0,\infty)\times [0,R]$. 

To do that, note first that the maps $t\to (t,0),t\to (t,f(t))$ from $[0,\infty)$ to their image in $[0,\infty)\times [0,R]$ are bilipschitz
\parent{the first one is in fact an isometric embedding}. Furthermore, the distance between $(t',0)$ and $(t,f(t))$ is bounded from above by
$|t'-t|+f(t)$ and from below by both $|t'-t|$ and $f(t)$. Consequently, up to a multiplicative constant, this distance equals $|t-t'| +e^{-at}$. The
same formula holds in $\mathbf{H}'\subset \frac{1}{a} \H^2$ for the distances between the points on $r$ and $f'(r)$.
\end{proof}

\begin{proof}[Proof of Proposition~\ref{prop:parabolicdynamics}]
Let $\mathbf{H}$ be the $f$-invariant convex subset of $\X$ described in Construction~\ref{rem:B'rank2}(2). We claim that $\mathbf{H}$ is the universal
cover of a quasicusp with $\langle f \rangle$ the deck transformation group. Since $|f|=0$ (Corollary~\ref{cor:not_rank_1}(2)), for each $i=1,\ldots, n$, the curve~$c_i$ bounding the
intersection $R_i\cap R_i'\subset \mathbf{H}$ maps under $\rho_+$ to the admissible line with equation $\alpha_1=\alpha_2+m\alpha_3$ for some integer
$m\geq 1$. Passing to the Euclidean coordinates on the plane $\beta_1+\beta_2+\beta_3=0$ for $\beta_i=\log \alpha_i$ and $\alpha_1\alpha_2\alpha_3=1$
(see Section~\ref{sub:X}), this means $$e^{\beta_1}=e^{\beta_2}+me^{\beta_3}=e^{\beta_2}+me^{-(\beta_1+\beta_2)}.$$ On that plane we have orthogonal
coordinates $$y=\frac{\beta_1-\beta_2}{\sqrt 2},t=\frac{\sqrt {3}(\beta_1+\beta_2)}{\sqrt 2}.$$
The equation above has then the form $$e^\frac{\sqrt \frac{2}{3}t+\sqrt{2}y}{2}=e^\frac{\sqrt \frac{2}{3}t-\sqrt{2}y}{2}+me^{-\sqrt \frac{2}{3}t},$$
i.e.\ $2\sinh{\frac{y}{\sqrt 2}}=me^{-\sqrt \frac{3}{2}t}$. For $y$ close to $0$, up to a multiplicative constant, $\sinh{\frac{y}{\sqrt{2}}}$ equals
$y$, and so $y=y(t)$ has the required asymptotics for an $a$-sector with $a=\sqrt \frac{3}{2}$. This justifies the claim.

By Lemma~\ref{lem:quasiH}, we have an equivariant quasi-isometry $\mathbf{H}\to \mathbf{H}'$ for $\mathbf{H}'$ a horodisc in a hyperbolic
plane~$\frac{1}{a}\H^2$ preserved by a parabolic isometry $f'$. The horodisc~$\mathbf{H}'$ is Gromov-hyperbolic and for any point $x'\in
\mathbf{H}'$, for $n\to \pm \infty$, the sequence~$f'^n(x')$ converges w.r.t.\ the Gromov product to the unique point in the Gromov boundary
of~$\mathbf{H}'$. Consequently, $\mathbf{H}$ is Gromov-hyperbolic and for any point $x\in \mathbf{H}$, for $n\to \pm \infty$, the sequence $f^n(x)$
converges w.r.t.\ the Gromov product to the unique point in the Gromov boundary of $\mathbf{H}$. Since $\mathbf{H}$ is proper, by \cite{BH}*{III.H.3.7(2) and 3.17(6)}, for $n\to \pm \infty$, the sequence $f^n(x)$ converges in the cone topology to $\zeta$.
\end{proof}

\subsection{Loxodromic isometries not of rank 1}

\begin{thm}
\label{thm:classify_isometries_hyperbolic} Let $f\in \T$ be loxodromic not of rank~$1$. Then $f$ belongs to a conjugate of $C$ or $B'$. Furthermore,
there is at most one such a conjugate of $C$ and at most one such a conjugate of $B'$.
\end{thm}

\begin{exa}
\label{exa:B'andC}
The map $f=(x_2,x_1+x_2^2, x_3)$ belongs to both $C$ and $gB'g^{-1}$ for $g=(x_3,x_2,x_1)$.
\end{exa}

To prove the uniqueness in Theorem~\ref{thm:classify_isometries_hyperbolic}, it will be convenient to use the following, whose proof we postpone to
Section~\ref{sec:limitnotfar}.

\begin{lem}
\label{lem:principal_relation} Let $\xi\neq\eta\in \partial_\infty \X$ be principal and not far. Then $\angle (\xi,\eta)=\frac{2\pi}{3}$, the
midpoint of the geodesic $\xi\eta$ in $\partial_\infty \X$ is antiprincipal, and there is a chamber containing geodesic rays representing $\xi$ and that
midpoint. The same statement holds if we interchange the words `principal' and `antiprincipal'.
\end{lem}

Note that since $\partial_\infty \X$ is $\mathrm{CAT}(1)$, the geodesic $\xi\eta$ above is unique.

\begin{rem}
\label{rem:dim1} $\partial_\infty \X$ has geometric dimension~1 (see \cite{K}), which can be justified by the following argument proposed by
Pierre-Emmanuel Caprace. Namely, $\partial_\infty \X$ embeds in the space of directions of any asymptotic cone of $\X$, where the sequence of
basepoints is constant. Since $\X$ has geometric dimension~2, we have that the asymptotic cone has also geometric dimension $2$ \cite{LY}*{Lem~11.1},
and so $\partial_\infty \X$ has geometric dimension $1$. 

Consequently, $\partial_\infty \X$ is locally an $\R$-tree, see e.g.\ the proof of
\cite{OP}*{Lem~6.3}.
\end{rem}

\begin{proof}[Proof of Theorem \ref{thm:classify_isometries_hyperbolic}]
Let $\gamma$ be an axis of $f$ in $\X$. If $f$ is not of rank~$1$, then $\gamma$ is not contracting. Let $p_0\in \gamma$ and for $k\in \Z$ let
$p_k=f^k(p_0)$. By Lemma~\ref{lem:new} applied with $\beta=\gamma$ and $\mathcal P=\{p_k\}$, for each $n>0$, there is a $(p_k,\frac{1}{n},n)$-quadrilateral
with a side in $\gamma$. Thus its translate under $f^{-k}$ is a $(p_0,\frac{1}{n},n)$-quadrilateral $\square_n$ with a side in $\gamma$. By
Lemma~\ref{lem:square}, we can assume that $\square_n$ is embedded.

Let $\overline \phi_n\colon \overline D_n\to \X$ be a reduced relative disc diagram with boundary $\square_n$ guaranteed by
Lemma~\ref{lem:discexists}. Note that $\overline \phi_n$ is an embedding by Theorem~\ref{thm:disc_embeds}. Then for each subcomplex $D_n\subset
\overline D_n$ with $I_n=D_n\cap \partial \overline D_n,$ the restriction $\widetilde \phi_n$ of $\overline \phi_n$ to $D_n$ is a disc diagram in
$\X$ relative to $I_n$.

Since $\nabla^+$ is locally finite and since $\overline \phi_n$ are reduced disc diagrams, by Theorem~\ref{thm:GB} there is a uniform bound on the
degree of a vertex in $\overline D_n$ at a given distance from~$q_n=\overline \phi_n^{-1}(p_0)$. Furthermore, since $\square_n$ is a
$(p_0,\frac{1}{n},n)$-quadrilateral, we can assume that the radius $n$ closed ball in $\overline D_n$ centred at $q_n$ is disjoint from $\partial
\overline D_n\setminus \overline \phi_n^{-1}(\gamma)$. Thus we can find $D_n$ as above so that $I_n=\widetilde \phi^{-1}_n(\gamma)$, the pairs
$(D_n,I_n)$ converge to a half-plane diagram $(D,\partial D)$, and all $q_n$ coincide.

Furthermore, we can assume that we have the limit $\phi\colon D\to \nabla^+$ of $\rho_+\circ \widetilde \phi_n$ bordered by two rays with union
$\gamma$. (Here $D^*=D'=D$.) Let $\gamma_n=\phi_n(I_n)$. Note that since $\square_n$ were $(p_0,\frac{1}{n},n)$-quadrilaterals, by Theorem
\ref{thm:GB} we have that $D$, with the degenerate piecewise smooth Euclidean metric pulled back under $\phi$ from~$\nabla^+$, is isometric to the
Euclidean half-plane. Let $f_D\colon D\to D$ be the translation that preserves $\partial D$ and acts on it as $f$ does on $\gamma$.

We apply a variant of Remark~\ref{rem:limit_union} with $D^+=D^-=D, \phi^{+}=\phi, \phi^{-}=f\circ \phi$, identity frillings, and $J^+=J^-=\partial
D$, which are lines instead of rays. By Remark~\ref{rem:isometric}(ii), we have that $R^\pm$ are convex in $D^\pm$. Thus $R^+=R^-$ is a Euclidean
strip (possibly $R^\pm=D^\pm$).

Consider first the case where $R^\pm\neq D^\pm$. Then $\widehat D$ from Remark~\ref{rem:limit_union} becomes a combinatorial complex isometric to the
Euclidean plane and $\widehat \phi \colon \widehat D\to \nabla^+$ is combinatorial and $\X$-reduced. Let $L^\pm\subset D^\pm,\widehat L\subset
\widehat D$ denote the folding loci. We then have $\widehat L \cap (D^\pm\setminus R^\pm)=L^\pm\setminus R^\pm$. Observe that $\widehat L \cap
(D^+\setminus R^+)$ is send to $\widehat L \cap (D^-\setminus R^-)$ by the glide reflection $f_{\widehat D}$ of $\widehat D$ that is the composition
of $f_{D}$ with the reflection in $\mathrm{fr}\, R^+=\mathrm{fr}\, R^-$.

All the edges of $\widehat L$ are oriented, since $\widehat D$ is locally Euclidean. By Remark~\ref{rem:straight}(ii), for each oriented edge $e$ of
$\widehat L$, the entire oriented geodesic ray in $\widehat D$ starting at $e$ lies in~$\widehat L$. By Lemma~\ref{lem:noncompact}(i), such a
geodesic ray cannot be parallel to $\mathrm{fr}\, R^+$, in view of the fact that $\rho_+(\gamma)$ is bounded. Furthermore, by
Corollary~\ref{cor:consise}, none of such geodesics can intersect, except when one is contained in another. Thus by the observation above on
$f_{\widehat D}$, we have that $\widehat L$ has no vertices~$v$ of degree $\geq 3$, since otherwise some of the six oriented geodesics issuing from
$v$ and $f_{\widehat D}(v)$ would intersect.

If a connected component of $\widehat L$ was a line $l$ intersecting $\mathrm{fr}\, R^+$, then either $l$ would intersect~$f_{\widehat D}(l)$, which
is a contradiction, or $l$ would be perpendicular to $\mathrm{fr}\, R^+$. Then the function $\widehat\phi^\angle\colon \mathbb S\to S^\angle$ as in
Definition~\ref{def:u} would simultaneously assign to the direction of~$l$, which is the opposite of the direction of~$f_{\widehat D}(l)$, the
principal and the antiprincipal point, contradiction. Consequently, $\widehat L$ is empty, which contradicts the last assertion of
Lemma~\ref{lem:noncompact}.

Second, consider the case where $R^+=D$. Then $f_D$ preserves the folding locus~$L$ of $D$. Thus, as in the previous case, by
Remark~\ref{rem:straight}(ii)
and Corollary~\ref{cor:consise}, all the vertices of order $\geq 3$ in~$L$ are at a uniform distance from $\partial D$. Since $L$ is again nonempty, and
contains no line parallel to $\partial D$, there is a half-plane $D_0\subseteq D$ such that $D_0\cap L$ is a family of parallel oriented geodesic
rays $(L_k)_{k=-\infty}^{\infty}$ starting or ending at $\partial D_0$. By Lemma~\ref{lem:noncompact}(ii), there is $M$ such that each
$\phi(L_k[M,\infty))$ is contained in a principal line. Choose $m$ and $n$ with $L_m(M)$ contained in $R^+_n$. Then for $L_k(M)=f^{-1}_D(L_m(M))$,
and $x_k=\widetilde\phi_n(L_k(M)),x_m=\widetilde\phi_n(L_m(M))$, we have $f(x_k)=x_m$. Consequently, $f$ maps the unique (see
Remark~\ref{rem:commonrays}) principal or antiprincipal ray starting at $x_k$ to the one starting at $x_m$. These rays are asymptotic in $\X$ since
$L_k$ and $L_m$ are asymptotic in~$D$. By Lemma~\ref{lem:easystab}, $f$ lies in a conjugate of $C$ or $B'$.

For the uniqueness of the conjugate of $B'$, suppose $f\in B'\cap gB'g^{-1}$ with $g\notin B'$ and let $\zeta\neq g\zeta$ be the antiprincipal points
represented by the antiprincipal rays in $\Ap^+_\id, \Ap^+_g$. Let $\xi,\eta$ be the limit points of $f$. By
Construction~\ref{rem:B'rank2}(1), it follows that $\zeta,g\zeta$ are the midpoints of length $\pi$ paths $\omega_\id\neq \omega_g
\subset\partial_\infty \X$ between $\xi$ and $\eta$. Since $\partial_\infty \X$ is $\mathrm{CAT}(1)$ and is locally an $\R$-tree
(Remark~\ref{rem:dim1}), we have that $\omega=\omega_\id\cdot \omega_g$ is a circle of length $2\pi$ isometrically embedded in $\partial_\infty \X$.
By Lemma~\ref{lem:circle}, we obtain a flat $F\subset \X$ with $\partial_\infty F=\omega$. Thus $\zeta$ and $g\zeta$ are at distance~$\pi$ but they
are not far. This contradicts Lemma~\ref{lem:principal_relation}.

The proof of the uniqueness of the conjugate of $C$ is analogous.
\end{proof}

\section{\texorpdfstring{$(\xi,\eta)$}{(xi,eta)}-diagrams}
\label{sec:xieta}

We will be using the following diagrams to describe pairs of points in $\partial_\infty \X$ that are not far (see Definition~\ref{def:Rsquare}).

\begin{constr}
\label{con:diagram_sequence}
Let $\xi\neq\eta\in \partial_\infty\X$ be not far, and let $x\in \X$ be such that $\angle_x(\xi,\eta)>0$.
We construct a following $\X$-reduced relative half-plane diagram $\phi' \colon D'\to \nabla^+$, which we call a \emph{$(\xi,\eta)$-diagram}.

Consider $u_n,v_n\in \X\cup
\partial_\infty  \X$, with $u_n\to \xi, v_n\to \eta$,
and at most one of each $u_n,v_n$ in $\partial_\infty  \X$, such that either:
\begin{enumerate}[(i)]
\item \label{con:diagram_sequencei}
for each closed metric ball $P$ in $\X$, the geodesics $u_nv_n$ are disjoint from $P$ for $n$ large enough, or
\item \label{con:diagram_sequenceii} there is a closed metric ball $P$ in $\X$ containing a point $p_n$ on each $u_nv_n$, such that there is a
    $(p_n,\frac{1}{n},n)$-quadrilateral $\square_n$ with a side $\bar u_n\bar v_n$ in $u_nv_n$.
\end{enumerate}

In case \eqref{con:diagram_sequencei},
by \cite{BH}*{II.9.8(4)}, if $\angle(\xi,\eta)<\pi$, then we can assume that $u_n,v_n$ lie on the geodesic rays from $x$ to $\xi,\eta$, respectively.
Let $\bar u_n\in xu_n\cap u_nv_n, \bar v_n\in xv_n\cap u_nv_n$ be closest possible to $x$. Let $\triangle_n$ be the geodesic triangle $x\bar u_n\bar
v_n$, which is an embedded closed path. Let $\overline \phi_n\colon \overline D_n\to \X$ be a reduced relative disc diagram with boundary
$\triangle_n$ guaranteed by Lemma~\ref{lem:discexists}. Note that $\overline \phi_n$ is an embedding by Theorem \ref{thm:disc_embeds}. Then for each
subcomplex $D_n\subset \overline D_n$ with $I_n=D_n\cap \partial \overline D_n,$ the restriction $\widetilde \phi_n$ of $\overline \phi_n$ to $D_n$
is a disc diagram in~$\X$ relative to $I_n$.

Since $\nabla^+$ is locally finite and since $\overline \phi_n$ are reduced disc diagrams, by Theorem~\ref{thm:GB} there is a uniform bound on the
degree of a vertex in $\overline D_n$ at a given distance from~$q_n=\overline \phi_n^{-1}(x)$. Furthermore, by the hypothesis in case (i), we can
assume that the radius~$n$ closed ball in $\overline D_n$ centred at $q_n$ is disjoint from $\overline \phi_n^{-1}(\bar u_n\bar v_n)$. Thus, after
possibly passing to a subsequence, we can find~$D_n$ as above so that $(D_n,I_n)$ converge to a half-plane diagram $(D,\partial D)$ and all~$q_n$
coincide with a point $q\in \partial D$. Furthermore, we can assume Construction~\ref{constr:conv}\eqref{constr:conv:ii}. By
Remark~\ref{rem:converge_nabla}, where $I_a,I_b$ are the connected components of $\partial D\setminus q$ and $I_c$ is empty, we also have
Construction~\ref{constr:conv}\eqref{constr:conv:iii}. This gives the frilling $D'$ of $D$, and the limit map $\phi'\colon D'\to \nabla^+$ of
$\phi_n=\rho_+\circ \widetilde \phi_n$ that is an $\X$-reduced relative half-plane diagram. Note that by \cite{BH}*{II.9.8(4)} we have that $I_a,I_b$
represent points in $\partial_\infty D$ at angle~$\angle(\xi,\eta)$.

In case \eqref{con:diagram_sequenceii}, 
by Lemma~\ref{lem:square}, we can assume that each $\square_n$ is an embedded closed path. Let $\overline \phi_n\colon \overline D_n\to \X$ be
reduced relative disc diagrams with boundary $\square_n$ guaranteed by Lemma~\ref{lem:discexists}. Again, $\overline \phi_n$ is an embedding by
Theorem \ref{thm:disc_embeds}. All $\rho_+(p_n)$ lie in a finite set of cells of $\nabla^+$. Thus, after possibly passing to a subsequence, we have
$D_n\subset \overline D_n$ with $I_n=D_n\cap
\partial \overline D_n\subset \overline \phi_n^{-1}(\bar u_n\bar v_n)$ such that $(D_n,I_n)$ converge to a half-plane diagram $(D,\partial D)$.
Moreover, all $q_n=\overline\phi^{-1}(p_n)$ coincide with a point $q\in \partial D$. Furthermore, denoting $\widetilde \phi_n=\overline \phi_{|D_n}$,
we have the frilling~$D'$ of~$D$, and the limit map $\phi'\colon D'\to \nabla^+$ of $\phi_n=\rho_+\circ \widetilde \phi_n$ that is an $\X$-reduced
relative half-plane diagram. Since $\square_n$ were $(p_n,\frac{1}{n},n)$-quadrilaterals, $D'$ is isometric to the Euclidean half-plane in~$\R^2$.
\end{constr}

\begin{lemma}
\label{lem:boundarygood} Let $\phi'\colon D'\to \nabla^+$ be a $(\xi,\eta)$-diagram. If $\xi$ is principal or antiprincipal, then there is $J'\subset
\partial D'$ with $\phi'_{|J'}$ an embedding into a principal line.
\end{lemma}

Note that if $\phi'$ is bordered by a geodesic ray representing $\xi$, which happens for example for $\angle(\xi,\eta)<\pi$, then Lemma~\ref{lem:boundarygood}
follows immediately from Lemma~\ref{lem:asymptoticparallel}.

\begin{proof} By Lemma~\ref{lem:asymptoticparallel}, there is a point $x'$ on the geodesic ray from $x$ to $\xi$
such that the geodesic ray $r$ from $x'$ to $\xi$ is principal or antiprincipal. Let $J_\xi\subset \partial D$ be the ray starting with $q$ with $J_\xi\cap I_n$ mapping into $x\bar
u_n$ or $p_n\bar u_n$ under $\widetilde \phi_n$.

In case \eqref{con:diagram_sequencei} of Construction~\ref{con:diagram_sequence}, since $u_n\to \xi$, the restrictions of $\widetilde \phi_n$ to some
$J\subseteq J_\xi$ converge to $r$ (in $C^1$ piecewise, up to a parametrisation). Let $l$ be the principal line containing~$\rho_+(r)$.

In case \eqref{con:diagram_sequenceii} of Construction~\ref{con:diagram_sequence}, since $p_n\in P$, after possibly passing to a subray of~$r$, each
geodesic ray from $p_n$ to $\xi$ contains a subray $r_n$ parallel to $r$ with $\rho_+(r_n)$ in a uniform neighbourhood of $\rho_+(r)$. Furthermore, since
$u_n\to \xi$, the restrictions of $\widetilde\phi_n$ to each compact subset of some subray $J\subset J_\xi$, become arbitrarily close to $r_n$. Since $\rho_+\circ \widetilde
\phi_n$ have a limit, we conclude that $\rho_+(r_n)$ and $\rho_+\circ \widetilde \phi_n(J)$ converge into some principal line $l$.

In both cases, $\phi(J)$ lies in $l$, and so does $\phi'(J')$ for the image $J'$ of $J$ in $\partial D'$ by Remark~\ref{rem:converge_nabla}.
\end{proof}

\begin{lem}
\label{lem:almostflat} Let $\phi'\colon D'\to \nabla^+$ be a $(\xi,\eta)$-diagram from
Construction~\ref{con:diagram_sequence}\eqref{con:diagram_sequencei}, and assume
$\angle_x(\xi,\eta)>\angle(\xi,\eta)-\eps$, for some $\eps>0$. Suppose furthermore $u_n,v_n\in \X$ and $\angle_x(\xi,u_n),\angle_x(\eta,v_n)<\eps$.
Then for any $D_n$ \parent{and hence for $D'$} the sum of the curvatures of vertices and edges in the interior of $D_n$ and in the interior of
$I_n\setminus q_n$ is $\geq -3\eps$. Furthermore, the angle at $q_n$ is $\leq\angle(\xi,\eta)$.
\end{lem}

\begin{proof}
Suppose first that $\angle (\xi,\eta)<\pi$, in which case we assumed that $\bar u_n,\bar v_n$ lie on the geodesic rays from $x$ to $\xi,\eta$,
respectively. 
By Lemma~\ref{lem:twoinone}, we observe $\angle_{\bar u_n}(x,\bar v_n)+\angle_{\bar v_n}(x,\bar u_n)\geq\pi-\angle(\xi,\eta)$.
Thus by Theorem~\ref{thm:GB}, the sum of the curvatures in question is
$\geq -\eps$. The assertion on the angle at $q_n$ also follows from the above observation and Theorem~\ref{thm:GB}.

Second, suppose $\angle (\xi,\eta)=\pi$, and so $\angle_x(\xi,\eta)>\pi-\eps$. Thus $\angle_x(\bar u_n,\bar v_n)>\pi-3\eps.$ Then again by
Theorem \ref{thm:GB}, the sum of the curvatures in question is $\geq -3\eps$.
\end{proof}

We need a similar result for $\xi=\eta$. We start with the following.

\begin{lem}
\label{lem:morerays} Let $X$ be a complete $\mathrm{CAT}(0)$ space, and let $d > 0$. Then there exists a function $\psi\colon [0,\infty)\to \R_{>0}$
with $\lim_{t\to \infty}\psi(t)=0$, such that for any disjoint geodesic rays $r,r'$ representing the same point $\xi \in \partial_\infty X$, with
$d_X\big(r(0),r'(0)\big)\leq d$, the projection~$p(t)$ of $r(t)$ to $r'[0,\infty)$ satisfies
$$\Big|\angle_{r(t)}\big(r(0),p(t)\big) -\frac{\pi}{2}\Big|\leq \psi(t).$$
\end{lem}

\begin{proof}
Without loss of generality, we can assume $b_\xi(r(t))=b_\xi(r'(t))=-t$ for a Busemann function $b_\xi$. Since the horoball $b_\xi\leq -t$ is a limit
of balls, we have the inequality $\angle_{r(t)}\big(\xi,r'(t)\big)\leq \frac{\pi}{2}$.  Consequently,
$\alpha(t)=\angle_{r(t)}\big(r(0),r'(t)\big)\geq \frac{\pi}{2}$, and analogously $\beta(t)=\angle_{r'(t)}\big(r'(0),r(t)\big)\geq \frac{\pi}{2}$.

Let $x$ be the last point on the geodesic $r(0)r'(0)$ whose projection to the geodesic $r(t)r'(t)$ equals $r(t)$.
Then we have $\angle_{r(t)}\big(r'(t),x\big)=\frac{\pi}{2}$.
Considering the triangle $r(0)r(t)x$, by \cite{BH}*{II.1.7\parent{5}} we have
$$
d \ge d_X\big(r(0),r'(0)\big)\geq d_X\big(r(0),x\big)\geq t\sin \big(\alpha(t)-\tfrac{\pi}{2}\big).
$$
Consequently, there is a function $\psi$ with $\lim_{t\to \infty}\psi(t)=0$ satisfying for all $r, r'$ the bounds $\frac{\psi(t)}{2} \ge
\alpha(t)-\frac{\pi}{2}$, and analogously $\frac{\psi(t)}{2} \ge \beta(t)-\frac{\pi}{2}$. Since the sum of the Alexandrov angles in the triangle
$r(t)r'(t)p(t)$ is $\leq \pi$, we obtain $\angle_{r(t)}\big(r'(t),p(t)\big)\leq \frac{\psi(t)}{2}$. Consequently,
\[
\Big|\angle_{r(t)}\big(r(0),p(t)\big) -\frac{\pi}{2}\Big| \leq
\Big|\angle_{r(t)}\big(r(0),r'(t)\big) -\frac{\pi}{2}\Big|+ \Big|\angle_{r(t)}\big(r'(t),p(t)\big)\Big|\leq \psi(t).
\qedhere
\]
\end{proof}

\begin{cor}
\label{cor:2rays}
Let $r\colon [0,\infty)\to X$ be a geodesic ray in a complete $\mathrm{CAT}(0)$ space~$X$. Then for any $\eps>0, d>0$ and any
non-increasing function $\delta\colon [0,\infty)\to \R_{>0}$, there is $T>0$ such that
\begin{itemize}
\item for any geodesic ray $r'\colon [0,\infty)\to X$ asymptotic to $r$ satisfying
\begin{itemize}
\item $d_X\big(r(0),r'(0)\big)\leq d,$ and
\item $d_X\big(r(t),r'[0,\infty)\big)\geq \delta(t)$ for each $t\in [0,\infty),$ and
\end{itemize}
\item for any $t\in (T,\infty)$, for any $0<\eps'<\min\{\delta(t)^4,d,\eps\}$, and for any point $s\in X$ at distance $\leq \eps'$ from
    $r'(t+4d)$ such that the projections $s(t),s(T)$ of $r(t),r(T)$ onto $sr'(0)$ are distinct,
\end{itemize}
we have that the Alexandrov angles of the geodesic quadrilateral $s(t)s(T)r(T)r(t)$ are $\geq \frac{\pi}{2}$ at $s(t)$ and $s(T)$ and $\geq
\frac{\pi}{2}-9\sqrt{d}\sqrt[4]{\eps}$ at $r(t)$ and $r(T)$.
\end{cor}

In particular, by Theorem~\ref{thm:GB}, for any reduced relative disc diagram $\phi\colon D\to X$ with boundary path $s(t)s(T)r(T)r(t)$, the angles
in~$D$ at the preimages under $\phi$ of $s(t),s(T),r(T),r(t)$ are nearly $\frac{\pi}{2}$.

\begin{figure}
\includegraphics{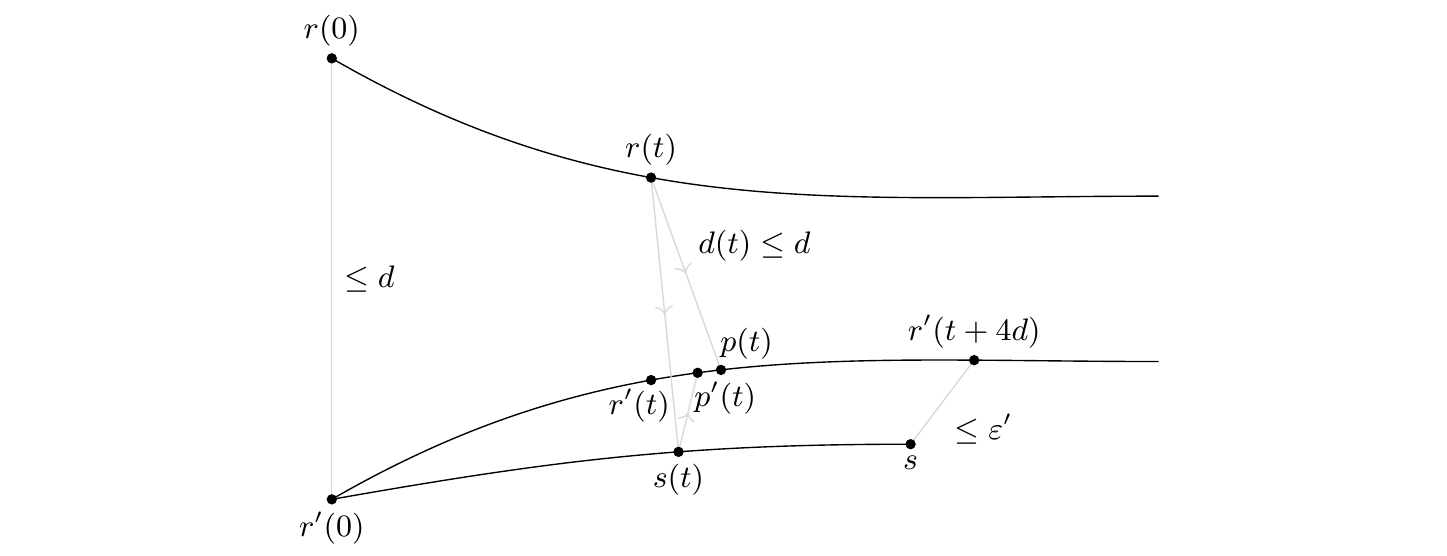}
\caption{Proof of Corollary \ref{cor:2rays}. Projections are denoted by arrows.}
\label{fig:de} 
\end{figure}

\begin{proof}
Choose $T$ so that the function $\psi$ from Lemma~\ref{lem:morerays} satisfies $\psi[T,\infty)<\sqrt{d}\sqrt[4]{\eps}$ (this choice does not matter
until the last line of the proof).
 
Let $p(t),p'(t)$ be the projections of $r(t),s(t)$ to $r'$, and let $d(t)=d_X\big(r(t),p(t)\big)$. See Figure~\ref{fig:de}. Since $p(t)$ is at
distance $\leq \eps'$ from a point on $r'(0)s$, and $s(t)$ is the projection of $r(t)$ onto $r'(0)s$, we have $d_X\big(r(t),s(t)\big)\leq d(t)+\eps'$.
In particular, we have $d_X\big(r(t), s(t)\big)< 2d$, and so $d_X\big(r'(t), s(t)\big)< 3d$, which implies that $s(t)$ is distinct from~$s$.
Analogously, $s(T)$ is distinct from~$s$. Similarly, taking $T>3d$ we have that $s(t)$ and~$s(T)$ are distinct from $r'(0)$. Then
$\angle_{s(T)}\big(s(t),r(T)\big)\geq \frac{\pi}{2}$ and $\angle_{s(t)}\big(s(T),r(t)\big)\geq \frac{\pi}{2}$.

Furthermore, we have $d_X\big(r(t),p'(t)\big)\leq d(t)+2\eps'$ and so, considering the triangle $r(t)p(t)p'(t)$, we obtain $d_X\big(p(t),p'(t)\big)\leq
\sqrt{4d(t)\eps'+4\eps'^2}$. Using the assumptions $d(t) \le d$ and $\eps' < d$, we get  $d_X\big(p(t),s(t)\big)\leq
\eps'+\sqrt{4d\eps'+4\eps'^2}< 4\sqrt {d\eps'}$. Analogously, the projection $p(T)$ of $r(T)$ onto $r'$ is at distance $< 4\sqrt {d\eps'}$ from
$s(T)$. Consequently, by \cite{BH}*{II.1.7(5)} we have that $\sin \angle_{r(t)}\big(p(t),s(t)\big)$ and $\sin \angle_{r(T)}\big(p(T),s(T)\big)$ are
$< 4\sqrt {d\eps'}/\delta(t)$. Since $\theta \leq 2\sin \theta$ for $\theta\in [0,\frac{\pi}{2}]$, and $\sqrt{\eps'}/\delta(t)<
\sqrt[4]{\eps'}$, these Alexandrov angles are $< 8 \sqrt{d}\sqrt[4]{\eps}$. By the definition of $\psi$, the differences between $\frac{\pi}{2}$
and $\angle_{r(t)}\big(p(t),r(T)\big)$ and $\angle_{r(T)}\big(p(T),r(t)\big)$ are $\leq\sqrt{d}\sqrt[4]{\eps}$ and the corollary follows.
\end{proof}

\section{Limit points that are not far}
\label{sec:limitnotfar}
 
In this section we describe particular pairs of limit points that are not far. 

\subsection{Angles between principal (and antiprincipal) points in \texorpdfstring{$\partial_\infty \X$}{dX}}

We first prepare the proof of Lemma~\ref{lem:principal_relation}.

\begin{lem}
\label{lem:anglesmatch} Let $\phi\colon D\to \nabla^+$ be an $\X$-reduced half-plane diagram with $\partial D=I_a\cdot I_b\cdot I_c$, where $I_a,I_c$
are geodesic rays, and $I_b$ is a possibly trivial geodesic segment. Suppose that the sum of the curvatures of vertices and edges in the interior of
$D$ and in the interior of $I_a,I_b, I_c$ is $\geq -\eps\geq -\eps_0$ from Corollary~\ref{cor:consise}. Let $\theta$ be the sum of the curvatures of
$I_a\cap I_b$ and $I_b\cap I_c$.

Then for each compact subcomplex $F\subset D$, there is a compact subcomplex $F\subseteq F'\subset D$, such that for any point $y_a\in I_a\setminus
F'$ and a unit vector $u_a$ tangent to~$I_a$ at~$y_a$, there exists a point $y_c\in I_c\setminus F$ such that the unit vector $u_c$ tangent to~$I_c$
at $y_c$, directed towards the same end of $\partial D$ as $u_a$, satisfies the following. Namely, there is an embedded path of length in
$[\theta-3\eps,\theta +3\eps]$ in the unit circle $S$, whose endpoints project to $\phi^\angle_{y_a}(u_a),\phi^\angle_{y_c}(u_c)\in S^\angle$.
\end{lem}
\begin{proof}
Choose $F'$ to be the minimal subcomplex of $D$ such that
\begin{itemize}
\item $D\setminus F'$ is simply connected, and
\item $F'$ contains $F$, and
\item $F'$ contains all the non-oriented components of the
folding locus $L$ of $D$ with 
\begin{itemize}
\item both ends on $\partial D$, and 
\item starting at $I_b$ or intersecting $F$.
\end{itemize}
\end{itemize}

Let $u_a$ be a unit vector tangent to~$I_a$ at a point $y_a$ outside $F'$. Then there is a piecewise geodesic path $\beta$ in $D\setminus D^0$ from
$y_a$ to a point $y_c\in I_c$ that intersects each non-oriented component $\LL$ of $L$ at most once. Furthermore, by our choice of $F'$, we can
assume that for each such $\LL$, there is a connected component of $\LL\setminus \beta$ that does not end on $\partial D$. Let $u_c$ be the unit
vector tangent to~$I_c$ at $y_c$, directed towards the same end of $\partial D$ as $u_a$.

We consider the non-oriented edges $f$ of $L$ intersected by $\beta$ at points $z$, and we denote by $\vec f$ a unit tangent vector to $f$ at~$z$
that is not directed towards $\partial D$. By Lemma~\ref{lem:curved_edges}, the total angle of all $\phi_z^\angle(\vec f)$ from the principal
direction is $\leq \eps$.

Let $D_\beta$ be the disc diagram that is the bounded connected component of $D\setminus \beta$. The sum of the curvatures outside $\beta$ in
$D_\beta$ lies in $[\theta-\eps, \theta]$. By Theorem \ref{thm:GB}, the sum of the curvatures along~$\beta$ \parent{including its endpoints} lies in
$[2\pi-\theta, 2\pi-\theta+\eps]$.
Thus, if we parallel transport $u_a$ along $\beta$, we obtain a vector $u$ in the obvious extension of~$S_{y_c}$
to the unit circle, at angle from $u_c$ in $[\theta-\eps,\theta]$. Furthermore, $\phi^\angle_{y_c}(u)$ and $\phi^\angle_{y_a}(u_a)$ differ by an
angle $\leq 2\eps$ since $\phi_z^\angle(\vec f)$ were at total angle $\leq \eps$ from $\partial S^\angle$.
\end{proof}

\begin{proof}[Proof of Lemma~\ref{lem:principal_relation}]
Suppose $\angle(\xi,\eta)\neq\frac{2\pi}{3}$. Choose $\eps\leq \eps_0$ from Corollary~\ref{cor:consise} such that the interval
$[\angle(\xi,\eta)-4\eps,\angle(\xi,\eta)+4\eps]$ does not contain $0$ and does not contain~$\frac{2\pi}{3}$. Choose $x$ so that $\angle_x(\xi,\eta)>
\angle(\xi,\eta)-\frac{\eps}{3}$. Let $\phi'\colon D'\to \nabla^+$ be a $(\xi,\eta)$-diagram for that $x$.

Suppose first that we are in case \eqref{con:diagram_sequenceii} of Construction~\ref{con:diagram_sequence}, where $D'$ is isometric to the Euclidean half-plane in~$\R^2$.
Consider a unit vector $u$ tangent to $\partial D'$.
Since $\xi,\eta$ are both principal, by Lemma~\ref{lem:boundarygood} the function $\phi'^\angle\colon \mathbb S\to S^\angle$ as in
Definition~\ref{def:u} simultaneously assigns to $u$ the
principal and the antiprincipal point, contradiction.

Now assume that we are in case \eqref{con:diagram_sequencei} of Construction~\ref{con:diagram_sequence}. Let $q'$ be the image of~$q$ in~$D'$. By
Lemma~\ref{lem:almostflat}, the sum of the curvatures of vertices and edges in the interior of $D'$ and in the interior of $\partial D' \setminus q'$
is $\geq -\eps$. Let $J_\xi,J_\eta\subset \partial D'$ be the rays~$J'$ from Lemma~\ref{lem:boundarygood} applied to $\xi$ and $\eta$. Then the
closure $F$ of $\partial D'\setminus (J_\xi\cup J_\eta)$ is compact. Choose $F'\supseteq F$ provided by Lemma~\ref{lem:anglesmatch}, and let $y_a\in
J_\xi\setminus F'$ with $\phi_{y_a}'^\angle(u_a)$ the principal point. Let $u_c\in S_{y_c}$ be the vector guaranteed by Lemma~\ref{lem:anglesmatch}.
By Lemma~\ref{lem:almostflat}, we have $\theta\in [\pi-\angle(\xi,\eta),\pi-\angle(\xi,\eta)+\eps]$, which is at distance $>3\eps$ from
$\frac{\pi}{3}$ and from $\pi$. Since $\phi_{y_c}'^\angle(u_c)$ is the antiprincipal point, there is no path in $S$ of length in $[\theta-
3\eps,\theta+3\eps]$ whose endpoints project to $\phi'^\angle_{y_a}(u_a),\phi'^\angle_{y_c}(u_c)$. This contradicts the conclusion of
Lemma~\ref{lem:anglesmatch}.

Consequently, we have $\angle (\xi,\eta)=\frac{2\pi}{3}$. Let $\zeta$ be the midpoint of the geodesic~$\xi\eta$ in $\partial_\infty  \X$. Let
$\X_\xi,\X_\eta$ be the connected components of the preimage of $\nabla^+\setminus \{\alpha_2=\alpha_3\}$ under $\rho_+$, containing the interior of some
(hence all by Lemma~\ref{lem:asymptoticparallel}) principal ray $r_\xi,r_\eta\colon [0,\infty)\to \X$ representing $\xi,\eta$, respectively. By
\cite{BH}*{II.9.10} (applied with $\frac{a'}{a}=2$), the geodesic ray $r_\zeta$ representing $\zeta$ starting at $r_\xi(1)$ satisfies
$$d_\X\big(r_\xi(2t),r_\zeta(t)\big)/t\nearrow \sqrt 3.$$ Consequently, $\rho_+(r_\zeta)$ is disjoint from $\{\alpha_2=\alpha_3\}$ and
so $r_\zeta\subset \X_\xi$. Furthermore, since $\sqrt 3<2,$ for any compact subset $K\subset \nabla^+$ we have that $\rho_+\big(r_\zeta(t)\big)$ is
disjoint from $K$ for $t$ sufficiently large. We have analogous properties for the geodesic ray $r'_\zeta$ representing $\zeta$ starting at
$r_\eta(1)$.

Since $\xi\neq \eta$, we have that $\X_\xi,\X_\eta$ are disjoint. Since $r_\zeta,r'_\zeta$ are asymptotic, we have that $\rho_+(r_\zeta)$ is
contained in the set $\{\alpha_2\leq m\alpha_3\}$, for some $m\geq 2$. We set~$K$ above to $\{\alpha_1=\alpha_2\leq m\alpha_3\}$, and we conclude
that a subray $r_\zeta[T,\infty)$ of $r_\zeta$ is contained in the preimage $\X'_\xi\subset \X_\xi$ under $\rho_+$ of $\{\alpha_3<\alpha_2\leq
m\alpha_3\}\setminus K$. Each connected component of $\X'_\xi$ is quasi-isometric to the connected component of the preimage under $\rho_+$ of
$\{\alpha_2=m\alpha_3\}\setminus K$, which is a tree with only one infinite path and that path is an antiprincipal ray. Consequently,
$r_\zeta[T,\infty)$ is antiprincipal, and it lies in the same chamber as the principal ray representing $\xi$ starting at $r_\zeta(T)$, as desired.

The last assertion of the lemma follows from interchanging in the proof the words `principal' and `antiprincipal', defining $K=\{\alpha_1\leq
m\alpha_2=m\alpha_3\}$ and replacing $\{\alpha_3<\alpha_2\leq m\alpha_3\}\setminus K$ by $\{\alpha_2<\alpha_1\leq m\alpha_2\}\setminus K$.
\end{proof}

The following application of Lemma~\ref{lem:principal_relation} produces distinct principal points in $\partial_\infty  \X$ at angle $\neq\frac{2\pi}{3}$ (and hence far).

\begin{lem}
\label{lem:2pi3} Let $g\in C$ be parabolic and let $\zeta$ be the class of the principal rays in $\Ap^+_\id$. Let $\eta\in \partial_\infty  \X$ be
principal with $\angle (\zeta,\eta)=\frac{2\pi}{3}$. Then $\angle \big(\eta,g(\eta)\big)\neq\frac{2\pi}{3}$ or $\angle
\big(\eta,g^2(\eta)\big)\neq\frac{2\pi}{3}$.
\end{lem}

\begin{proof}
Suppose by contradiction $\angle\big(\eta,g(\eta)\big)=\angle \big(\eta,g^2(\eta)\big)=\frac{2\pi}{3}$. Consider the unique geodesics
$\omega_1,\omega_2,\omega_3\subset
\partial_\infty  \X$ joining $\zeta,\eta,$ and $g(\eta)$. The midpoints $\xi_i$ of~$\omega_i$ are antiprincipal by Lemma~\ref{lem:principal_relation}.
If $\omega_1\cdot \omega_2$ is not a locally embedded path, then $\angle (\xi_1,\xi_2)<\frac{2\pi}{3}$, and so by Lemma~\ref{lem:principal_relation}
we have $\xi_1=\xi_2=\xi_3$. Consequently, $g$ fixes the entire geodesic $\zeta\xi_1$, which by Lemma~\ref{lem:principal_relation} constitutes the
boundary of a chamber. Thus by Lemma~\ref{lem:easystab}, we have that $g$ lies in a conjugate of $C\cap B'=H$. 
By Proposition~\ref{prop:stab},
$g$ is elliptic, which is a contradiction. 

Since $\partial_\infty \X$ is $\mathrm{CAT}(1)$ and is locally an $\R$-tree (Remark~\ref{rem:dim1}), we have that $\omega=\omega_1\cdot \omega_2\cdot \omega_3$ is isometrically
embedded and thus by Lemma~\ref{lem:circle} there is a flat $F\subset \X$ with $\partial_\infty  F=\omega$. Since $g^2$ is also not elliptic, we have
analogously a flat $F'$ with $\partial_\infty  F'$ containing $\zeta,\eta,g^2(\eta)$.

We will now show that $F\cap F'$ is nonempty. Choose $x\in F, x'\in F'$ and let $a=d_\X(x,x')$. Let $r_\zeta,r_\eta\colon [0,\infty)\to F,
r'_\zeta,r'_\eta\colon [0,\infty)\to  F'$ be the geodesic rays representing $\zeta,\eta$ starting at $x,x'$. Let $y$ be a point in the sector bounded by
$r_\zeta,r_\eta$ in $F$ and at distance $>a$ from $r_\zeta\cup r_\eta$. We can assume that $y$ lies in the interior of a $2$-cell. Let
$R>d_\X(x,y)+a$, and let $\beta\subset F$ be the arc of the circle centred at $x$ of radius $R$ from~$r_\zeta(R)$ to~$r_\eta(R)$. Let $\beta'\subset
F'$ be the arc of the circle centred at $x'$ of radius~$R$ from~$r'_\zeta(R)$ to~$r'_\eta(R)$. Then the corresponding points of the closed paths
$\alpha=r_\zeta[0,R]\cdot\beta\cdot r_\eta[R, 0]$ and $\alpha'=r'_\zeta[0,R]\cdot\beta' \cdot r'_\eta[R, 0]$ are at distance $\leq a$, since they are
points of the corresponding asymptotic geodesic rays starting at $x,x'$. Consequently, $\alpha'$ is homotopic to $\alpha$ in $\X-y$. By projecting to a small
metric ball around $y$ homeomorphic to a disc, we see that $\alpha$ is homotopically nontrivial in $\X-y$. Thus the same is true for $\alpha'$, and
so $y\in F'$. Consequently, $F\cap F'$ contains the geodesic ray representing $\zeta$ starting at $y$. Analogously, $F\cap g(F)$ and $F'\cap g(F)$ contain
geodesic rays representing $\zeta$.

We will now prove that there is a geodesic ray $r$ representing $\zeta$ in $F\cap F'\cap g(F)$. For contradiction, assume that all of $V=F\cap F',V'=F\cap
g(F),V''=g(F)\cap F'$ are disjoint. Let $b_\zeta$ be a Busemann function at $\zeta$. Note that $V,V'\subset F$ are nonempty, closed, convex
(Remark~\ref{rem:isometric}(i)) and invariant under positive translation in the direction~$\zeta$. Let $c,c'\colon \R\to F$ parametrise $\partial
V,\partial V',$ in a way that for $t$ sufficiently large we have $b_\zeta(c(t))=b_\zeta(c'(t))=-t$. There is $t$ such that for $T>t$ the curves
$\partial V,\partial V'$ are at $c(T),c'(T)$ at Alexandrov angle $>\frac{\pi}{6}$ from the geodesic $c(T)c'(T)$. We can arrange that the same
condition holds for $c'(T)c''(T)\subset gF$ joining $\partial V'$ with $\partial V''$ and for $c''(T)c(T)\subset F'$ joining $\partial V''$
with~$\partial V$. Then the geodesic triangle $c(T)c'(T)c''(T)$ has all Alexandrov angles $>\frac{\pi}{3}$, which is a contradiction. 

Consequently, there is a geodesic ray $r$ representing $\zeta$ in $F\cap F'\cap g(F)$. Note that $r$ lies in $\partial V,\partial V',$ and $\partial V''$.
Choose $y\in r$ and let $\gamma$ be the geodesic $yg(y)$. Note that $\gamma\cap V'=y$, since otherwise $g$ would fix the midpoint of $\gamma$ or
would fix $g(\eta)$. See Figure~\ref{fig:gamma}. Consequently, $g^{-1}(\gamma)\cdot\gamma$ is a local geodesic. Thus $\bigcup_{k\in \Z}g^k(\gamma)$
is an axis of $g$, and so $g$ is loxodromic, which is a contradiction.
\end{proof}

\begin{figure}
\includegraphics{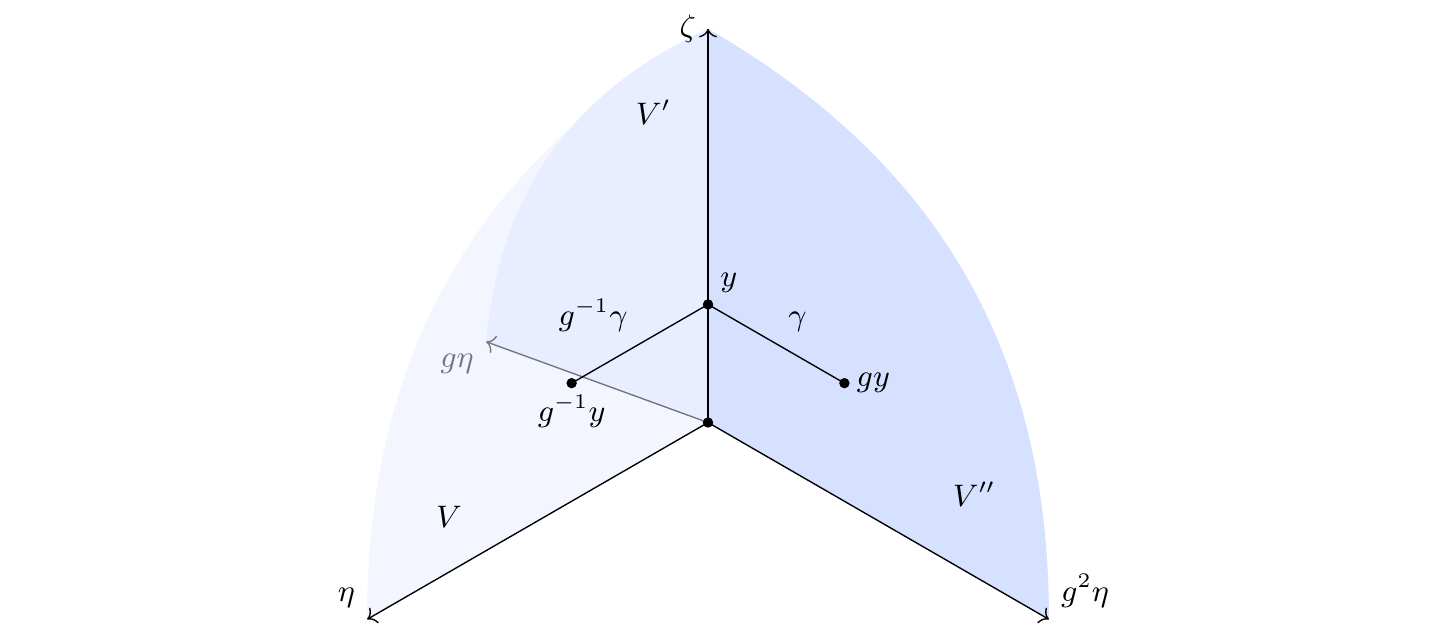}
\caption{}
\label{fig:gamma}
\end{figure}

\subsection{Reaching far}

\begin{prop}
\label{prop:far_limit} Let $g,g'\in \T$ be conjugate and parabolic of nonzero translation length or loxodromic not of rank~$1$. 
Suppose that for principal points $\zeta,\zeta'\in \partial_\infty \X$ fixed by $g,g',$ respectively (if these points exist), we have $\zeta\neq \zeta'$, and the same holds for $\zeta,\zeta'$ antiprincipal. 
\parent{Equivalently, $g,g'$ do not belong to a common conjugate of $C$ or $B'$}. 
Then any limit points $\xi$ of $g$
and $\eta$ of $g'$ are far.
\end{prop}

\begin{proof} To begin with, we justify $\xi\neq \eta$. Otherwise, if, say, $g\in B'$ and $\zeta\neq \zeta' \in \partial_\infty \X$ are the antiprincipal points
fixed by $g,g'$, then, by Construction~\ref{rem:B'rank2}\eqref{831}, we have length $\frac{\pi}{2}$ paths $\omega,\omega'
\subset\partial_\infty \X$ between $\xi=\eta$ and $\zeta,\zeta'$, respectively. By Lemma~\ref{lem:principal_relation}, we have that $\zeta$ is at
distance $\frac{2\pi}{3}$ from $\zeta'$. Since $\partial_\infty \X$ is $\mathrm{CAT}(1)$ and is locally an $\R$-tree (Remark~\ref{rem:dim1}), the
geodesic $\zeta\zeta'$ is contained in the union of $\omega$ and~$\omega'$, and so in particular the midpoint of $\zeta\zeta'$, which is principal by
Lemma~\ref{lem:principal_relation}, lies in $\omega$ and $\omega'$. Consequently, this midpoint is represented by a geodesic ray in~$\mathbf H$ from
Construction~\ref{rem:B'rank2}. However, such a ray does not satisfy Lemma~\ref{lem:asymptoticparallel}, which is a contradiction.

The remaining proof is divided into five parts.

\begin{partt}
Defining $D'^+$ and $D'^-$.
\end{partt}

Choose $x$ so that $\angle_x(\xi,\eta)> \angle(\xi,\eta)-\frac{\eps}{3}$ for $\eps<\eps_0$ from Corollary~\ref{cor:consise}. Let $\phi'^+\colon
D'^+\to \nabla^+$ be a $(\xi,\eta)$-diagram from Construction~\ref{con:diagram_sequence}.

Denote by $q'\in \partial D'^+$ the image of $q\in \partial D^+$, and by $J^+_\xi,J^+_\eta$ the closures of the appropriate connected components of
$\partial D'^+\setminus q'$. Recall that $D'^+$ is $\mathrm{CAT}(0)$, and $J^+_\xi,J^+_\eta$ are geodesic rays in $D'^+$ representing points in
$\partial_\infty D'^+$ at angle $\angle(\xi,\eta)$. 
Let $\beta^+_n=\widetilde \phi^+_{n|J^+_\xi\cap I_n}$.

Let $r$ representing $\xi$ be an $\mathbf{H}$-generic ray guaranteed by Remark~\ref{rem:generic_exist}(i). Let $r'_n$ representing $\xi$ be starting
at $p_n$ (in case~\eqref{con:diagram_sequenceii} of Construction~\ref{con:diagram_sequence}) or $x$ (in case~\eqref{con:diagram_sequencei}). Suppose first that:
\begin{equation}
\tag{$\heartsuit$} \label{heartsuit}
\text{for each $t$ there is $\delta(t)>0$ such that each $r'_n$ it at distance $\geq \delta(t)$ from~$r(t)$.}
\end{equation}
Let $d$ be the maximal distance from $r(0)$ to $x$ or a point of $P$. For $\eps$ above replaced by $\eps^4/18^4d^2$, let $T>0$ be provided by
Corollary~\ref{cor:2rays}, which applies to all $r'_n$. Then for each $t>T,$ for sufficiently small $\eps'>0$ there is $n$ with $s$ in $\beta^+_n$
satisfying the second bullet point of Corollary~\ref{cor:2rays}. By Lemma~\ref{lem:discexists}, we have reduced relative disc diagrams $\widetilde
\phi^-_n\colon D^-_n\to \X$ (embedded by Theorem \ref{thm:disc_embeds}) bounded by geodesic quadrilaterals $r(t)r(T)s(T)s(t)$ with total curvature
$\geq -\eps$ outside the preimages under $\widetilde \phi^-_n$ of $s(t),s(T),r(T),r(t)$. See Figure~\ref{fig:D}.

\begin{figure}
\includegraphics{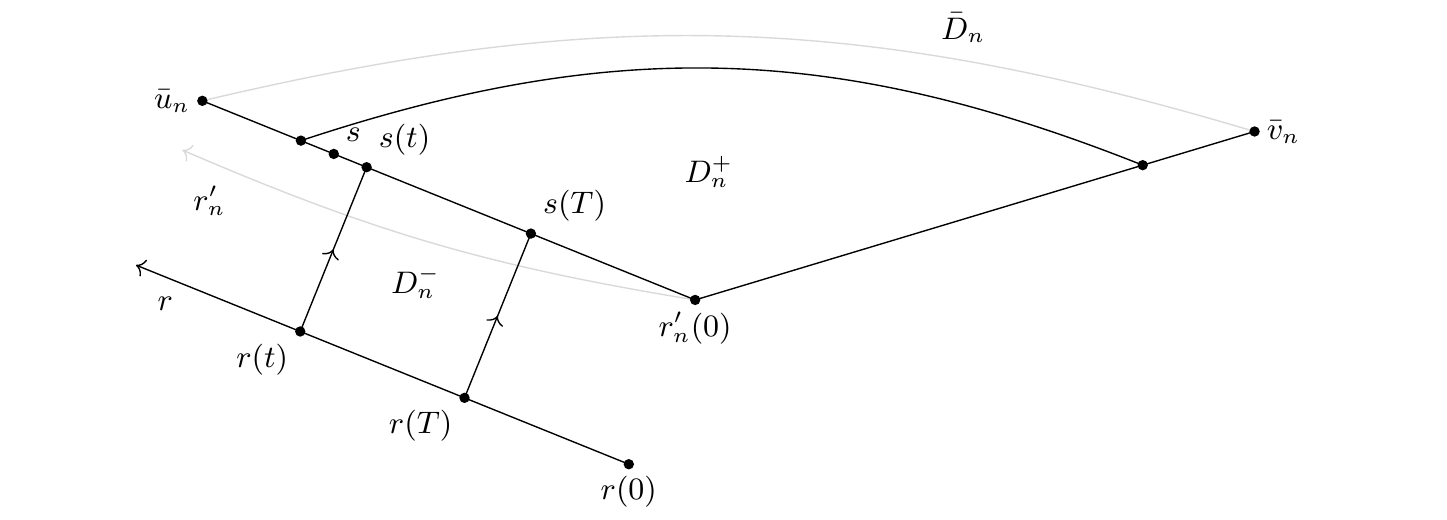}
\caption{The images of $D^\pm_n$.}
\label{fig:D}
\end{figure}

Let $I^-_n\subset \partial D^-_n$ be the union of the three sides not mapping under $\widetilde \phi^-_n$ to $r(t)s(t)$.
By passing to a subsequence and subdiagrams, 
we can assume that $(D^-_n,I^-_n)$ converge to a half-plane diagram $(D^-,\partial D^-)$ with $\partial D^-=J\cdot I_b\cdot I_c$, where the
intersections of $J,I_b,I_c$ with $\partial D^-_n$ map under $\widetilde \phi^-_n$ into $r(t)r(T), r(T)s(T), s(T)s(t)$, respectively. Since
$r(T)s(T)$ has uniformly bounded length, the first part of the second bullet point in Remark~\ref{rem:converge_nabla} implies the second part. Thus
$D^-$ has the frilling~$D'^-$ with the $\X$-reduced relative half-plane diagram $\phi'^-\colon D'^-\to \nabla^+$ that is the limit of
$\phi^-_n=\rho_+\circ \widetilde \phi^-_n$ bordered by $r_0=r_{|[T,\infty)}$, where we identify $[T,\infty)$ with $J \subset \partial D'^-$.

\begin{partt} Calculating the angle.
\end{partt}

By Corollary~\ref{cor:2rays} and Theorem~\ref{thm:GB}, the angles of $J\cap I_b,I_b\cap J_c$ are at distance $\leq \eps$ from $\frac{\pi}{2}$. We now
apply three times Lemma~\ref{lem:anglesmatch}, to $D'^-,D'^+$ (using Lemma~\ref{lem:almostflat}), and to the analogue $D'^-(\eta)$ of $D'^-$ with $\xi$ replaced by $\eta$. For $\eps$
sufficiently small, we obtain vectors $u,u(\eta)$ tangent to $J,J(\eta)$ at points $y,y(\eta)$, with $\phi'^{-\angle}_y(u),
\phi'^{-\angle}_{y}(u)(\eta)$ that are projections of the endpoints of an embedded path in $S$ of length arbitrarily close to $\pi-
\angle(\xi,\eta)$. Furthermore, we can arrange that $y,y(\eta)$ are in arbitrary subrays of $J,J(\eta)$. Since $r_0,r_0(\eta)$ are $\mathbf
H$-generic, this implies that $\phi'^{-\angle}_y(u), \phi'^{-\angle}_y(u)(\eta)$ are arbitrarily close to the midpoint of $S^\angle$. It follows that
$\angle(\xi,\eta)=\frac{\pi}{3},\frac{2\pi}{3},$ or $\pi$.

\begin{partt} Curvature in $Q\subset \widehat D'$.
\end{partt}

Let $\widehat \phi_n \colon \widehat D_n\to\X$ be the unions of $\widetilde\phi_n^\pm$ at $s(T)s(t)$, to which we apply Remark~\ref{rem:limit_union}
describing their limit $\widehat \phi' \colon \widehat D'\to\nabla^+$. In the case where $\overline{D'^-\setminus R^-}$ is not homeomorphic to a
half-plane, we use the more general Remark~\ref{rem:bizarreunion}. This covers the case of empty $\overline{D'^-\setminus R^-}$, where $\widehat
D'=\overline{D'^+\setminus R^+}$. After possibly increasing $T$, we can assume that $R^+$ is disjoint from $J^+_\eta$, and so $J^+_\eta\subset
\widehat D'$. Note that $\widehat \phi'$ is bordered by $r_0$. Observe that a geodesic ray in $D'^+\setminus R^+$ asymptotic to $J^+_\xi$ in $D'^+$
is also asymptotic to $J$ in~$\widehat D'$. In particular, the rays $J$ and $J^+_\eta$ represent points in $\partial_\infty \widehat D'$ at angle
$\angle(\xi,\eta)$.

At this point, let us note that if \eqref{heartsuit} was not satisfied, and so there is $T$ such that a subsequence of~$r'_n$ becomes arbitrarily
close to $r(T)$, then there are subrays of $r'_n$ converging to $r_0=r_{|[T,\infty)}$. This implies that $\beta^+_n$ have subpaths converging to
$r_0$ and so $\phi'^+$ is weakly bordered by~$r_0$ (see Definition~\ref{def:bordered}). After modifying~$\phi^+_n$ on the cells intersecting $J$, we
have that the same $\phi'^+$ is bordered by~$r_0$, and we set $\widehat\phi'=\phi'^+$ in the remaining discussion (where the following claim is
trivial).

We claim that the sum of the curvatures of the vertices and edges of $\widehat D'$ is finite. Indeed, if $\angle(\xi,\eta)<\pi$, then by
\cite{BH}*{II.9.8(4)}, the geodesics in $\widehat D'$ joining the points on $J$ and $J^+_\eta$ sufficiently far from a basepoint are disjoint from
any given compact subcomplex of~$\widehat D'$. Consequently, by Theorem~\ref{thm:GB}, there is a uniform bound on the total curvature in any compact
subcomplex of $\widehat D'$, as desired. For $\angle(\xi,\eta)=\pi$, one repeats this argument after subdividing $\widehat D'$ into two half-plane
diagrams along a ray representing at point in $\partial_\infty \widehat D'$ at distance $\frac{\pi}{2}$ from the points represented by $J$ and
$J^+_\eta$. This justifies the claim.

By the claim, there exists a compact subcomplex $F\subset \widehat D'$, such that the sum of the curvatures in $\widehat D'$ of the vertices and
edges outside $F$ is $\geq -\eps_0$. Let $l_k\in J$ be the sequence of transition points of $r_0$. There is $k_0$ such that the unbounded connected
component of $J\setminus l_{k_0}$ is disjoint from $F$. Applying Lemma~\ref{lem:generic} to a neighbourhood of that component in $\widehat D'$, we
obtain that for $k\geq k_0$ the points $l_k$ are contained in single edges of the folding locus $\widehat L$ and at angle to $J$ in
$\big(\frac{\pi}{2}-\eps_0,\frac{\pi}{2}+\eps_0\big)$ and converging to~$\frac{\pi}{2}$.

After possibly increasing $k_0$, a geodesic $\gamma$ in $\widehat D$, issuing from $l_{k_0}$ at angle $\frac{\pi}{2}+\eps_0$ to the unbounded
connected component of $J\setminus l_{k_0}$, separates $F$ from $l_{k_0+1}$. Cutting~$\widehat D'$ along~$\gamma$, we obtain a half-plane
diagram~$Q$ containing $l_{k_0+1}$, where the sum of all curvatures, except the curvature of $l_{k_0}$, and possibly except of the curvature of other endpoint of $\gamma$,
if $\gamma$ is bounded, is $\geq -\eps_0$.

\begin{partt} Folding locus.
\end{partt}

Let $L_k\subset \widehat L$ be the connected component of $\widehat L$ minus the degree $>2$ vertices of~$\widehat L$, containing~$l_k$. For each
$k\geq k_0$, we have that $L_k\cap Q$ are geodesics in $Q$, as are $J\cap Q$ and $\gamma$. Considering the angles that $L_k$ and $\gamma$ make with
$J$, we obtain that if $L_k\cap Q$ is bounded, then it does not end in $J\cup \gamma$. In particular, $L_k\cap Q=L_k$.

If $L_k$ is contained in a non-oriented component $\LL$ of $\widehat L$, then analogously $\LL\cap Q$ does not end in $J\cup \gamma$, and so $\LL\cap
Q=\LL$. Furthermore, since $J$ is not asymptotic to $J^+_\eta$, after possibly increasing $k_0$, we have that if $\LL$ has both ends in $\partial Q$,
then it has arbitrarily large length. Thus, by Lemma~\ref{lem:curved_edges}, for each $x\in L_k$ and $\vec{f}\in S_x$ the direction of an edge
of~$L_k$, directed away from $l_k$, we have that $\widehat\phi'^\angle_x(\vec{f})\in S^\angle$ is at angle $\leq\eps_0$ from a point in~$\partial
S^\angle$ (same point for all~$k$). The same would hold for other edges of the folding locus starting at $J$ sufficiently far from $\gamma$, which is excluded by the last assertion
of Lemma~\ref{lem:generic}. Consequently, after possibly increasing $k_0$, we have $\widehat L\cap J\cap Q=\{l_k\}_{k\geq k_0}$.

\begin{figure}
\includegraphics{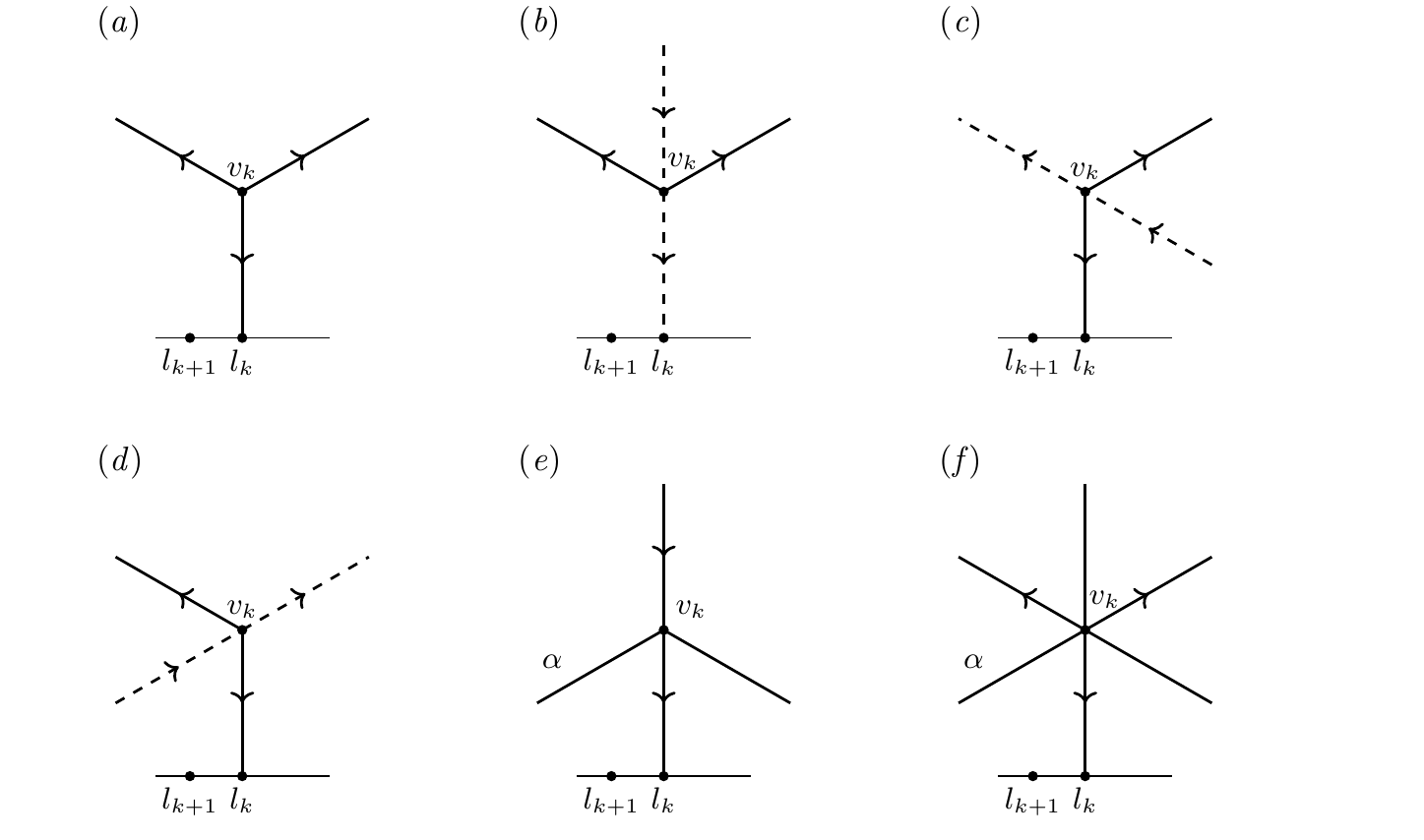}
\caption{}
\label{fig:branchinganti}
\end{figure}

\begin{figure}
\includegraphics{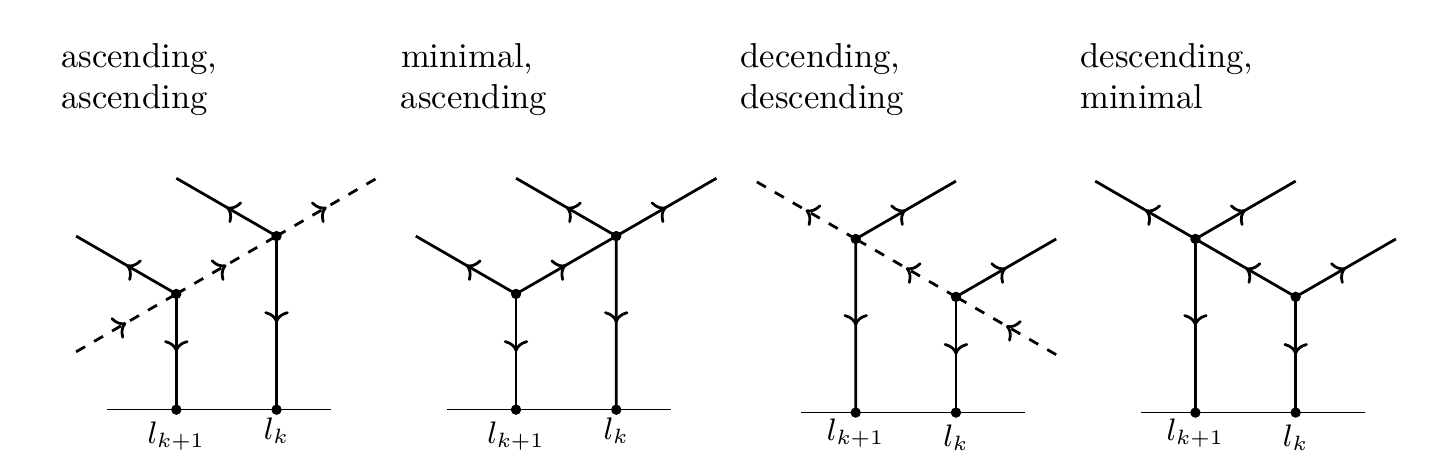}
\caption{}
\label{fig:doubles}
\end{figure}

Suppose first that all $\widehat\phi'^\angle_x(\vec{f})\in S^\angle$ above are at angle $\leq\eps_0$ from the antiprincipal point. If $L_k$ is
bounded and ends with an interior vertex $v_k$ of $Q$, then by Corollary~ \ref{cor:consise}, we have one of the configurations in Figure~\ref{fig:branchinganti}, where the dashed
edges might be oriented as indicated or non-oriented. 
However, configurations~(e,f) can be excluded in the following way. If the non-oriented component $\alpha$ intersects~$\partial Q$, then this contradicts the previous
paragraph. If $\alpha$ is disjoint from~$\partial Q$, then cutting~$\alpha$ into two and applying Lemma~\ref{lem:curved_edges}, we see that $\alpha$
contributes at least $\frac{2\pi}{3}$ to the curvature of $Q$, which is a contradiction. 
Analogously, the edges in configuration~(d) must be oriented.

We call $v_k$ as in configuration~(c) \emph{descending},
in~(d)~\emph{ascending}, and in~(a,b)~\emph{minimal}. 
Let $\alpha'$ be the connected component of $\widehat L$ minus the degree $>2$ vertices of~$\widehat L$, containing the north-west edge in (a,b,c) or the south-west edge in (d). If $\alpha'$ is non-oriented (which might happen only for $v_k$ descending), and ends at $v$ in Figure~\ref{fig:concise}(h), then we call $v_k$ \emph{wicked}.

Suppose that $v_k$ is not wicked. Since $J$ and $J^+_\eta$ represent points at $\partial_\infty \widehat D'$ at angle $\geq
\frac{\pi}{3}>\frac{\pi}{6}$, we have that $\alpha'$ does not end on $\partial Q$. Consequently,
$L_{k+1}$ also ends with an interior vertex $v_{k+1}$ of $Q$. Moreover, applying Corollary~ \ref{cor:consise} to the endpoint of $\alpha'$ distinct from $v_k$, and recalling that the edges in configuration (d) must be oriented, we conclude that this endpoint coincides with $v_{k+1}$, and
 $v_{k+1},v_k$ are either 
 \begin{itemize}
 \item
 ascending, ascending (see Figure~\ref{fig:doubles}), or 
 \item
 minimal, ascending (see Figure~\ref{fig:doubles} for the case that uses the minimal configuration from Figure~\ref{fig:branchinganti}(a)), or, symmetrically, 
\item 
descending, descending, or 
\item
descending, minimal.
\end{itemize}

Note that for $v_k$ wicked we could have other combinations. 
However, all $v_{k'}$ with $k>k'\geq k_0$ are then descending and not wicked. Consequently, after possibly replacing $k_0$ with $k+1$, we can assume that none of $v_k$ are wicked.

If the sequence $v_{k+2},v_{k+1},v_k$ was descending, minimal, ascending, we would obtain a contradiction with Remark~\ref{rem:straight}(ii) applied
to the north-east edge in Figure~\ref{fig:branchinganti}(c) and the north-west edge in Figure~\ref{fig:branchinganti}(d). Furthermore, we cannot have
an infinite sequence of ascending vertices. Consequently, after possibly increasing $k_0$, all $v_k$ are descending, see Figure~\ref{fig:descending},
left. The remaining possibility is that all $L_k$ are connected components of $\widehat L$, where $L_k$ might be unbounded (see Figure~\ref{fig:descending}, centre), or bounded (possibly
non-oriented) if they end on~$\partial Q$ (see Figure~\ref{fig:descending}, right).

\begin{figure}
\includegraphics{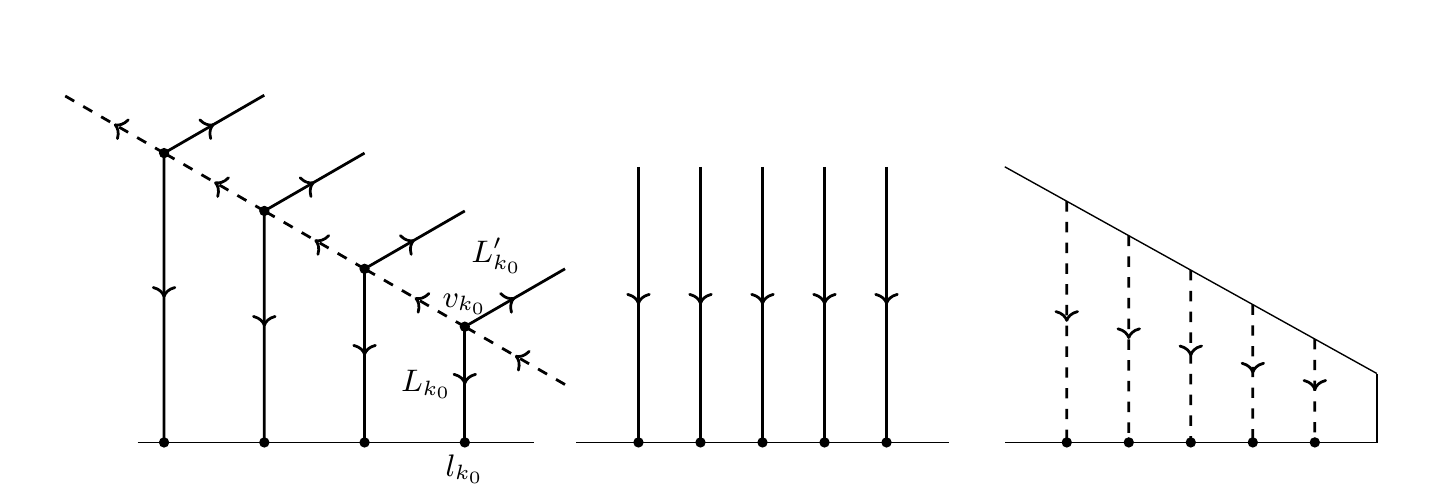}
\caption{}
\label{fig:descending}
\end{figure}

\begin{figure}
\includegraphics{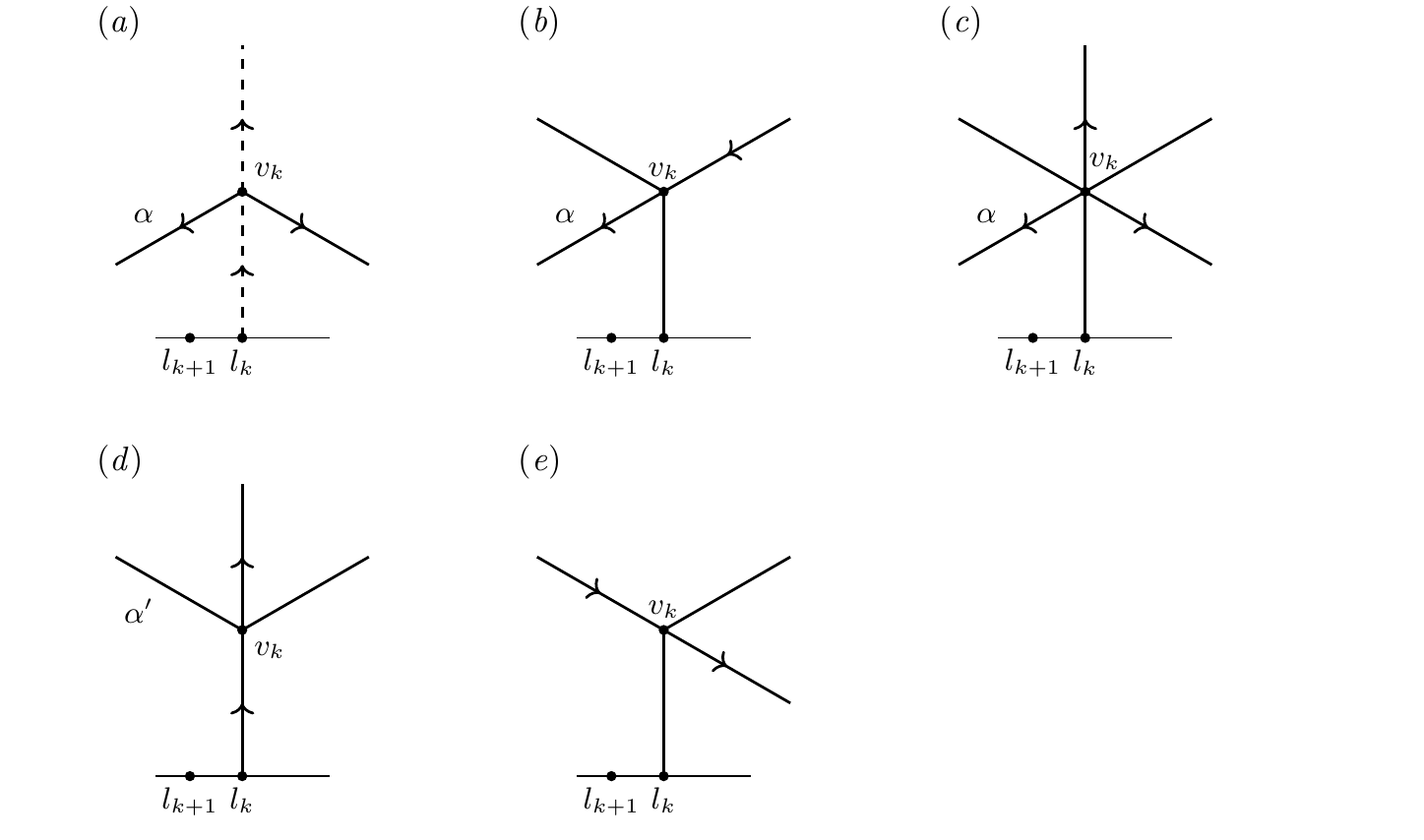}
\caption{}
\label{fig:branching}
\end{figure}

Second, suppose that all $\widehat \phi'^\angle_x(\vec{f})\in S^\angle$ are at angle $\leq\eps_0$ from the principal point. If $L_k$ is bounded and
ends with an interior vertex $v_k$ of $Q$, then we have one of the configurations in Figure~\ref{fig:branching}, where the dashed edges in (a) might
be oriented as indicated or non-oriented. In (a,b,c), the oriented geodesic $\alpha$ starting at $v_k$, given by Remark~\ref{rem:straight}(ii),
terminates at $J$ at angle contradicting $\widehat L\cap J\cap Q=\{l_k\}_{k\geq k_0}$. In~(d), the non-oriented component $\alpha'$ is disjoint from
$\partial Q$, since $J$ and $J^+_\eta$ represent points at $\partial_\infty \widehat D'$ at angle $>\frac{\pi}{6}$. Cutting $\alpha'$ into two and
applying Lemma~\ref{lem:curved_edges}, we see that $\alpha'$ contributes at least $\frac{2\pi}{3}$ to the curvature of $Q$, which is a contradiction.
It follows that all $v_k$ are as in (e), see Figure~\ref{fig:descendingbis}, left, or all $L_k$ are connected
components of~$\widehat L$, see Figure~\ref{fig:descendingbis}, centre and right (where some $L_k$ are possibly non-oriented).

\begin{figure}
\includegraphics{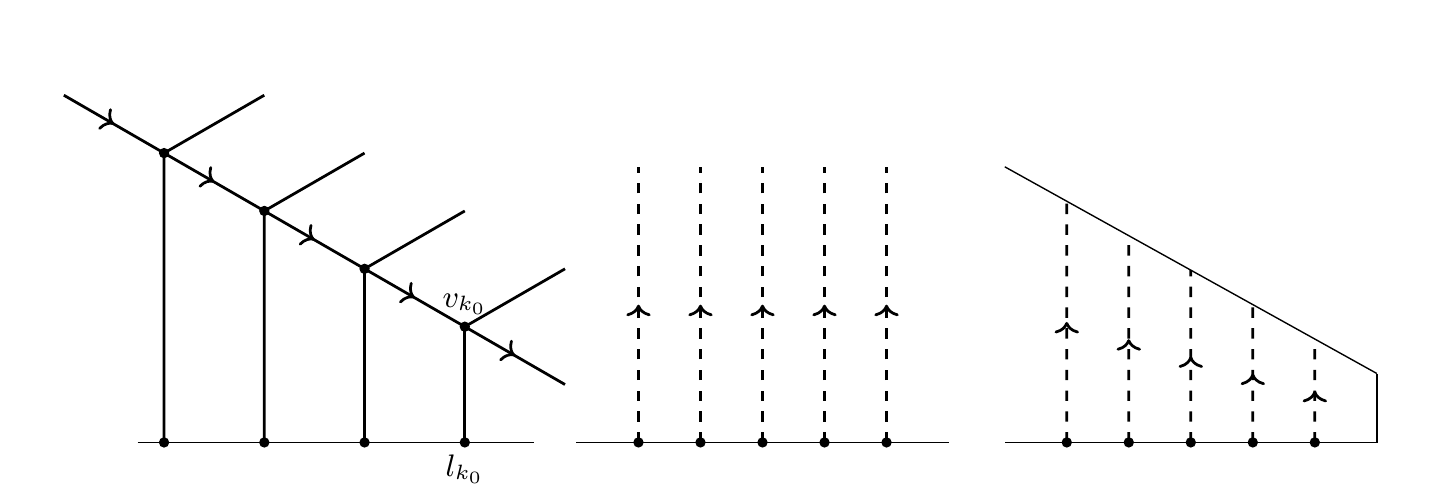}
\caption{}
\label{fig:descendingbis}
\end{figure}

\begin{partt} Interaction.
\end{partt}

To summarise all the possibilities for both $\xi$ and $\eta$, consider first the situation where there is $k_0$ such that for all $k\geq k_0$ the
$L_k$, denoted from now on by $L_k(\xi)$, are unbounded. Since they are at angle to $J$ converging to $\frac{\pi}{2}$ as $k\to \infty$, they are
geodesic rays representing points in $\partial_\infty \widehat D'$ at distance $\geq \frac{\pi}{2}$ from the point represented by $J$. Suppose that
after replacing $\xi$ by~$\eta$ we reach the same conclusion for $L_k(\eta)$. Then the subrays of $L_k(\xi)$ and $L_k(\eta)$ lying in $D'^+$ are
disjoint. Since they represent points in $\partial_\infty D'^+$ at distance $\geq \frac{\pi}{2}$ from $J^+_\xi,J^+_\eta$, they have to be asymptotic.
Proposition~\ref{prop:diagram->fixedpoint}(1), applied twice, implies
that $g,g'$ fix a common principal or antiprincipal point in~$\partial_\infty \X$, which is a contradiction.

Second, suppose that $L_k(\xi)$ are bounded but end on $\partial \widehat D'$. If $L_k(\eta)$ are unbounded, then subrays of $L_k(\eta)$ and
$L_k(\xi)$ lying in $D'^+$ intersect, which is a contradiction. If $L_k(\eta)$ are bounded connected components of the appropriate folding locus,
then, after renumbering, subrays of $L_k(\xi)$ and $L_k(\eta)$ lying in $D'^+$ coincide. This implies that $J^+_\xi$ and $J^+_\eta$ are asymptotic,
which is a contradiction.

Finally, suppose that we have the configuration from Figure~\ref{fig:descending}, left. Consider the north-east oriented geodesics $L'_k\subset
\widehat L$ starting at $v_k$. In particular, if $L_k(\eta)$ are unbounded, then subrays of $L_k(\eta)$ and $L'_k$ lying in $D'^+$ intersect, which
is a contradiction. Otherwise, we have that subsegments of $L_k(\eta)$ coincide with subsegments of~$L'_k$, since otherwise they would intersect in
$D'^+$. See Figure~\ref{fig:fir}, left. The analogous remaining possibility is described in Figure~\ref{fig:fir}, right.

\begin{figure}
\includegraphics{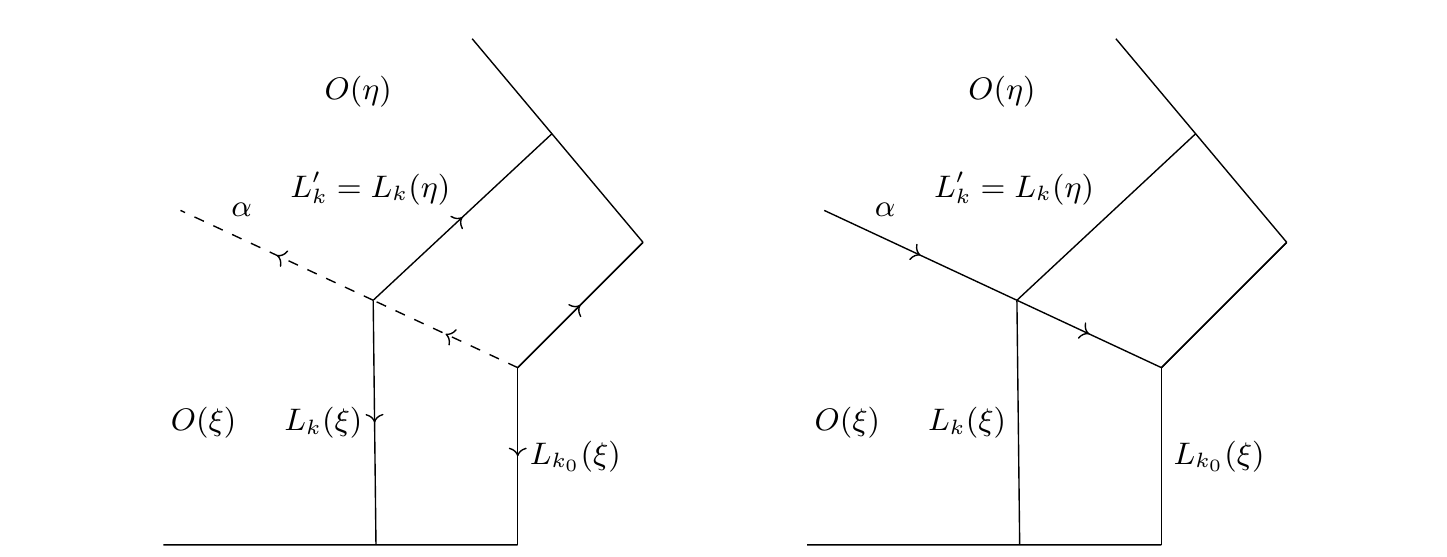}
\caption{}
\label{fig:fir}
\end{figure}

In both of these configurations, Let $O(\xi)$ be the half-plane diagram in $\widehat D'$ bounded by $L_{k_0}(\xi)$, and the ray $\alpha\subset
\widehat L$ containing all $v_k$ for $k\geq k_0$, and let $O(\eta)$ be defined analogously.
Applying Proposition~\ref{prop:diagram->fixedpoint}(2) to $O(\xi)$ and~$O(\eta)$, we find $k$ such that the unique antiprincipal or principal ray
starting at $v_k$ represents an antiprincipal or principal point in $\partial_\infty \X$ fixed by both $g$ and $g'$, which is a contradiction.
\end{proof}

\section{Proof of the Tits alternative}
\label{sec:proof}

We start with the following preparatory lemmas. In the applications of the first one, the set $\partial$ will be the visual boundary $\partial_\infty \X$, which explains the notation. 

\begin{lem}
\label{lem:set}
Let $G$ be a group acting on a set~$\partial$.
Let $Z$ be a finite subset of~$\partial$. Suppose that for each $u\in G$, the sets $Z$
and $u(Z)$ intersect. Then there is a nonempty finite subset of the union of $u(Z)$ over $u\in G$ that is invariant under $G$.
\end{lem}

Let $|Z|$ denote the number of elements of~$Z$. Before we start the proof, note that for $|Z|=2$, the sets $u(Z)$ have either an element in common or take on only three values $\{\xi,\eta\},\{\eta,\zeta\},\{\zeta,\xi\}$. In particular, the invariant nonempty finite set can be taken as $\bigcap_{u \in G} u(Z)$ or $\bigcup_{u \in G} u(Z)$. For $|Z|=3$, this is no longer true. Namely, consider the 2-dimensional simplicial complex $Y$ built of a triangle with vertices $\xi,\eta,\zeta$, and infinitely many triangles $T_i^\zeta$ (resp.\ $T_i^\eta,T_i^\xi$) sharing the edge $\xi\eta$ (resp.\ $\xi\zeta, \eta\zeta$). Let $G$ be the automorphism group of $Y$, let $\partial$ be the vertex set of $Y$ and let $Z$ be the vertex set of $T_1^\zeta$. Then $\bigcap_{u \in G} u(Z)=\emptyset, \bigcup_{u \in G} u(Z)=\partial$, and the required $G$-invariant finite set is $\{\xi,\eta,\zeta\}$.

We will be applying Lemma~\ref{lem:set} only with $|Z|\leq 4$.

\begin{proof}[Proof of Lemma~\ref{lem:set}]
We prove the lemma by induction on $|Z|$. If $|Z|=1$, then $Z$ is invariant under $G$. If $|Z|>1$, then let
$Z'\subsetneq Z$ be a maximal subset contained in infinitely many distinct translates $u(Z)$, where $u\in G$. If $Z'$ is empty, then the union of
all $u(Z)$ over $u\in G$ is finite and constitutes the desired finite set invariant under~$G$. If $Z'$ is nonempty, then we claim that for each $u\in
G$ the set $u(Z')$ intersects $Z'$, which allows to appeal to the induction and finishes the proof.

To justify the claim, observe that by the maximality of $Z'$, there is a sequence $(u_i)_{i \ge 1}$ in $G$ such that the sets $u_i(Z)\setminus Z'$ are pairwise disjoint. Consequently, the sets $uu_i(Z)\setminus u(Z')$ are pairwise disjoint. In particular, we can fix $j$ such that $uu_j(Z)\setminus u(Z')$ is disjoint
from~$Z'$. Since $uu_j(Z)$ intersects all $u_i(Z)$, and all the sets $u_i(Z)\setminus Z'$ are pairwise disjoint, we have that $uu_j(Z)$ intersects
$Z'$. By the choice of~$j$, we have that $u(Z')$ intersects~$Z'$, justifying the claim.
\end{proof}

\begin{lem}
\label{lem:marcin} Let $G$ be a group and let $F$ be a finite group. Suppose that for each finitely generated subgroup $H<G$ we have a nonempty
family $\mathcal F_H$ of homomorphisms from $H$ to $F$ satisfying the following. For all finitely generated subgroups $H<H'$ of $G$, and each $f\in
\mathcal F_{H'}$, we have $f_{|H}\in \mathcal F_{H}$. Then there is a homomorphism $f\colon G\to F$ such that for each finitely generated $H<G$ we
have $f_{|H}\in \mathcal F_{H}$.
\end{lem}
\begin{proof} The following proof was found by Marcin Sabok. Let $K=F^G$ be the compact space of all the functions from $G$ to $F$. For each finitely
generated subgroup $H<G$, let $K_H$ be the nonempty subset of $K$ consisting of all $f\in K$ satisfying $f_{|H}\in \mathcal F_{H}$. Since $F$ is
finite and $H$ is finitely generated, we have that $\mathcal F_{H}$ is finite, and so $K_H$ is closed. Note that the family $\{K_H\}$ has the finite
intersection property. Indeed, for finitely generated subgroups $H_1,\ldots, H_k$ of $G$, the group $H=\langle H_1,\ldots, H_k\rangle$ is also
finitely generated, and so any element of $K_H$ belongs to all $K_{H_i}$. Since $K$ is compact, there is $f\in K$ in the intersection of $\{K_H\}$,
as desired.
\end{proof}
\begin{cor}
\label{cor:marcin}
 Let $n,m>0$ and let $G$ be a group. Suppose that each finitely generated subgroup of $G$ contains a non-abelian free group or a
solvable subgroup of index $\leq n$ and step $\leq m$. Then $G$ satisfies the strong Tits alternative.
\end{cor}
\begin{proof} Suppose that $G$ does not contain a non-abelian free
group. Let $F$ be the symmetric group on $n$ elements. For each finitely generated subgroup $H<G$, let $\mathcal F_H$ be the nonempty family of
homomorphisms from $H$ to $F$ whose kernel is solvable of step $\leq m$. By Lemma~\ref{lem:marcin}, there is a homomorphism $f\colon G\to F$ such
that for each finitely generated $H<G$ we have $f_{|H}\in \mathcal F_{H}$. Let $G'$ be the kernel of $f$. Thus each finitely generated subgroup of
$G'$ is solvable of step $\leq m$. Consequently, $G'$ is solvable of step $\leq m$.
\end{proof}

\begin{proof}[Proof of the Main Theorem.] Let $G< \T$. 
We distinguish three cases according to the type of elements in $G$ with respect to the action of $\T$ on $\X$.
\smallskip

\noindent \textbf{Case 1.} All elements of $G$ are elliptic.

\smallskip

Consider any finitely generated subgroup $H$ of $G$.  By \cite{NOP}*{Cor~2.6}, the group $H$ has a global fixed point in $\X$. By
Remark~\ref{rem:BKTits}(i,ii), the group $H$ contains a non-abelian free group or a solvable subgroup of uniformly bounded step and index. Hence $G$
satisfies the strong Tits alternative by Corollary~\ref{cor:marcin}.

\smallskip

\noindent \textbf{Case 2.} $G$ contains a parabolic element $g$ or a loxodromic element $g$ that is not of rank $1$.

\smallskip
By Remark~\ref{rem:BKTits}(iii) and Lemma \ref{lem:easystab}, we can suppose that $G$ is not \emph{elementary} in the sense that it is not virtually contained in the stabiliser
of a principal or antiprincipal point in~$\partial_\infty \X$. In other words, there is no finite subset of the union of the principal and
antiprincipal points in~$\partial_\infty \X$ that is invariant under $G$.

Suppose first that $g$ has zero translation length. Let $\zeta$ be the unique principal point in $\partial_\infty \X$ fixed by~$g$ guaranteed by
Theorem \ref{thm:classify_isometries_parabolic}. Since $G$ is not elementary, there is $h\in G$ with $h(\zeta)\neq \zeta$. If $\angle\big(\zeta,h(\zeta)\big)\neq \frac{2\pi}{3}$, then $\zeta$ and $h(\zeta)$ are far by Lemma~\ref{lem:principal_relation}. We set $h'=h$ in that case. If $\angle\big(\zeta,h(\zeta)\big)=\frac{2\pi}{3}$, then 
by Lemma~\ref{lem:2pi3}, we have $\angle\big(h(\zeta),gh(\zeta)\big)\neq\frac{2\pi}{3}$ or $\angle\big(h(\zeta),g^2h(\zeta)\big)\neq \frac{2\pi}{3}$. Furthermore, $h(\zeta)\neq gh(\zeta), g^2h(\zeta)$ by the uniqueness in Theorem~\ref{thm:classify_isometries_parabolic} applied to $g$ and $g^2$. Hence for $h'=h^{-1}gh$ or $h'=h^{-1}g^2h$, we have that $\zeta$ and $h'(\zeta)$ are far by Lemma~\ref{lem:principal_relation}. 

Let
$g'=h'gh'^{-1}$. By Proposition~\ref{prop:Ruane} and Proposition~\ref{prop:parabolicdynamics}, for sufficiently large $n$ we have that
$f_n=(g')^ng^{-n}$ has far limit points, in arbitrary neighbourhoods of $\zeta, h'(\zeta)$. Since $G$ is not elementary, by Lemma~\ref{lem:set} there
is $u\in G$ such that $Z=\{\zeta,h'(\zeta)\}$ is disjoint from $u(Z)=\{u(\zeta),uh'(\zeta)\}$. In particular, for $n$ sufficiently large, $f_n$ and
$uf_nu^{-1}$ have distinct limit points. Thus by Corollary~\ref{cor:BF}, there are powers of
$f_n$ and $uf_nu^{-1}$ that generate a non-abelian free group.

Second, suppose that $g$ has nonzero translation length. By Theorems~\ref{thm:classify_isometries_parabolic}
and~\ref{thm:classify_isometries_hyperbolic}, there is a unique principal point $\zeta$ and/or a unique antiprincipal point~$\zeta'$ fixed by~$g$
in~$\partial_\infty \X$. Let $Z\subset \partial_\infty \X$ be $\{\zeta\},\{\zeta'\}$ (if only one of the two points exists), or $\{\zeta,\zeta'\}$.
Since $G$ is not elementary, by Lemma~\ref{lem:set} there is $h\in G$ with $Z\cap h(Z)=\emptyset$. By Proposition~\ref{prop:far_limit}, the limit
points $\xi,\eta$ of $g$ and $h(\xi),h(\eta)$ of $g'=hgh^{-1}$ are far. Again, by Proposition~\ref{prop:Ruane} and
Proposition~\ref{prop:parabolicdynamics}, for sufficiently large $n$ we have that $f_n=(g')^ng^{-n}$ has far limit points, in arbitrary
neighbourhoods of $\xi, h(\xi)$. Since $G$ is not elementary, applying Lemma~\ref{lem:set} to the set $Z\cup h(Z)$ in the place of $Z$, there is
$u\in G$ such that $\zeta,\zeta',h(\zeta),h(\zeta')$ are distinct from $u(\zeta),u(\zeta'),uh(\zeta),uh(\zeta')$. By
Proposition~\ref{prop:far_limit}, the points $\xi,h(\xi), u(\xi),uh(\xi)$ are distinct. Thus for $n$ sufficiently large, $f_n$ and $uf_nu^{-1}$ have distinct limit points, and we conclude as before.
\smallskip

\noindent \textbf{Case 3.} $G$ contains a loxodromic element $g$, and all its elements are loxodromic of rank $1$ or elliptic.

\smallskip

If a finite index subgroup $G'$ of $G$ fixes one of the two limit points $\xi\in \partial_\infty \X$ of $g$, then let $b_\xi$ be a Busemann function
for $\xi$. The group $G'$ acts on the level sets of $b_\xi$ giving rise to a homomorphism $G'\to \R$, whose kernel $H$ preserves the level
sets. By Case~1, we can assume that there is $h\in H$ that is not elliptic. Thus $h$ is loxodromic. Let $\sigma\colon \R\to \X$ be an axis of $h$.
Since $h\in H$, the axis $\sigma$ is contained in a level set, say $b_\xi=0$. For each $s\in \R, t\geq 0$, let $\sigma(s,t)$ be the projection of
$\sigma(s)$ onto the horoball $b_\xi\leq -t$. Thus for each $s,s'\in \R, t\geq 0$, we have $d_\X\big(\sigma(s,t),\sigma(s',t)\big)\leq
d_\X\big(\sigma(s),\sigma(s')\big)$. On the other hand, since $h$ acts on the curve $\sigma(\cdot, t)$ as a translation, and since $\sigma$ was an
axis of $h$, we have a converse inequality, and so $d_\X\big(\sigma(s,t),\sigma(s',t)\big)= d_\X\big(\sigma(s),\sigma(s')\big)$. Consequently,
$\sigma$ bounds an isometrically embedded Euclidean half-plane $\sigma(\cdot,\cdot)$, and so $h$ is not of rank~$1$, which is a contradiction.

Thus, by Lemma~\ref{lem:set}, there is $u\in G$ such that $g$ and $ugu^{-1}$ have distinct limit points, and we conclude as before.
\end{proof}

\end{document}